\documentclass{amsart}

\usepackage[all,hyperref,numberbysection]{bi-discrete}
\usepackage{amsmath}
\usepackage{upgreek}
\usepackage{microtype}
\numberwithin{equation}{section}
\usepackage[backend=biber,maxbibnames=5,maxalphanames=5,style=alphabetic,sorting=nty,sortcites=true,bibencoding=utf8,giveninits,url=false,isbn=false]{biblatex}
\AtBeginBibliography{\small}

\usepackage{a4wide}

\begin{document}

\author{Sun-Sig Byun}
\author{Kyeong Bae Kim}
\author{Deepak Kumar}
\address{Sun-Sig Byun: Department of Mathematical Sciences and Research Institute of Mathematics, Seoul National University, Seoul 08826, Korea}
\email{byun@snu.ac.kr}

\address{Kyeongbae Kim: Department of Mathematical Sciences, Seoul National University, Seoul 08826, Korea}
\email{kkba6611@snu.ac.kr}

\address{Deepak Kumar: Research Institute of Mathematics, Seoul National University, Seoul 08826, Korea}
\email{deepak.kr0894@gmail.com}

\makeatletter
\@namedef{subjclassname@2020}{\textup{2020} Mathematics Subject Classification}
\makeatother


\thanks{S. Byun was supported by NRF-2022R1A2C1009312, K. Kim was supported by IRTG 2235/NRF-2016K2A9A2A13003815 and
D. Kumar was supported by NRF-2021R1A4A1027378.}

\title{Global Calder\'on-Zygmund theory for fractional Laplacian type equations}

\begin{abstract}
 We establish several fine boundary regularity results of weak solutions to non-homogeneous $s$-fractional Laplacian type equations. In particular, we prove sharp Calder\'on-Zygmund type estimates of $u/d^s$ depending on the regularity assumptions on the associated kernel coefficient including VMO, Dini continuity or the H\"older continuity, where $u$ is a weak solution to such a nonlocal problem and $d$ is the distance to the boundary function of a given domain. Our analysis is based on point-wise behaviors of maximal functions of $u/d^s$.
\end{abstract}

\maketitle



\section{Introduction}
In this paper, we investigate the weak solution to the following nonlocal problem:
\begin{equation}
\label{eq: main}
\left\{
\begin{alignedat}{3}
\mathcal{L}_Au&= f&&\qquad \mbox{in  $\Omega$}, \\
u&=0&&\qquad  \mbox{in $\bbR^n\setminus \Omega$},
\end{alignedat} \right.
\end{equation}
where $\mathcal{L}_A$ is the nonlocal operator defined as
\begin{equation}\label{eq.nonl.op}
    \mathcal{L}_Au(x)=\mathrm{P.V.}\int_{\bbR^n}\frac{u(x)-u(y)}{|x-y|^{n+2s}}A(x,y) \,dy \quad x\in \bbR^n,
\end{equation}
where $n\geq 2$, $s\in(0,1)$ and the associated kernel coefficient $A=A(x,y):\bbR^n\times \bbR^n\to \bbR$ is measurable and satisfies 
\begin{equation}\label{cond.kernel}
    A(x,y)=A(y,x)\quad\text{and}\quad\Lambda^{-1}\leq A(x,y)\leq \Lambda\quad\mbox{for any $x,y\in\bbR^n$}
\end{equation}
for some constant $\Lambda\geq1$, and $\Omega$ is $C^{1,\alpha}$-domain for some $\alpha\in(0,s)$. 

In the literature, operators as in \eqref{eq.nonl.op} are known as the fractional Laplacian type nonlocal operator and they appear in ample amount of physical models, see for instance \cite{FerRos24,hitch} and references therein.  Particularly, when the coefficient $A$ is constant, \eqref{eq.nonl.op} corresponds to the well-known standard fractional Laplacian. Heuristically,  under a suitable structure assumption on the coefficient $A$, operator $\mathcal{L}u$ can be understood as a nonlocal analog of the divergence form operator of the kind $\mathrm{div}(M(x)Du)$, where $M$ satisfies the uniform ellipticity conditions (see for instance, \cite[Chapter 2]{FerRos24}). 

Taking into account the aforementioned observations, it is customary to mention fine regularity results for local problems with different kinds of coefficients. To this end, let us consider the following problem:
\begin{align}\label{eq:local-prot}
  \mathrm{div}(M(x)Du)=f,  
\end{align}
where the coefficient function $M$ is uniformly elliptic and bounded, and $f$ is a measurable function.
It is by now classical that weak solutions of \eqref{eq:local-prot} are H\"older continuous and their gradients have a slightly better integrability than the natural energy space {under some suitable integrability assumptions on $f$} but without imposing any regularity assumption on the coefficient function $M$. Some available counter examples assert that this is the best regularity one can expect in this case. Subsequently, it was observed that under small BMO assumption on $M$, the gradient of the solution exhibits higher integrability properties in the spirit of the classical Calder\'on-Zygmund theory known for the Laplacian (cf. \cite{CP,ByuWan04}). 
It is also known that even the continuity of the coefficient $M$ is not sufficient to conclude the boundedness of the gradient (see for instance, \cite{Jin-Maz-Saff}). Nevertheless, the classical theory suggests that the Dini continuity is sufficient to get such a bound on the gradient when $f$ is in the Lorentz space $L^{n,1}$. The latter condition is also optimal in this regard (see for instance, \cite{cianchi92,DuzMin10cvpde}). We also refer to \cite{CianMaz11} for global Lipschitz regularity results for the solution of the Dirichlet problem with the homogeneous boundary condition under a suitable regularity assumption on the domain.
Furthermore, concerning the H\"older continuity of the gradient of a solution $u$, we have that when $M\in C^{\alpha}$ {and $f\in L^{\frac{n}{1-\alpha},\infty}$}, then $Du\in C^{\alpha}$ (cf. \cite{KuuMin12}). 

Concerning nonlocal problems, interior regularity for linear problems is now rather well understood. In particular, for H\"older's regularity, Harnack's inequality, self-improving properties and zero-order potential estimates, without being exhaustive, we refer to \cite{Kas09,KaWe23,KuuMinSir15s,KuuMinSir15,KimLeeLee23,BehDieOkRol24}.  For higher interior regularity results obtained under additional regularity assumptions on the kernel coefficient,  we refer to \cite{Now21h,Now23,Now23i,BraLin17,BraLinSch18,MenSchYee21,Coz17,BogDuzLiaBisSer24,BogDuzLiBisSer24d,DieKimLeeNow24,CafSil11,Fal20,FerRos24s,KuuSimYan22} and references therein.

In contrast to local problems, the boundary regularity for nonlocal problems witnesses quite a peculiar behavior. Particularly, the order of interior H\"older regularity of solutions does not match that of up to the boundary irrespective of the smoothness assumption on the boundary and data of the problem (e.g., the coefficient and the right-hand side). For instance, it has been observed that the function $u(x):=(1-|x|^2)^s_{+}$ solves 
\begin{align*}
    (-\Delta)^s u = {\mathrm{const.}} \quad \text{in }B_1^+; \quad u=0 \mbox{ \, \, in }\bbR^n\setminus B_1^+,
\end{align*}
which fails to be more regular than $C^s$ up to the set $\{x\in B_1\,:\,x_n=0\}$. More precisely, in a recent article \cite{RosWei24}, a global $C^s$ regularity is established for the weak solution to \eqref{eq: main} in $C^{1,\alpha}$ domains when the kernel coefficient is merely translation invariant. Prior to this work, similar results were obtained only for kernels satisfying both the translation invariant condition and a homogeneity condition. Recently, in \cite{Gru24}, the regularity assumption on the boundary is relaxed to that of $C^{1,\mathrm{Dini}}$ to obtain such an optimal regularity up to the boundary when the right-hand side datum is bounded. In this direction, some related global regularity results concerning the H\"older continuity for solutions to nonlocal problems can be found in  \cite{ FerRos24,KorKuuPal16,IanMosSqu16,Gks-acv} and references therein.  We refer to \cite{AbdFerLeoYou23, Borth-jfa,Peral-dcds} and their references for global Calder\'on-Zygmund type theory involving the solution $u$. In a recent preprint \cite{KimWei24}, the non-translation invariant kernel case is dealt with. In particular, it is proved that when the domain is of $C^{1,\alpha}$, the kernel coefficient $A$ is in $C^{\alpha}$ in the sense of \eqref{ass2} below and $f\in L^q$, where $q>n/s$, then weak solutions are in the space $C^{s}$ up to the boundary of the domain.

In recent times, it has become quite an intriguing task to look for some higher-order boundary regularity theory involving the solutions that could improve with an improvement on the regularity of the data of the problem; that is, not as restrictive as the H\"older continuity or Sobolev regularity of solutions up to the boundary. In this direction, the regularity theory for the quotient $u/d^s$, where $d(x)\coloneqq \mathrm{dist}(x,\partial\Omega)$, has attracted much attention from researchers since the seminal works of Ros-Oton-Jerra \cite{RosSer17,Ros-duke,Ros-jde,Ros-jmpa} and Grubb \cite{Gru15,Gru18}. Heuristically, when $s\to 1^-$, the quotient $u/d^s$ can be viewed as the normal derivative of $u$ at the boundary points. In the nonlocal setting, it has been observed that the function $u/d^s$ enjoys much higher regularity up to the boundary compared to that of only $u$, if the data of the problem is smooth enough. Specifically, in \cite{Ros-jmpa}, it is established that when $f\in L^\infty$ and the domain is of $C^{1,1}$, then $u/d^s\in C^\beta$, for some small $\beta>0$. Subsequently, in \cite{Gru15}, it is proved that when the kernel coefficient is translation invariant, smooth and homogeneous, the right-hand side is also smooth and the domain is $C^\infty$, then $u/d^s$ is also smooth up to the boundary. In fact, in the same paper, regularity theory for linear pseudo-differential operators is established. Some higher H\"older continuity results of $u/d^s$ are also obtained for translation invariant and homogeneous kernels in the works \cite{Fal19,Abatg-ros}. We refer to the recent monograph \cite{FerRos24} for a detailed discussion on similar results for problems with different kinds of homogeneous or translation invariant kernel coefficients. In \cite{IanMos, IanMosSqu20}, the H\"older continuity of $u/d^s$ for nonlinear fractional p-Laplacian equations is proved. On the other hand, we refer to \cite{ChoKimRyu23,DonRyu24,Gru15,Gru18} for various kinds of global Calder\'on-Zygmund type estimates of the function $u/d^s$ for nonlocal problems with different types of kernel but not including non-translation invariant ones. Recently, in \cite{KimWei24}, for non-translation invariant and H\"older continuous kernel in the sense of \eqref{ass2}, H\"older regularity of $u/d^s$ is proved provided that the frozen kernel as in \eqref{defn.az} is homogeneous.

To the best of our knowledge, there is no regularity result available for the quotient $u/d^s$ in the spirit of Calder\'on-Zygmund for the solution of \eqref{eq: main} with possibly discontinuous coefficient and non-translational invariant $A$. Moreover, in the limiting case, to get the boundedness of $u/d^s$, the available works in the literature seem to impose the condition of H\"older continuity (in some form) or the translation invariance property on the coefficient function.

Our aim is to find suitable regularity assumptions on the associated kernel coefficient $A$ to derive sharp global Calder\'on-Zygmund type estimates for $u/d^s$ with respect to the datum $f$. To be precise, we will prove the following:
\begin{enumerate}
    \item If $A$ is VMO in $B_R(x_0)\times B_R(x_0)$, then
    \begin{equation*}
    f\in L^q(\Omega\cap B_R(x_0))\Longrightarrow u/d^s\in L^{\frac{nq}{n-qs}}(\Omega\cap B_{R/2}(x_0))
\end{equation*}
for any $q\in[2_*,n/s)$, where $2_*\coloneqq \frac{2n}{n+2s}$. 
\item If $A$ is Dini continuous in $B_R(x_0)\times B_R(x_0)$, then
\begin{align*}
    f\in L^{n/s,1}(\Omega\cap B_R(x_0))\Longrightarrow u/d^s\in{L^\infty(\Omega\cap B_{R/2}(x_0))}\quad\text{and}\quad u\in C^s(B_{R/2}(x_0)).
\end{align*}
\item If $A$ is {$\alpha$-}H\"older continuous in $B_R(x_0)\times B_R(x_0)$, then
\begin{equation*}
    f\in L^q(\Omega\cap B_R(x_0))\Longrightarrow u/d^s\in W^{\sigma,\frac{nq}{n-(s-\sigma)q}}(\Omega\cap B_{R/2}(x_0))
\end{equation*}
and
\begin{equation*}
    f\in L^{\frac{n}{s-\sigma},\infty}(\Omega\cap B_R(x_0))\Longrightarrow u/d^s\in C^{\sigma}(\Omega\cap B_{R/2}(x_0))
\end{equation*}
for any $q\in [2_*,\infty)$ and $\sigma\in(0,\min\{\alpha,s\})$. 
\end{enumerate}

Before presenting a localized version of problem \eqref{eq: main} and prescribing a notion of a weak solution to it, we first recall relevant function spaces.

Let us fix a domain $Q\subset\bbR^n$.
We first recall the fractional Sobolev space $W^{\sigma,p}(Q)$ with  $\sigma\in(0,1)$ and $1\leq p<\infty$, which is defined as
\begin{equation*}
    W^{\sigma,p}(Q)\coloneqq\{g:Q\to\bbR\,:\,\|g\|_{W^{\sigma,p}(Q)}<\infty\},
\end{equation*}
where 
\begin{equation*}
    \|g\|_{W^{\sigma,p}(Q)}\coloneqq \left(\int_{Q}\int_{Q}\frac{|g(x)-g(y)|^p}{|x-y|^{n+\sigma p}}\,dx\,dy\right)^{\frac1p}+\left(\int_Q|g|^p\,dx\right)^{\frac1p}.
\end{equation*}
We also recall the following tail space 
\begin{equation*}
   L^1_{2s}(\bbR^n)\coloneqq\left\{g:\bbR^n\to\bbR\,:\,\|g\|_{L^1_{2s}(\bbR^n)}\coloneqq\int_{\bbR^n}\frac{|g(y)|}{(1+|y|)^{n+2s}}\,dy<\infty\right\},
\end{equation*}
and for any $B_R(x_0)\subset\bbR^n$, we write 
\begin{equation*}
    \mathrm{Tail}(g;B_R(x_0))\coloneqq R^{2s}\int_{\bbR^n\setminus B_R(x_0)}\frac{|g(y)|}{|y-x_0|^{n+2s}}\,dy,
\end{equation*}
which was first introduced in \cite{DicKuuPal}.
For any measurable function $h:\bbR^n\to\bbR$, we define
\begin{equation*}
    X_{h}^{s,2}(Q,Q')\coloneqq\{g\in W^{s,2}(Q')\,:g\equiv h\quad\text{in }\bbR^n\setminus Q\},
\end{equation*}
where $Q'$ is a domain satisfying $Q\Subset Q'$.
Since we note from \cite[Lemma 2.11]{BraLinSch18} that $X_0^{s,2}(Q,Q')\subset X_{0}^{s,2}(Q,\bbR^n)$ for any $Q'\Supset Q$, we can write 
\begin{equation*}
    X_0^{s,2}(Q)\coloneqq X_{0}^{s,2}(Q,\bbR^n).
\end{equation*}

We now introduce a localized problem of \eqref{eq: main} and define its weak solutions.
\begin{definition}
    Let us fix $f\in L^{2_*}(\Omega\cap B_R(x_0))$. We say that $u\in W^{s,2}(B_R(x_0))\cap L^1_{2s}(\bbR^n)$ is a weak solution to 
    \begin{equation}
\label{eq: defn}
\left\{
\begin{alignedat}{3}
\mathcal{L}_Au&= f&&\qquad \mbox{in  $\Omega \cap B_R(x_0)$}, \\
u&=0&&\qquad  \mbox{in $B_R(x_0)\setminus \Omega$},
\end{alignedat} \right.
\end{equation}
if $u\in W^{s,2}(B)$ for some $B\Supset \Omega\cap B_R(x_0)$ and for any $\psi\in X^{s,2}_0(\Omega\cap B_R(x_0))$, there holds
\begin{align*}
    \int_{\bbR^n}\int_{\bbR^n}\frac{(u(x)-u(y))(\psi(x)-\psi(y))}{|x-y|^{n+2s}}A(x,y)\,dx\,dy=\int_{\Omega\cap B_R(x_0)}f\psi\,dx.
\end{align*}
\end{definition}
\begin{remark}
    Observe from \cite[Proposition 2.12]{BraLinSch18} that there is a unique weak solution $u\in X^{s,2}_{0}(\Omega)$ to \eqref{eq: main} when $f\in L^{2_*}(\Omega)$. Therefore, the solution $u$ to \eqref{eq: main} is also a weak solution to \eqref{eq: defn}.
\end{remark}

Since the problem \eqref{eq: defn} is scaling-invariant, it suffices {to confine our} analysis to a solution $u\in W^{s,2}(B_1)\cap L^1_{2s}(\bbR^n)$ to
\begin{equation}
\label{eq: thm}
\left\{
\begin{alignedat}{3}
\mathcal{L}_Au&= f&&\qquad \mbox{in  $\Omega\cap B_1$}, \\
u&=0&&\qquad  \mbox{in $B_1\setminus \Omega$}.
\end{alignedat} \right.
\end{equation}

We now introduce a notion of $\delta$-vanishing condition for the kernel coefficient $A$.
\begin{definition}\label{defn.del.van}
Let us fix $\delta$ and $R>0$. We say that $A$ is $(\delta,R)$-vanishing in $B\times B\subset\bbR^n\times\bbR^n$ if for any $B_r(x_0)\subset B$ with $r\in(0,R]$,
    \begin{equation*}
    \dashint_{B_r(x_0)\times B_r(x_0)}|A-(A)_{B_r(x_0)\times B_r(x_0)}|\,dx\,dy\leq \delta.
\end{equation*}
\end{definition}

We now state the main results of the paper. We first present the $L^q$ estimates of $u/d^s$ when the coefficient is possibly discontinuous.

\begin{theorem}\label{thm.vmo}
    Let $u\in W^{s,2}( B_{1})\cap L^1_{2s}(\bbR^n)$ be a weak solution to \eqref{eq: thm}. For any $q\in[2_*,n/s)$ and $\rho_0>0$, there is a constant $\delta=\delta(n,s,\Lambda,q)$ such that if $A$ is $(\delta,\rho_0)$-vanishing in $B_1\times B_1$, then we have 
\begin{equation}\label{est.thm.vmo}
    \|u/d^s\|_{L^{\frac{nq}{n-sq}}(\Omega\cap B_{1/2})}\leq c\left(\|u\|_{L^1_{2s}(\bbR^n)}+\|f\|_{L^q(\Omega\cap B_{1})}\right)
\end{equation}
for some constant $c=c(n,s,\Lambda,\alpha,q,\rho_0)$.
\end{theorem}

\begin{remark}\label{rmk.sharp.vmo}
  We assert that the estimates of Theorem \ref{thm.vmo} are sharp when we only require $A$ to satisfy a $(\delta,R)$-vanishing condition. Indeed, in the Appendix \ref{appen}, for any $\epsilon>0$, we construct a weak solution $u$ to \eqref{eq: defn} with $A\equiv 1$ and $f\in L^q(\Omega\cap B_1)$ such that $u/d^s\notin L^{\frac{nq}{n-sq}+\epsilon}(\Omega\cap B_{1/2})$.  We also point out that in the forthcoming Section \ref{sec5.3}, we prove that the estimate \eqref{est.thm.vmo} holds when $A$ satisfies a slightly different assumption from Definition \ref{defn.del.van}. In particular, such an assumption holds for a translation invariance-type kernel coefficient.
\end{remark}

\begin{remark}
  In \cite{AbeGru23}, it is shown that  $u/d^s\in H^{s,q}\subset L^{\frac{nq}{n-sq}}$ when $A\equiv 1$ or $A$ is translation invariant and homogeneous, and $\Omega$ is a $C^{1,\tau}$-domain, where $\tau>2s$. Thus, concerning a higher-integrability result of $u/d^s$, Theorem \ref{thm.vmo} relaxes the regularity assumption on the kernel coefficient $A$ by allowing it to have some discontinuity and the domain to be $C^{1,\alpha}$-regular for any $\alpha>0$.
\end{remark}

We introduce notions of Dini and H\"older continuity of the kernel coefficient.
\begin{definition}\label{defn.con}
    We say that $A$ is Dini continuous in $B_R(x_0)\times B_R(x_0)$ if 
    there is a non-decreasing function ${{\omega}}:\bbR_+\to\bbR_+$ with ${{\omega}}(0)=0$ such that
    \begin{equation*}
        |A(x_1,y_1)-A(x_2,y_2)|\leq {{\omega}}(\max\{|x_1-x_2|,|y_1-y_2|\})\quad\text{if }x_1,x_2,y_1,y_2\in B_R(x_0)
    \end{equation*}
 and 
\begin{equation}
    \int_{0}^{R}{{{\omega}}(\rho)}\frac{\,d\rho}{\rho} < \infty.
\end{equation}
In particular, if $\omega(\rho)\leq c\rho^\alpha$ for some constant $c\geq1$, then we say that $A$ is H\"older continuous of order $\alpha$ in $B_R(x_0)$.
\end{definition}

We now provide a boundedness result of $u/d^s$ when $A$ is Dini continuous.
\begin{theorem}\label{thm.dini}
    Let $u\in W^{s,2}(\Omega\cap B_{1})\cap L^1_{2s}(\bbR^n)$ be a weak solution to \eqref{eq: thm} with $A$ being Dini continuous in $B_1\times B_1$. If $f\in L^{n/s,1}(\Omega\cap B_{1})$, then we have 
    \begin{align*}
        \|u/d^s\|_{L^\infty(\Omega\cap B_{1/2})}\leq c\left(\|u\|_{L^1_{2s}(\bbR^n)}+\|f\|_{L^{n/s,1}(\Omega\cap B_{1})}\right),
    \end{align*}
    where $c=c(n,s,\Lambda,\alpha,{{\omega}})$. 
\end{theorem}
From this, we deduce a global $C^s$-regularity.
\begin{corollary}\label{cor.cs}
    Let $u\in W^{s,2}(\Omega\cap B_{1})\cap L^1_{2s}(\bbR^n)$ be a weak solution to \eqref{eq: thm} with $A$ being Dini continuous in $B_1\times B_1$. If $f\in L^{n/s,1}(\Omega\cap B_{1})$, then we have 
    \begin{align*}
        [u]_{C^s( B_{1/2})}\leq c\left(\|u\|_{L^1_{2s}(\bbR^n)}+\|f\|_{L^{n/s,1}(\Omega\cap B_{1})}\right),
    \end{align*}
    where $c=c(n,s,\Lambda,\alpha,{{\omega}})$. 
\end{corollary}

\begin{remark}\label{rmk.cs1}
We would like to mention that our estimates given in Theorem \ref{thm.dini} and Corollary \ref{cor.cs} are true when $A$ is Dini continuous in a weaker sense as in \eqref{ass1}. Indeed, we employ a well-known result that a weak solution $v$ to \eqref{eq: thm} with a constant kernel coefficient and $f=0$ satisfies $v/\oldphi\in C^{\epsilon}$ for some $\epsilon>0$ and some barrier function $\oldphi$ which is comparable to $d^s$, in order to obtain decay estimates as in Lemma \ref{lem.decay}. However, such a H\"older regularity is also true for a weak solution $v$ to \eqref{eq: thm} with a translation invariant kernel instead of constant kernels (see \cite[Theorem 1.6]{RosWei24}). 
    Thus, we obtain Theorem \ref{thm.dini} and Corollary \ref{cor.cs} which also hold for translation invariant kernel coefficients (see Section \ref{sec5.3} for details).
\end{remark}

\begin{remark}
Global $C^s$-regularity for solution to \eqref{eq: thm} is known only when $A$ satisfies \eqref{ass2} below, and $f\in L^{n/s+\epsilon}(\Omega\cap B_1)$ for any $\epsilon>0$. However, our result as in Corollary \ref{cor.cs} allows us to weaken the regularity assumptions on both the kernel coefficient and the right-hand side. 
\end{remark}

\begin{theorem}\label{thm.hol.1}
 Let $u\in W^{s,2}(\Omega\cap B_{1})\cap L^1_{2s}(\bbR^n)$ be a weak solution to \eqref{eq: thm} with $A$ being H\"older continuous of order $\alpha$ in $B_1\times B_1$.
 Let us also assume that, there is a constant $K\geq1$ such that 
 \begin{equation}\label{ass.thm.hol}
     [A]_{C^{\alpha}(B_1\times B_1)}\leq K.
 \end{equation}
 If $\sigma\in(0,\alpha)$ and $f\in L^q(\Omega\cap B_1)$  with $q\in[2_*,n/(s-\sigma))$, then 
    \begin{equation}\label{est.thm.hol.1}
        \|u/d^s\|_{ W^{\sigma,\frac{nq}{n-(s-\sigma)q}}(\Omega\cap B_{1/2})}\leq c\left(\|u\|_{L^1_{2s}(\bbR^n)}+\|f\|_{L^q(\Omega\cap B_1)}\right),
    \end{equation}
    where $c=c(n,s,\Lambda,\alpha,q,\sigma,K)$.
\end{theorem}

\begin{remark}
As in Remark \ref{rmk.sharp.vmo}, we also observe that the estimates of Theorem \ref{thm.hol.1} are sharp. Since if the estimates \eqref{est.thm.hol.1} were true when the space $W^{\sigma,\frac{nq}{n-(s-\sigma)q}}$ is replaced by either $W^{\sigma+\epsilon,\frac{nq}{n-(s-\sigma)q}}$ or $u/d^s\in W^{\sigma,\frac{nq}{n-(s-\sigma)q}+\epsilon}$ for some $\epsilon>0$, then by the Sobolev embedding, the estimates \eqref{est.thm.vmo} would hold with $L^{\frac{nq}{n-sq}}$ replaced by $L^{\frac{nq}{n-sq}+\varsigma}$ for some small $\varsigma>0$. This is not possible by the sharpness result of Theorem \ref{thm.vmo}. Therefore, the estimates \eqref{est.thm.hol.1} are also sharp.
\end{remark}

\begin{theorem}\label{thm.hol.2}
 Let $u\in W^{s,2}(\Omega\cap B_{1})\cap L^1_{2s}(\bbR^n)$ be a weak solution to \eqref{eq: thm}, where $A$ is H\"older continuous of order $\alpha$ in $B_1\times B_1$ and satisfies \eqref{ass.thm.hol} for some constant $K\geq1$. For any $\sigma\in (0,\alpha)$, if $f\in L^{n/(s-\sigma),\infty}(\Omega\cap B_1)$, then 
    \begin{equation*}
        [u/d^s]_{ C^{\sigma}(\overline{\Omega}\cap B_{1/2})}\leq c\left(\|u\|_{L^1_{2s}(\bbR^n)}+\|f\|_{L^{n/(s-\sigma),\infty}(\Omega\cap B_1)}\right),
    \end{equation*}
      where $c=c(n,s,\Lambda,\alpha,\sigma,K)$.
\end{theorem}

\begin{remark}
 We remark that in  Section \ref{sec5.3}, we obtain the estimates given in Theorem \ref{thm.hol.1} and Theorem \ref{thm.hol.2} are true under the assumption that $A$ satisfies slightly different continuity condition as in \eqref{ass2} below and the frozen kernel $A_z$ which is defined in \eqref{defn.az} below is homogeneous, see Section \ref{sec5.3} for more details.
\end{remark}

\begin{remark}
A few comments are in order concerning the regularity assumption on the right-hand side term $f$ in Theorem \ref{thm.hol.2} as compared to the available results. Let us fix $q\in(\frac{n}{s},\frac{n}{s-\alpha})$ with $\sigma\coloneqq s-n/q$. We note from \cite{KimWei24} that
    \begin{equation*}
        f\in L^{q}(\Omega\cap B_1)\Longrightarrow u/d^s\in C^{s-n/q}(\overline{\Omega}\cap B_{1/2}).
    \end{equation*}
    However, our result implies
    \begin{equation*}
        f\in L^{q,\infty}(\Omega\cap B_1)\Longrightarrow u/d^s\in C^{s-n/q}(\overline{\Omega}\cap B_{1/2}).
    \end{equation*}
    Since
    \begin{equation*}
         L^q(\Omega\cap B_1)\subsetneq L^{q,\infty}(\Omega\cap B_1)\subsetneq L^{q-\epsilon}(\Omega\cap B_1)
    \end{equation*}
    for any $\epsilon>0$, Theorem \ref{thm.hol.2} can be seen as a borderline result.
\end{remark}

Let us give an overview of our approach to get the desired results. Indeed, we start with a localization argument which allows us to assume $u\in W^{s,2}(\bbR^n)$ and $u\equiv 0$ in $\bbR^n\setminus B_2$. In addition, if $A$ is continuous in $B_1\times B_1$, we can always assume $A$ is globally continuous by such an argument.

For the case when $A$ is VMO, we employ an approximation method as in \cite{CP}, and Vitali's type covering argument, in order to derive 
\begin{equation*}
    \|M^{\Omega}_{2^{-6}}(u/d^s)\|_{L^{\frac{nq}{n-sq}}}\lesssim \|u\|_{L^1_{2s}(\bbR^n)}+\|M^{\Omega}_{2_*s,2^{-3}}(|f|^{2_*})\|_{L^{\frac{nq}{2_*(n-qs)}}}^{\frac{1}{2_*}}.
\end{equation*}

For the case when the associated kernel coefficient is continuous, we establish a suitable excess decay estimate as below:
\begin{align*}
    E(u;B_{\rho r})&\lesssim \rho^\alpha E(u;B_r)+\rho^{-n}({\pmb{\omega}}(r)+{\pmb{\omega}}_{G}(r))(E(u;B_r)+(u/d^s)_{B_r})\\
    &\quad+\rho^{-n}\max\{d(0),r\}^{-s}\mathrm{Tail}(u(\cdot){\pmb{\omega}}(|\cdot|);B_r)+ \rho^{-n}\left(\dashint_{\Omega\cap B_{r}}(r^s|f|)^{2_*}\,dx\right)^{\frac1{2_*}},
\end{align*}
where 
\begin{align*}
    {E}(u;B_{\rho})&\coloneqq \dashint_{B_{\rho}}|u/\oldphi-(u/\oldphi)_{B_{\rho}}|\,dx+\max\{\oldphi(0),\rho^s\}^{-1}\mathrm{Tail}(u-(u/\oldphi)_{B_{\rho}}\oldphi;B_{\rho})
\end{align*}
and the function $\oldphi$ is a barrier function that is comparable to the distance function $d^s$ (see Section \ref{sec5} for more details about this barrier) and we have written ${\pmb{\omega}}$ and ${\pmb{\omega}}_G$ as a local and a global oscillation-type quantity of the coefficient $A$ (see \eqref{defn.omega.sec5} below), respectively. We would like to mention that this excess functional was first introduced by \cite{KimWei24}. In light of the localization argument, we can handle the term ${\pmb{\omega}}_G(\cdot)$ as in the local term ${\pmb{\omega}}(\cdot)$. Therefore, based on the approach given in \cite{DuzMinjfa} and splitting the tail term appropriately together with suitable use of a weight function ${\pmb{\omega}}(|\cdot|)$, we finally establish 
\begin{align}\label{tech.pt}
    M^{\Omega}_{2^{-6}}(u/d^s)(0)\lesssim \|u\|_{L^1_{2s}(\bbR^n)}+W^{|f|^{2_*}\mbox{\Large$\chi$}_{\Omega}}_{\frac{2_*s}{2_*s+1},2_*+1}(0,2^{-2}),
\end{align}
if $A$ is Dini continuous.

Based on arguments given in \cite{KuuMin12} together with \eqref{tech.pt}, we now obtain a more sharp estimate 
\begin{align*}
    M^{\#,\Omega}_{\sigma,2^{-6}}(u/d^s)(0)\lesssim \|u\|_{L^1_{2s}(\bbR^n)}+W^{|f|^{2_*}\mbox{\Large$\chi$}_{\Omega}}_{\frac{2_*s}{2_*s+1},2_*+1}(0,2^{-2})+\left({M}^{\Omega}_{2_*(s-\sigma),2^{-2}}(|f|^{2_*})(0)\right)^{\frac1{2_*}}
\end{align*}
if $A$ is H\"older continuous with order $\alpha$, where $\sigma\in (0,\alpha)$.
We point out that in the proof, we indeed use a nonlocal fractional sharp maximal function which was first introduced by \cite{DieNow23} (see \eqref{defn.glosharp} for its precise definition).

The rest of the article has the following organization. In Section 2, we introduce some notations and provide embedding results. The localization argument and some well-known regularity results of solutions to fractional Laplacian type equations are also mentioned in this section. In Section 3, we provide comparison estimates. In Section 4, we prove Theorem \ref{thm.dini}. In Section 5, we prove Theorem \ref{thm.dini}, Corollary \ref{cor.cs}, Theorem \ref{thm.hol.1} and Theorem \ref{thm.hol.2}. In the appendix, we provide examples to show the sharpness of our estimates.

\section{Preliminaries}
Let us denote by $c$ to mean a universal constant which is bigger or equal to 1.  In addition, a parenthesis is used to denote the relevant dependencies of the constant $c$, e.g., $c=c(n,s)$, if the constant $c$ depends only on $n$ and $s$.

We first observe that there are constants $\rho_0=\rho_0(\Omega)\in(0,1]$ and $c_0=c_0(\Omega)\in(0,1]$ such that for any $\rho\in(0,\rho_0]$,
\begin{equation*}
    \inf_{y_0\in\partial\Omega}\frac{|B_\rho(y_0)\setminus\Omega|}{|B_\rho|}\geq c_0\quad\text{and}\quad \inf_{x_0\in\overline{\Omega}}\frac{|B_\rho(x_0)\cap\Omega|}{|B_\rho|}\geq c_0,
\end{equation*}
as $\Omega$ is a $C^{1,\alpha}$-domain. Therefore, we deduce that for any $\rho\in(0,1]$, there hold
\begin{equation}\label{rmk.bdry}
 \inf_{y_0\in\partial\Omega}\frac{|B_\rho(y_0)\setminus\Omega|}{|B_\rho|}\geq c_1\quad\text{and}\quad \inf_{x_0\in\overline{\Omega}}\frac{|B_\rho(x_0)\cap\Omega|}{|B_\rho|}\geq c_1,
\end{equation}
where $c_1=c_0\rho_0^n$ depends only on $\Omega$.

For convenience in writing, we abbreviate
\begin{equation*}
    \mathsf{data}=\mathsf{data}(n,s,\Lambda,\alpha,\Omega).
\end{equation*}
We now recall the truncated Wolff potential of a function $g\in L^1(B_{2R}(x_0))$ as
\begin{equation*}
    \mathbf{W}^{|g|}_{\beta,p}(z,r)\coloneqq\int_{0}^{r}\left(\frac{|g|(B_{\rho}(z))}{\rho^{n-\beta p}}\right)^{\frac1{p-1}}\frac{\,d\rho}{\rho}\quad\text{for any }r\in(0,R] \quad\text{and}\quad z\in B_R(x_0),
\end{equation*}
where $p\geq1$ and $\beta>0$.
By \cite{Cia11}, we have 
\begin{align}\label{ineq1.wolff}
    \| \mathbf{W}^{|g|}_{\beta,p}(z,R)\|_{L^{\frac{nq(p-1)}{n-\beta pq}}(B_{R/2}(x_0))}\leq c\|g\|_{L^q(B_{2R}(x_0))}
\end{align}
for some constant $c=c(n,q,p,\beta,R)$ whenever $\beta p<q$. In addition, we also observe from \cite{Cia11} that 
\begin{align}\label{ineq2.wolff}
    \| \mathbf{W}^{|g|}_{\beta,p}(z,R)\|_{L^{\infty}(B_{R/2}(x_0))}\leq c\|g\|_{L^{\frac{n}{\beta p},\frac1{(p-1)}}(B_{2R}(x_0))}
\end{align}
for some constant $c=c(n,p,\beta)$, whenever $\beta p<n$. We refer to \cite{SteWei71} for more details about the Lorentz spaces.

In what follows, when we consider the set $\Omega\cap B_R(x_0)$, we always assume $x_0\in \overline{\Omega}$.

\noindent
For any $g\in L^1(\Omega\cap B_{2R}(x_0))$ with $R\leq 1$ and $\beta\in[0,n)$, we define a fractional maximal function by
\begin{equation*}
    M^{\Omega}_{\beta,R}(g)(z)\coloneqq\sup_{0<\rho<R}\rho^{\beta}\dashint_{\Omega\cap B_\rho(z)}|g|\,dx\quad\text{for any }z\in \overline{\Omega}\cap B_R(x_0).
\end{equation*}
In particular, we write $M^{\Omega}_R(g)(z)\coloneqq M^{\Omega}_{0,R}(g)(z)$.
By \eqref{rmk.bdry}, we have for any $z\in\overline{\Omega}\cap B_R(x_0)$,
\begin{equation*}
     M^{\Omega}_{\beta,R}(g)(z)\leq c\sup_{0<\rho<R}\rho^\beta\dashint_{B_\rho(z)}|g\mbox{\Large$\chi$}_{\Omega}|\,dx\leq cM_{\beta,R}(g\mbox{\Large$\chi$}_{\Omega})(z)
\end{equation*}
for some constant $c=c(n,\Omega)$, where $M_{\beta,R}(g)$ is the standard truncated fractional maximal function. Therefore, using this and \cite[Theorem 3.1]{KinSak03}, we have
\begin{equation}\label{ineq1.pre.fmax}
\begin{aligned}
     \|M^{\Omega}_{\beta,R}(g)(z)\|_{L^{\frac{nq}{n-q\beta}}(\Omega\cap B_R(x_0))}\leq c\|M_{\beta,R}(g\mbox{\Large$\chi$}_{\Omega})(z)\|_{L^{\frac{nq}{n-q\beta}}(\Omega\cap B_R(x_0))} &\leq c\|g\mbox{\Large$\chi$}_{\Omega}\|_{L^q(B_{2R}(x_0))}\\
     &\leq c\|g\|_{L^q(\Omega\cap B_{2R}(x_0))}
    \end{aligned}
\end{equation}
for some constant $c=c(n,\beta,R,\Omega)$. In addition, from \cite[Proposition 2.5]{DieNow23}, we get
\begin{equation}\label{ineq2.pre.fmax}
\begin{aligned}
    \|M^{\Omega}_{\beta,R}(g)(z)\|_{L^{\infty}(\Omega\cap B_R(x_0))}\leq \|M_{\beta,R}(g\mbox{\Large$\chi$}_{\Omega})(z)\|_{L^{\infty}(\Omega\cap B_R(x_0))}&\leq c\|g\mbox{\Large$\chi$}_{\Omega}\|_{L^{n/\beta,\infty}(B_{2R}(x_0))}\\
    &\leq c\|g\|_{L^{n/\beta,\infty}(\Omega\cap B_{2R}(x_0))}
\end{aligned}
\end{equation}
for some constant $c=c(n,\beta,R,\Omega)$. We refer to \cite{SteWei71} for more details about the Marcinkiewicz spaces.
We next introduce a fractional sharp maximal function as
\begin{equation*}
    M^{{\#,\Omega}}_{\beta,R}(g)(z)\coloneqq\sup_{0<\rho\leq R}\rho^{-\beta}\dashint_{\Omega\cap B_\rho(z)}|g-(g)_{\Omega\cap B_\rho(z)}|\,dx\quad\text{for any }z\in \overline{\Omega}\cap B_R(x_0),
\end{equation*}
where $g\in L^1(\Omega\cap B_{2R}(x_0))$ with $R\leq1$ and $\beta\in[0,1]$.

We first assert the following inequality.
\begin{lemma}\label{lem.sharp}
Let $g\in L^1(\Omega\cap B_{2R}(x_0))$ with $R\leq 1$. Then for any $\beta\in(0,1]$ and all $x,y\in \overline{\Omega}\cap B_{R/4}(x_0)$, we have
    \begin{equation*}
        |g(x)-g(y)|\leq c\left(M^{{\#},\Omega}_{\beta,R}(g)(x)+M^{{\#},\Omega}_{\beta,R}(g)(y)\right)|x-y|^\beta
    \end{equation*}
    for some constant $c=c(n,\beta,\Omega)$.
\end{lemma}
\begin{proof}
    We will follow the same approach as in \cite[Proposition 1]{KuuMin14g}. Let us denote $\rho\coloneqq |x-y|\leq R/2$. Then we have 
    \begin{equation}\label{rel.yx}
        \Omega\cap B_{\rho}(y)\subset \Omega\cap B_{2\rho}(x).
    \end{equation}
 By \eqref{rmk.bdry},  we observe that
    \begin{align*}
        |(g)_{\Omega\cap B_{\rho_i}(x)}-(g)_{\Omega\cap B_{\rho_{i+1}}(x)}|\leq \dashint_{\Omega\cap B_{\rho_{i+1}}(x)}|g-(g)_{\Omega\cap B_{\rho_{i}}(x)}|\,d\tilde{x} &\leq c\dashint_{\Omega\cap B_{\rho_{i}}(x)}|g-(g)_{\Omega\cap B_{\rho_{i}}(x)}|\,d\tilde{x} \nonumber\\
        &\leq c(2^{-i}\rho)^{\beta}M^{{\#},\Omega}_{\beta,R}(g)(x)
    \end{align*}
    for some constant $c=c(n,\Omega)$, where $\rho_i=2^{-i}\rho$. Thus we have
    \begin{align*}
        |(g)_{\Omega \cap B_{\rho_i}(x)}-(g)_{\Omega \cap B_{\rho}(x)}|\leq c\rho^{\beta}M^{{\#},\Omega}_{\beta,R}(g)(x)
    \end{align*}
    for some constant $c=c(n,\beta)$, which is independent of $i$. Using this and the fact that $\Omega\cap B_{\rho_i}(x)=B_{\rho_i}(x)$ if $i$ is sufficiently large, we have 
    \begin{align}\label{ineq1.lem.sharp}
        |g(x)-(g)_{\Omega \cap B_{2\rho}(x)}|\leq c|x-y|^{\beta}M^{{\#},\Omega}_{\beta,R}(g)(x)
    \end{align}
    for some constant $c=c(n,\beta,\Omega)$. Similarly, we have 
    \begin{align}\label{ineq2.lem.sharp}
        |g(y)-(g)_{\Omega \cap B_{\rho}(y)}|\leq c|x-y|^{\beta}M^{{\#},\Omega}_{\beta,R}(g)(y).
    \end{align}
    In addition, we get 
    \begin{align*}
        |(g)_{\Omega \cap B_{2\rho}(x)}-(g)_{\Omega \cap B_{\rho}(y)}|\leq \dashint_{\Omega\cap B_{\rho}(y)}|g-(g)_{\Omega\cap B_{2\rho}(x)}|\leq \dashint_{\Omega\cap B_{2\rho}(x)}|g-(g)_{\Omega\cap B_{2\rho}(x)}|,
    \end{align*}
    where we have used \eqref{rel.yx} and \eqref{rmk.bdry} for the last inequality.
    Therefore, we obtain
    \begin{align}\label{ineq3.lem.sharp}
        |g(x)-g(y)|&\leq  |g(x)-(g)_{\Omega \cap B_{2\rho}(x)}|+|g(y)-(g)_{\Omega \cap B_{\rho}(y)}|+|(g)_{\Omega \cap B_{2\rho}(x)}-(g)_{\Omega \cap B_{\rho}(y)}|.
    \end{align}
    Combining all the estimates \eqref{ineq1.lem.sharp}, \eqref{ineq2.lem.sharp} and \eqref{ineq3.lem.sharp}, we get the desired result of the lemma.
\end{proof}

Using this, we have the following fractional Sobolev embedding via the sharp maximal function.
\begin{lemma}\label{lem.max}
    Let $\beta\in (0,1)$ and $p>1$. If $t\in (0,\beta)$ and $g\equiv0$ on $B_{R}(x_0)\setminus \Omega$ with $R\leq1$, then we have 
    \begin{align*}
        \|g\|_{W^{t,q}(\Omega\cap B_{R/4}(x_0))}\leq c\|M^{\#,\Omega}_{\beta,R/4}(g)\|_{L^p(\Omega\cap B_{R/2}(x_0))}+c\|g\|_{L^p(\Omega\cap B_{R/2}(x_0))},
    \end{align*}
    where $q\coloneqq \frac{np}{n-(\beta-t)p}$ and $c=c(n,\beta,t,p,\Omega,R)$.
\end{lemma}
\begin{proof}
 We may assume $x_0=0$ and $R=1$. We note that there is a radius $r_0=r_0(\Omega)\in(0,1/2]$ such that for any $z\in \partial\Omega$, there is a $C^{1,\alpha}$ function $\gamma\equiv\gamma_z:\bbR^{n-1}\to \bbR$ such that 
    \begin{equation}\label{assm.gamma}
        \Psi(\Omega\cap B_{r_0}(z))=\{(x',x_n)\in B_{r_0}(z)\,:x_n>\gamma(x')\}\quad\text{and}\quad \|\gamma\|_{C^{1,\alpha}}\leq c
    \end{equation}
    for some constant $c=c(\Omega)$, where $\Psi\equiv\Psi_z:\bbR^n\to\bbR^n$ is a rotation around the point $z$. Let us write 
    \begin{equation*}
        \Phi(x',x_n)\coloneqq(x',x_n-\gamma(x'))\quad\text{for any }x\in B_{r_0}(z).
    \end{equation*}
    Then there are constants $a_0=a_0(\Omega)\in(0,16)$ and $r_1=r_1(\Omega)\in(0,1)$ such that 
    \begin{equation}\label{ineq1.lem.max}
        (\Phi\circ\Psi)(\Omega\cap B_{a_0r_0}(z))\subset B_{r_1}^+(z',0)\subset B_{2r_1}^+(z',0)\subset  (\Phi\circ\Psi)(\Omega \cap B_{r_0/8}(z))
    \end{equation}
    and
    \begin{equation*}
        (\Phi\circ\Psi)(\partial\Omega\cap B_{a_0r_0}(z))\subset \{(x',x_n)\in B_{2r_1}\,:\,x_n=0\}\subset(\Phi\circ\Psi)(\partial\Omega\cap B_{r_0/8}(z)),
    \end{equation*}
    where we write 
    \begin{equation*}
       B_{2r_1}^+(z',0)\coloneqq\{(x',x_n)\in B_{2r_1}(z',0)\,:\,x_n>0\}.  
       \end{equation*}
    We first prove for any $z\in \partial\Omega$,
    \begin{equation}\label{ineq.fir.emb}
        \|g\|_{W^{t,q}(\Omega\cap B_{a_0r_0}(z))}\leq c\|M^{\#,\Omega}_{\beta,1/4}(g)\|_{L^p(\Omega\cap B_{r_0}(z))}+c\|g\|_{L^p(\Omega\cap B_{r_0}(z))},
    \end{equation}
    where $c=c(n,\beta,t,p,\Omega)$. 
    We now define 
    \begin{equation*}
        G(x)\coloneqq g((\Phi\circ\Psi)^{-1}(x))
    \end{equation*}
    to observe that 
    \begin{equation}\label{ineq11.lem.max}
    \begin{aligned}
        &\int_{\Omega\cap B_{a_0r_0}(z)}\int_{\Omega\cap B_{a_0r_0}(z)}\frac{|g(x)-g(y)|^q}{|x-y|^{n+tq}}\,dx\,dy\\
        &= \int_{ (\Phi\circ\Psi)(\Omega\cap B_{a_0r_0}(z))}\int_{ (\Phi\circ\Psi)(\Omega\cap B_{a_0r_0}(z))}\frac{|G(x)-G(y)|^q}{|x-y|^{n+tq}}\frac{|x-y|^{n+tq}}{|(\Phi\circ\Psi)^{-1}(x)-(\Phi\circ\Psi)^{-1}(y)|^{n+tq}}\,dx\,dy\\
        &\leq \int_{B_{r_1}^{+}(z',0)}\int_{B_{r_1}^{+}(z',0)}\frac{|G(x)-G(y)|^q}{|x-y|^{n+tq}}\,dx\,dy
    \end{aligned}
    \end{equation}
    for some constant $c=c(n,q,t,\Omega)$, where we have used change of variables, \eqref{assm.gamma}, \eqref{ineq1.lem.max} and
    \begin{equation*}
        |\mathrm{det}(\nabla (\Phi\circ\Psi)^{-1})|(x)=1.
    \end{equation*}
    We next define the function $\overline{G}:B_{2r_1}(z',0)\to\bbR$ as \begin{align*}
        \overline{G}(z',z_n)\coloneqq\begin{cases}
            G(z',z_n)&\quad\text{if }z_n\geq0,\\
            G(z',-z_n)&\quad\text{if }z_n<0.
        \end{cases}
    \end{align*}
    We then observe that for any $x,y\in B_{2r_1}^{+}(z',0)$, 
    \begin{equation}\label{defn1.overlineg}
    \begin{aligned}
        &|\overline{G}(x',x_n)-\overline{G}(y',y_n)|\\
        &=|g((\Phi\circ\Psi)^{-1}(x))-g((\Phi\circ\Psi)^{-1}(y))|\\
        &\leq c\left(M^{{\#},\Omega}_{\beta,r_0/2}(g)((\Phi\circ\Psi)^{-1}(x))-M^{{\#},\Omega}_{\beta,r_0/2}(g)((\Phi\circ\Psi)^{-1}(y))\right)|(\Phi\circ\Psi)^{-1}(x)-(\Phi\circ\Psi)^{-1}(y)|^\beta\\
        &\leq c\left(M^{{\#},\Omega}_{\beta,r_0/2}(g)((\Phi\circ\Psi)^{-1}(x))-M^{{\#},\Omega}_{\beta,r_0/2}(g)((\Phi\circ\Psi)^{-1}(y))\right)|x-y|^{\beta}
    \end{aligned}
    \end{equation}
    for some constant $c=c(n,\beta,\Omega)$, where we have used Lemma \ref{lem.sharp} and the second inequality in \eqref{assm.gamma}. Similarly, if  $x,y\in B_{2r_1}(z',0)$ with $x_n<0$ and $y_n>0$, then we get 
    \begin{equation}\label{defn.overlineg}
    \begin{aligned}
        &|\overline{G}(x',x_n)-\overline{G}(y',y_n)|\\
        &=|{G}(x',-x_n)-{G}(y',y_n)|\\
                &=|g((\Phi\circ\Psi)^{-1}(x',-x_n))-g((\Phi\circ\Psi)^{-1}(y))|\\
        &\leq c\left(M^{{\#},\Omega}_{\beta,r_0/2}(g((\Phi\circ\Psi)^{-1})(x',-x_n)-M^{{\#},\Omega}_{\beta,r_0/2}(g)((\Phi\circ\Psi)^{-1}(y))\right)|x-y|^{\beta},
    \end{aligned}
    \end{equation}
    where we have also used 
    \begin{equation*}
        |(x',-x_n)-(y',y_n)|\leq |x-y|\quad\text{if }x_n<0\quad\text{and}\quad y_n>0.
    \end{equation*}
    Combining \eqref{defn1.overlineg} and \eqref{defn.overlineg}, and noting \eqref{assm.gamma}, we deduce that
    \begin{equation}\label{ineq2.lem.max}
    \begin{aligned}
        \sup_{0<|h|<r_1/8}\left(\int_{B_{r_1}(z',0)}\frac{|\overline{G}(x+h)-\overline{G}(x)|^{p}}{|h|^{tp}}\,dx\right)^{\frac1{{p}}}&\leq c\left(\int_{B^{+}_{2r_1}(z',0)}|M^{{\#},\Omega}_{\beta,r_0/2}(g)((\Phi\circ\Psi)^{-1}(x))|^{{p}}\,dx\right)^{\frac{1}{{p}}}\\
        &\leq c\left(\int_{\Omega\cap B_{r_0/8}(z)}|M^{{\#},\Omega}_{\beta,1/4}(g)(x)|^{{p}}\,dx\right)^{\frac{1}{{p}}}.
    \end{aligned}
    \end{equation}
    As in the proof of \cite[Proposition 2.8]{DieNow23}, we get 
    \begin{align}\label{ineq3.lem.max}
        \|\overline{G}\|_{W^{t,q}(B_{r_1}(z',0))}\leq c\|\overline{G}\|^q_{L^p(B_{r_1}(z',0))}+c\sup_{0<|h|<r_1/8}\left(\int_{(B_{r_1}(z',0))}\frac{|\overline{G}(x+h)-\overline{G}(x)|^{{p}}}{|h|^{t{p}}}\,dx\right)^{\frac1{{p}}}
    \end{align}
    for some constant $c=c(n,s,\beta,p,\Omega)$. 
    Combining all the estimates \eqref{ineq11.lem.max}, \eqref{ineq2.lem.max} and \eqref{ineq3.lem.max}, we have \eqref{ineq.fir.emb}. 
    
    On the other hand, if $B_{r_0}(z)\subset \Omega$, we get \eqref{ineq.fir.emb} directly from  \cite[Proposition 2.8]{DieNow23}. Therefore, by a standard covering argument, we have the desired estimate of the lemma.
\end{proof}

We now present a localization argument that gives us the freedom to replace the local regularity assumptions on the weak solution and the kernel coefficient with the respective global ones. 
\begin{lemma}\label{lem.localization}
 Let $u\in W^{s,2}(B_{4R}(x_0))\cap L^1_{2s}(\bbR^n)$ be a weak solution to 
    \begin{equation}
\label{eq: localization}
\left\{
\begin{alignedat}{3}
\mathcal{L}_Au&= f&&\qquad \mbox{in  $\Omega \cap B_R(x_0)$}, \\
u&=0&&\qquad  \mbox{in $B_R(x_0)\setminus \Omega$},
\end{alignedat} \right.
\end{equation}
where $A$ satisfies \eqref{cond.kernel}. For any kernel coefficient $a=a(x,y)$ satisfying \eqref{cond.kernel}, there is a weak solution $\widetilde{u}\in W^{s,2}(\bbR^n)$ to 
\begin{equation}
\label{eq: localized}
\left\{
\begin{alignedat}{3}
\mathcal{L}_{\widetilde{A}}\widetilde{u}&= f+g_1+g_2&&\qquad \mbox{in  $\Omega \cap B_{R/2}(x_0)$}, \\
\widetilde{u}&=0&&\qquad  \mbox{in $B_{R/2}(x_0)\setminus \Omega$},\\
\widetilde{u}&=u&&\qquad  \mbox{in $B_{5R/8}(x_0)$},
\end{alignedat} \right.
\end{equation}
where $\widetilde{A}$ satisfies  \eqref{cond.kernel} along with $\widetilde{A}(x,y)=A(x,y)$ in $B_{5R/8}(x_0)\times B_{5R/8}(x_0)$,  $\widetilde{A}(x,y)=a(x,y)$ in $\bbR^{2n}\setminus (B_{3R/4}(x_0)\times B_{3R/4}(x_0))$, $|g_1(x)|\leq cR^{-2s}|u(x)|$ in $B_{R/2}(x_0)$  and
$g_2\in L^\infty(B_{R/2}(x_0))$ with the estimate
\begin{equation*}
    \|g_2\|_{L^\infty(B_{R/2}(x_0))}\leq cR^{-2s}\mathrm{Tail}(u;B_{5R/8}(x_0)).
\end{equation*}
Furthermore, if $A$ is continuous in $B_{R}(x_0)\times B_{R}(x_0)$ with 
\begin{equation}\label{ass.loc}
    |A(x,y)-A(x_0,x_0)|\leq \omega(\max\{|x-x_0|,|y-x_0|\})\quad\text{for any }x,y\in B_R(x_0),
\end{equation}
for some non-decreasing function $\omega:\bbR_+\to\bbR_+$ 
 with $\omega(0)=0$ and if we select $a=A(x_0,x_0)$, then $\widetilde{A}$ is continuous in $\bbR^n\times \bbR^n$ and satisfies
 \begin{equation}\label{cond.tildea}
    |\widetilde{A}(x,y)-\widetilde{A}(x_0,x_0)|\leq \widetilde{\omega}(\max\{|x-x_0|,|y-x_0|\})\quad\text{for any }x,y\in\bbR^n,
\end{equation}
where 
\begin{align*}
    \widetilde{\omega}(\rho)=\begin{cases}
        \omega(\rho)\quad&\text{if }\rho\leq R,\\
        0&\text{if }\rho> R.
    \end{cases} 
\end{align*}
\end{lemma}
\begin{proof}
    We may assume $x_0=0$. Let us take a cut off function $\phi\in C_c^\infty(B_{3R/4})$ with $\phi\equiv 1$ on $B_{5R/8}$. We denote 
    \begin{equation}\label{defn.tildeua}
        \widetilde{u}(x)\coloneqq u(x)\phi(x)\quad\text{and}\quad \widetilde{A}(x,y)\coloneqq A(x,y)\phi(x)\phi(y)+a(x,y)(1-\phi(x)\phi(y)).
    \end{equation}
 For any $\psi \in X_0^{s,2}(\Omega\cap B_{R/2})$, we have 
    \begin{align}\label{eq1.lem.loc}
        \int_{\bbR^n}\int_{\bbR^n}\frac{(u(x)-u(y))(\psi(x)-\psi(y))}{|x-y|^{n+2s}}A(x,y)\,dx\,dy=\int_{\Omega\cap B_{R/2}}f\psi\,dx.
    \end{align}
    We next observe 
    \begin{align*}
        &\int_{\bbR^n}\int_{\bbR^n}\frac{(\widetilde{u}(x)-\widetilde{u}(y))(\psi(x)-\psi(y))}{|x-y|^{n+2s}}(\widetilde{A}(x,y)-A(x,y))\,dx\,dy\\
        &=2\int_{\bbR^n\setminus B_{5R/8}}\int_{B_{R/2}}\frac{(u(x)-\widetilde{u}(y))\psi(x)}{|x-y|^{n+2s}}(\widetilde{A}(x,y)-A(x,y))\,dx\,dy
    \end{align*}
    and
    \begin{equation}\label{eq2.lem.loc}
    \begin{aligned}
        &\int_{\bbR^n}\int_{\bbR^n}\frac{[(\widetilde{u}(x)-\widetilde{u}(y))-({u}(x)-{u}(y))](\psi(x)-\psi(y))}{|x-y|^{n+2s}}A(x,y)\,dx\,dy \\
        &=2\int_{\bbR^n\setminus B_{5R/8}}\int_{B_{R/2}}\frac{(-\widetilde{u}(y)+{u}(y))\psi(x)}{|x-y|^{n+2s}}A(x,y)\,dx\,dy.
    \end{aligned}
    \end{equation}
    Combining the above three equations, we obtain
    \begin{align*}
        \int_{\bbR^n}\int_{\bbR^n}\frac{(\widetilde{u}(x)-\widetilde{u}(y))(\psi(x)-\psi(y))}{|x-y|^{n+2s}}\widetilde{A}(x,y)\,dx\,dy=\int_{\Omega\cap B_{R/2}}(f+g_1+g_2)\psi,
    \end{align*}
    where 
    \begin{equation*}
        g_1(x)=2\int_{\bbR^n\setminus B_{5R/8}}\frac{u(x)}{|x-y|^{n+2s}}(\widetilde{A}(x,y)-A(x,y))\,dy
    \end{equation*}
    and
    \begin{equation*}
        g_2(x)=\int_{\bbR^n\setminus B_{5R/8}}\frac{(-4\widetilde{u}(y)+2{u}(y))}{|x-y|^{n+2s}}A(x,y)\,dy
    \end{equation*}
    for any $x\in B_{R/2}$. This implies that $\widetilde{u}\in W^{s,2}(\bbR^n)$ is a weak solution to \eqref{eq: localization} and
    \begin{align*}
        |g_1(x)|\leq cR^{-2s}|u(x)|\text{ for any } x\in B_{R/2}\quad\text{and}\quad \sup_{x\in B_{R/2}}|g_2(x)|\leq cR^{-2s}\mathrm{Tail}(u;B_{5R/8})
    \end{align*}
    for some constant $c=c(n,s,\Lambda)$. 

     We now assume that $A$ is continuous in $B_R\times B_R$ satisfying \eqref{ass.loc} with $x_0=0$. Then
     \begin{align*}
         |\widetilde{A}(x,y)-\widetilde{A}(0,0)|&=|A(x,y)\phi(x)\phi(y)+A(0,0)(1-\phi(x)\phi(y))-A(0,0)|\\
         &=|(A(x,y)-A(0,0))\phi(x)\phi(y)|.
     \end{align*}
    If $(x,y)\in B_R\times B_R$, we use \eqref{ass.loc} with $x_0=0$ to get \eqref{cond.tildea}. On the other hand, if $(x,y)\notin \bbR^{2n}\setminus(B_R\times B_R)$, we use the fact that $\phi\equiv 0$ in $\bbR^n\setminus B_{3R/4}$ to get \eqref{cond.tildea}. This completes the proof.
\end{proof}

\begin{remark}\label{rmk.localization}
 We remark that if we do not want to localize the coefficient function $A$, then choosing a cut off function $\phi\in C_{c}^\infty(B_{{15R/16}}(x_0))$ with $\phi\equiv 1$ on $B_{7R/8}(x_0)$ taking into account  \eqref{eq1.lem.loc} and \eqref{eq2.lem.loc}, we observe that $\widetilde{u}\in W^{s,2}(\bbR^n)$ is a weak solution to 
\begin{equation*}
\left\{
\begin{alignedat}{3}
\mathcal{L}_{{A}}\widetilde{u}&= f+g_2&&\qquad \mbox{in  $\Omega \cap B_{3R/4}(x_0)$}, \\
\widetilde{u}&=0&&\qquad  \mbox{in $B_{3R/4}(x_0)\setminus \Omega$},\\
\widetilde{u}&=u&&\qquad  \mbox{in $B_{15R/16}(x_0)$},
\end{alignedat} \right.
\end{equation*}
where $g_2\in L^\infty(B_{3R/4}(x_0)$ with 
\begin{equation*}
    \|g_2\|_{L^\infty(B_{3R/4})(x_0)}\leq cR^{-2s}\mathrm{Tail}(u;B_{15R/16}(x_0))
\end{equation*}
for some constant $=c(n,s,\Lambda)$.
\end{remark}

We next provide the following fractional Hardy type inequality, which will be frequently used in the remaining parts.
\begin{lemma}\label{lem.sobpoi.ds}
    Let $g\in W^{s,2}(\bbR^n)$ with $g=0$ in $\bbR^n\setminus (\Omega \cap B_R(x_0) )$ with $R\leq1$. Then we have  
    \begin{equation*}
        \int_{\Omega\cap B_R(x_0)}\left|\frac{g}{d^s}\right|^2\,dx\leq \frac{c}{\max\left\{1,d^s(x_0)/R^s\right\}^2}[g]^2_{W^{s,2}(\bbR^n)}
    \end{equation*}
    for some constant $c=c(n,s,\Omega)$. 
\end{lemma}
\begin{proof}
    We first assume $d(x_0)>2R$. Then $B_{3R/2}(x_0)\subset \Omega$ and $d(x)>d(x_0)/4$ for any $x\in B_{5R/4}(x_0)$. Therefore, by \cite[Lemma 4.7]{Coz17}, we have 
    \begin{align}\label{ineq1.lem.sobpoi}
        \int_{\Omega\cap B_R(x_0)}\left|\frac{g}{d^s}\right|^2\,dx\leq c\frac{R^{2s}}{d^{2s}(x_0)}\int_{\Omega\cap B_R(x_0)}\left|\frac{g}{R^s}\right|^2\,dx\leq c\frac{R^{2s}}{d^{2s}(x_0)}[g]^2_{W^{s,2}(\bbR^n)}.
    \end{align}
    We next assume $d(x_0)\leq 2R$. Then by \cite[Lemma 2.4]{KimWei24}, we have 
    \begin{align}\label{ineq2.lem.sobpoi}
        \int_{\Omega\cap B_R(x_0)}\left|\frac{g}{d^s}\right|^2\,dx\leq c[g]^2_{W^{s,2}(\bbR^n)}
    \end{align}
    for some constant $c=c(n,s,\Omega)$. Combining two estimates \eqref{ineq1.lem.sobpoi} and \eqref{ineq2.lem.sobpoi} yields the desired result of the lemma.
\end{proof}

We prove self-improving properties of the fractional Laplacian type equations up to the boundary. 
\begin{lemma}\label{lem.self}
Let $R\in(0,1]$ and let $w\in W^{s,2}(\bbR^n)$ be a weak solution to 
    \begin{equation}
\label{eq: self}
\left\{
\begin{alignedat}{3}
\mathcal{L}_Aw&= 0&&\qquad \mbox{in  $\Omega \cap B_{4R}(x_0)$}, \\
w&=0&&\qquad  \mbox{in $B_{4R}(x_0)\setminus \Omega$}.
\end{alignedat} \right.
\end{equation}
Then there is a constant $\delta_0=\delta_0(\mathsf{data})\in(0,1)$ such that $w\in W^{s+\frac{n\delta_0}{2+2\delta_0},2+2\delta_0}(B_{2R}(x_0))$ with the estimate 
\begin{equation}\label{ineq.self.1}
    R^{s+\frac{n\delta_0}{2+2\delta_0}-\frac{n}{2}}[w]_{W^{s+\frac{n\delta_0}{2+2\delta_0},2+2\delta_0}(B_{2R}(x_0))}\leq c\left(R^{s-\frac{n}{2}}\|w\|_{L^{2}(B_{3R}(x_0))}+\mathrm{Tail}(w;B_{3R}(x_0))\right),
\end{equation} 
where $c=c(n,s,\Lambda)$.
\end{lemma}

\begin{proof}
We may assume $x_0=0$. For $\tau\coloneqq\min\{s,1-s\}/2$, we define
\begin{equation*}
    W(x,y)\coloneqq\frac{{w}(x)-{w}(y)}{|x-y|^\tau}\quad\text{and}\quad\mu(E)\coloneqq\int_{E}\frac{\,dx\,dy}{|x-y|^{n-2\tau}},
\end{equation*}
where $x,y\in\bbR^n$ and $E\subset\bbR^{2n}$ is any measurable set. We are now going to prove that there is a constant $c=c(\mathsf{data})$ such that
\begin{equation}\label{ger.ineq}
\begin{aligned}
    \left(\dashint_{\mathcal{B}_\rho(z)}W^2\,d\mu\right)^{\frac12}&\leq \frac{c}{\sigma}\left(\dashint_{\mathcal{B}_{2\rho}(z)}W^{q}\,d\mu_\tau\right)^{\frac{1}{q}}+{\sigma}\sum_{k=1}^\infty 2^{-k(s-\tau)}\left(\dashint_{\mathcal{B}_{2^k\rho}}W^q\,d\mu\right)^{\frac1q},
\end{aligned}
\end{equation}
whenever $\mathcal{B}_\rho(z)\coloneqq B_\rho(z)\times B_\rho(z)$ for any  $z\in  B_{2R}$, $\rho\leq R/100$ and $\sigma\in(0,1)$, where we have set
\begin{equation}\label{cond.q.self}
    q\coloneqq \frac{2n+4\tau}{n+2s+2\tau}\quad\text{and}\quad \eta\coloneqq \frac{s-\tau}{n+2\tau}.
\end{equation}
We first assume $B_{3\rho/2}(z)\cap \partial\Omega=\emptyset$. Then we have $({w}-({w})_{B_{\rho}(z)})\phi^2\in X_{0}^{s,2}(\Omega\cap B_{4R})$ for any $\phi\in C_c^\infty(B_{5\rho/4}(z))$ with $\phi\equiv 1$ on $B_{\rho}(z)$. Thus by testing $({w}-({w})_{B_{\rho}(z)})\phi^2$ to \eqref{eq: self} with $x_0=0$ and by following the same lines as in the proof of \cite[Theorem 3.2]{KuuMinSir15s}, we obtain
\begin{align*}
    \rho^{-n+2s}[{w}]^2_{W^{s,2}(B_\rho(z))}&\leq \frac{c}{\sigma}\dashint_{B_{2\rho}(z)}|{w}-({w})_{B_{2\rho}(z)}|^2\,dx+\sigma\left(\rho^{2s}\int_{\bbR^n\setminus B_{2\rho}(z)}\frac{|{w}-({w})_{B_{2\rho}(z)}|}{|y-z|^{n+2s}}\,dy\right)^2
\end{align*}
for some constant $c=c(\mathsf{data})$. Thus, in this case, the proof of \eqref{ger.ineq} can be completed exactly as in \cite[Proposition 4.4]{KuuMinSir15s}. 

We next assume $B_{3\rho/2}(z)\cap \partial\Omega\neq\emptyset$. By \eqref{rmk.bdry}, we have that for any $a\geq 2$ with $a\rho\leq 1$,
\begin{equation*}
     |B_{a\rho}(z)\setminus \Omega|>c|B_{a\rho}|
\end{equation*} holds for some constant $c=c(n,\Omega)$. If $1<a\rho\leq 4\mathrm{diam}(\Omega)$, then
\begin{equation*}
    |B_{a\rho}(z)\setminus \Omega|>|B_{\rho_0}(z)\setminus \Omega|\geq c_0|B_{\rho_0}|\geq c_0\left(\frac{\rho_0}{4\mathrm{diam}(\Omega)}\right)^n|B_{a\rho}|,
\end{equation*}
where the constant $\rho_0=\rho_0(\Omega)$ is determined in \eqref{rmk.bdry}. On the other hand, if $a\rho>4\mathrm{diam}(\Omega)$, then there is a point $y\in B_{a\rho}(z)$ such that $B_{a\rho/8}(y)\subset B_{a\rho}(z)$ but $B_{a\rho/8}(y)\cap \Omega=\emptyset$. Thus, we have 
\begin{equation*}
    |B_{a\rho}(z)\setminus \Omega|>|B_{a\rho/8}(y)\setminus \Omega|>1/8^n|B_{a\rho}|
\end{equation*}
whenever $a\rho>4\mathrm{diam}(\Omega)$. Combining all the three cases, we have  for any $a\geq2$
\begin{equation}\label{ineq.measden}
     |B_{a\rho}(z)\setminus \Omega|>c|B_{a\rho}|,
\end{equation}
where $c=c(n,\Omega)$.
Therefore, by \cite[Corollary 4.9]{Coz17} with $s$ and $p$ ,
replaced by $s+\tau-2\tau/q$ and $q$, respectively, we have for any $a\geq 2$,
\begin{equation}\label{ineq1.sob}
\begin{aligned}
    \left(\dashint_{B_{a\rho}(z)}|{w}|^2 \,dx\right)^{\frac12}&\leq c(a\rho)^{s+\tau-2\tau/q}\left(\int_{B_{a\rho}(z)}\dashint_{B_{a\rho}(z)}\frac{|{w}(x)-{w}(y)|^{q}}{|x-y|^{n+q(s+\tau)-2\tau}}\,dx\,dy\right)^{\frac1q}\\
    &\leq c(a\rho)^{2(s+\tau)}\left(\dashint_{\mathcal{B}_{a\rho}(z)}|W|^q\,d\mu\right)^{\frac1q}
\end{aligned}
\end{equation}
for some constant $c=c(\mathsf{data})$, where we have also used \eqref{cond.q.self} for the last inequality. Similarly, as in the proof of \cite[Theorem 3.2]{KuuMinSir15s} together with a suitable choice of testing function $w\phi^2$, where $\phi\in C^\infty_c(B_{15\rho/8}(z))$ with $\phi\equiv 1$ on $B_{7\rho/4}(z)$, we get
\begin{align*}
    \rho^{-n+2s}[w]^2_{W^{s,2}(B_{\rho}(z))}&\leq\rho^{-n+2s}[w]^2_{W^{s,2}(B_{7\rho/4}(z))}\\
    &\leq \frac{c}\sigma\dashint_{B_{2\rho}(z)}|w|^2\,dx+\sigma\left(\rho^{2s}\int_{\bbR^n\setminus B_{2\rho}(z)}\frac{|w|}{|y-z|^{n+2s}}\,dy\right)^2.
\end{align*}
We next follow the same computations as in the proof of \cite[Proposition 4.4]{KuuMinSir15s} along with \eqref{ineq1.sob} to see that \eqref{ger.ineq} holds.

Since we have a reverse H\"older's type inequality on the diagonal as in \eqref{ger.ineq}, by following the same lines as in the proof of \cite[Theorem 1.3]{KuuMinSir15s}, we obtain 
\begin{align*}
    \left(\dashint_{\mathcal{B}_{2R}}|W|^{2+\delta_1}\,d\mu\right)^{\frac1{2+\delta_1}}&\leq c\left[\sum_{k=1}^\infty 2^{-k(s-\tau)}\left(\dashint_{B_{3R\times 2^k}}|W|^2\,d\mu\right)^{\frac12}+R^{s-\tau}\left(\dashint_{\mathcal{B}_{3R}}F^2\,d\mu\right)^{\frac12}\right]
\end{align*}
for some constant $\delta_1=\delta_1(\mathsf{data})\in(0,1)$, where $c=c(\mathsf{data})$. After a few simple computations and using the fact that 
\begin{equation*}
    \sum_{k=1}^\infty 2^{-k(s-\tau)}\left(\dashint_{B_{3R\times 2^k}}|W|^2\,d\mu\right)^{\frac12}\leq cr^{-n/2}[w]_{W^{s,2}(\bbR^n)},
\end{equation*} we obtain 
\begin{align*}
    R^{s+\frac{n\delta_0}{2+2\delta_0}-\frac{n}{2}}[w]_{W^{s+\frac{n\delta_0}{2+2\delta_0},2+2\delta_0}(B_{2R})}&\leq cR^{-\frac{n}{2}}(R^s[w]_{W^{s,2}(B_{23R/8})}+\|w\|_{L^2(B_{23R/8})})\\
    &\quad+c\mathrm{Tail}(w;B_{5R/2}),
\end{align*}
where $c=c(\mathsf{data})$. We now use standard energy inequalities as in \cite[Theorem 3.2]{KuuMinSir15s} to get the desired estimate \eqref{ineq.self.1}.
\end{proof}

We end this section with the following local boundedness result up to the boundary.
\begin{lemma}\label{lem.locb}
    Let $v\in W^{s,2}(B_R(x_0))\cap L^1_{2s}(\bbR^n)$ be a weak solution to 
    \begin{equation*}
\left\{
\begin{alignedat}{3}
\mathcal{L}_Av&= 0&&\qquad \mbox{in  $\Omega \cap B_{R}(x_0)$}, \\
v&=0&&\qquad  \mbox{in $B_{R}(x_0)\setminus \Omega$}.
\end{alignedat} \right.
\end{equation*}
Then there is a constant $c=c(\mathsf{data})$ such that 
\begin{equation}\label{lem.locb.ineq}
    \|v\|_{L^\infty(B_{R/2}(x_0))}\leq c(R^{-n}\|v\|_{L^1(B_R(x_0))}+\mathrm{Tail}(v;B_R(x_0))).
\end{equation}
\end{lemma}
\begin{proof}
    By \cite[Theorem 2 and Theorem 5]{KorKuuPal16}, we have 
    \begin{equation*}
    \|v\|_{L^\infty(B_{r/2}(z))}\leq c(r^{-n/2}\|v\|_{L^2(B_r(z))}+\mathrm{Tail}(v;B_r(z)))
    \end{equation*}
    for some constant $c=c(\mathsf{data})$, whenever $z\in \overline{\Omega}\cap B_{3R/4}(x_0)$ and $r\leq R/4$. Therefore, by following the same lines as the proof of \cite[Corollary 2.1]{KuuMinSir15s}, we get \eqref{lem.locb.ineq}.
\end{proof}

\section{Comparison estimates}\label{sec.comp}
In this section, we provide several comparison estimates.

In light of the localization argument as in Lemma \ref{lem.localization}, we may assume that $u\in W^{s,2}(\bbR^n)$ is a solution to 
\begin{equation}\label{eq: loc.comp}
\left\{
\begin{alignedat}{3}
\mathcal{L}_Au&= f&&\qquad \mbox{in  $\Omega \cap B_{R}(x_0)$}, \\
u&=0&&\qquad  \mbox{in $B_{R}(x_0)\setminus \Omega$},
\end{alignedat} \right.
\end{equation}
where $f\in L^{2_*}(\Omega\cap B_{R}(x_0))$ with $R\leq 1$.

Let us fix $B_{4r}(z)\Subset B_R(x_0)$ with $z\in\overline{\Omega}$. We present our first comparison estimate below.
\begin{lemma}\label{comp1.conti}
Let $w\in W^{s,2}(\bbR^n)$ be the weak solution to 
\begin{equation}\label{eq.comp.w}
\left\{
\begin{alignedat}{3}
\mathcal{L}_Aw&= 0&&\qquad \mbox{in  $\Omega \cap B_{2r}(z)$}, \\
w&=u&&\qquad  \mbox{in $\mathbb{R}^n\setminus(\Omega\cap B_{2r}(z))$}.
\end{alignedat} \right.
\end{equation}
 Then 
 \begin{align}\label{ineq3.comp1.conti}
          r^{-n}[u-w]^2_{W^{s,2}(\bbR^n)}\leq c\left(\dashint_{ \Omega\cap B_{2r}(z)}(r^{s}|f|)^{2_*}\,dx\right)^{\frac2{2_*}}
     \end{align}
     for some constant $c=c(\mathsf{data})$.
\end{lemma}
\begin{proof}
By \cite[Proposition 2.12]{BraLinSch18}, let  $w\in X^{s,2}_{u}(\Omega\cap B_{2r}(z),B_{3r}(z))$ be the weak solution to \eqref{eq.comp.w}. Since $u\in W^{s,2}(\bbR^n)$, we have $w\in W^{s,2}(\bbR^n)$ and $u-w\in X_0^{s,2}(\Omega\cap B_{2r}(z))$. Thus, by testing $u-w$ to 
    \begin{equation*}
        \mathcal{L}_A(u-w)=f\quad\text{in }\Omega\cap B_{2r}(z),  
    \end{equation*}
     we have 
     \begin{align*}
         [u-w]^2_{W^{s,2}(\bbR^n)}\leq c\int_{\Omega\cap B_{2r}(z)}|f||u-w|\,dx
     \end{align*}
     where $c=c(\mathsf{data})$. We now use Young's inequality and the Sobolev embedding as in \cite[Lemma 4.8]{Coz17} to obtain
     \begin{align}\label{ineq3.comp1.cont.aux}
          r^{-n}[u-w]^2_{W^{s,2}(\bbR^n)}\leq  \frac{r^{-n}}{2}[u-w]^2_{W^{s,2}(\bbR^n)}+c\left(\dashint_{\Omega\cap B_{2r}(z)}(r^{s}|f|)^{2_*}\,dx\right)^{\frac2{2_*}}
     \end{align}
     for some constant $c=c(\mathsf{data})$, where we have used \eqref{rmk.bdry} to see that 
     \begin{equation*}
         r^n/c\leq|\Omega\cap B_{2r}(z)|\leq  cr^n.
     \end{equation*}
 From \eqref{ineq3.comp1.cont.aux}, the conclusion of the lemma follows immediately. 
\end{proof}

We next provide a second comparison estimate.
\begin{lemma}\label{comp2.conti}
    Let $v\in W^{s,2}(\bbR^n)$ be a weak solution to 
    \begin{equation}
\label{eq.comp.v}
\left\{
\begin{alignedat}{3}
\mathcal{L}_{a_0}v&= 0&&\qquad \mbox{in  $\Omega \cap B_{r}(z)$}, \\
v&=w&&\qquad  \mbox{in $\bbR^n\setminus (\Omega\cap B_{r}(z) )$},
\end{alignedat} \right.
\end{equation}
where $a_0$ is a kernel coefficient satisfying \eqref{cond.kernel} and $w$ is the solution to \eqref{eq.comp.w}. Then, we have 
\begin{equation}\label{ineq3.comp2}
    \begin{aligned}
        &r^{-n}[w-v]_{W^{s,2}(\bbR^n)}^2\\
        &\leq c\big( r^{-2n}\|w/r^s\|^2_{L^{1}(B_{2r}(z))}+r^{-2s}\mathrm{Tail}(w;B_{2r}(z))^2 \big)r^{-\frac{2n\delta_0}{1+\delta_0}}\|A-a_0\|_{L^{\frac{2(1+\delta_0)}{\delta_0}}(B_{2r}(z)\times B_{2r}(z))}^{2}\\
        &\hspace{0.2pt}+cr^{4s}\left(\sup_{x\in B_r(z)}\int_{\bbR^n\setminus B_{2r}(z)}\frac{|A-a_0|}{|y-z|^{n+2s}}\,dy\right)^2 \big(r^{-2n}\|w/r^s\|^2_{L^{1}(B_{2r}(z))}+r^{-2s}\mathrm{Tail}(w;B_{2r}(z))^2\big)\\
        &+ cr^{2s}\left(\sup_{x\in B_r(z)}\int_{\bbR^n\setminus B_{2r}(z)}\frac{|w(y)||A-a_0|}{|y-z|^{n+2s}}\,dy\right)^2,
    \end{aligned}
    \end{equation}
where the constant $\delta_0=\delta_0(\mathsf{data})\in(0,1)$ is determined in Lemma \ref{lem.self}.
\end{lemma}
\begin{proof}
We may assume $z=0$. By testing $w-v$ to
    \begin{equation*}
        \mathcal{L}_{A-a_0}w+\mathcal{L}_{a_0}(w-v)=0\quad\text{in }\Omega\cap B_r,
    \end{equation*}
    we get 
    \begin{align*}
         [w-v]_{W^{s,2}(\bbR^n)}^2
         &\leq c\int_{B_{3r/2}}\int_{B_{3r/2}}\frac{|w(x)-w(y)||(w-v)(x)-(w-v)(y)||A-a_0|}{|x-y|^{n+2s}}\,dx\,dy\\
         &\quad+c\int_{\bbR^n\setminus B_{3r/2}}\int_{B_{r}}\frac{|w(x)-w(y)||(w-v)(x)||A-a_0|}{|x-y|^{n+2s}}\,dx\,dy\eqqcolon J_1+J_2.
    \end{align*}
With the help of Young's inequality and H\"older's inequality, we now estimate $J_1$ as
    \begin{align*}
       J_1 &\leq \frac{[w-v]_{W^{s,2}(\bbR^n)}}{8}+c\int_{B_{3r/2}}\int_{B_{3r/2}}\frac{|w(x)-w(y)|^2|A-a_0|^2}{|x-y|^{n+2s}}\,dx\,dy\\
        &\leq \frac{[w-v]_{W^{s,2}(\bbR^n)}}8+c[w]^{2}_{W^{s+\frac{n\delta_0}{2+2\delta_0},2+2\delta_0}(B_{3r/2})}\left(\int_{B_{3r/2}} \int_{B_{3r/2}}|A-a_0|^{\frac{2(1+\delta_0)}{\delta_0}}\,dx\,dy\right)^{\frac{\delta_0}{1+\delta_0}}
    \end{align*}
    for some constant $c=c(\mathsf{data})$.
    By Lemma \ref{lem.self}, we further simplify the above display as
    \begin{equation}\label{ineq1.lem.comp2}
    \begin{aligned}
        J_1&\leq \frac{[w-v]_{W^{s,2}(\bbR^n)}}8+c\left(\|w\|^2_{L^{2}(B_{7r/4})}+r^{n-2s}\mathrm{Tail}(w;B_{7r/4})^2\right)r^{-\frac{2n\delta_0}{1+\delta_0}}\|A-a_0\|_{L^{\frac{2(1+\delta_0)}{\delta_0}}(B_{2r}\times B_{2r})}^{2}
    \end{aligned}
    \end{equation}
    for some constant $c=c(\mathsf{data})$.
    We next note $|x-y|\geq |y|/2$ for any $y\in\bbR^n\setminus B_{3r/2}$ and $x\in B_{r}$. Thus, 
    \begin{align*}
        J_2&\leq c\int_{B_r}|w(x)||(w-v)(x)|\int_{\bbR^n\setminus B_{3r/2}}\frac{|A-a_0|}{|y|^{n+2s}}\,dx\,dy\\
        &\quad+c\int_{B_r}|(w-v)(x)|\int_{\bbR^n\setminus B_{3r/2}}\frac{|w(y)||A-a_0|}{|y|^{n+2s}}\,dx\,dy\eqqcolon J_{2,1}+J_{2,2}.
    \end{align*}
    We first estimate $J_{2,1}$ as 
    \begin{align*}
        J_{2,1}&\leq c\int_{B_r}|w(x)||(w-v)(x)|\int_{B_{2r}\setminus B_{3r/2}}\frac{|A-a_0|}{|y|^{n+2s}}\,dx\,dy\\
        &\quad+c\int_{B_r}|w(x)||(w-v)(x)|\int_{\bbR^n\setminus B_{2r}}\frac{|A-a_0|}{|y|^{n+2s}}\,dx\,dy\coloneqq J_{2,1,1}+J_{2,1,2}.
    \end{align*}
    By Young's inequality, H\"older's inequality and the Sobolev embedding, we have
    \begin{align*}
        J_{2,1,1}&\leq \frac{c}{r^{2s}}\int_{B_r}|w(x)|^2\dashint_{B_{2r}}|A-a_0|^2\,dy\,dx+\frac{[w-v]_{W^{s,2}(\bbR^n)}^2}{8}\\
        &\leq \frac{cr^{\frac{n\delta_0}{1+\delta_0}
        }}{r^{2s}}\|w\|^{2}_{L^{2+2\delta_0}(B_{r})}\left(\dashint_{B_{2r}}\dashint_{B_{2r}}|A-a_0|^{\frac{2(1+\delta_0)}{\delta_0}}\,dy\,dx\right)^{\frac{\delta_0}{1+\delta_0}}+\frac{[w-v]_{W^{s,2}(\bbR^n)}^2}{8}\\
        &\leq c[w]^{2}_{W^{s+\frac{n\delta_0}{2+2\delta_0},2+2\delta_0}(B_{r})}\left(\int_{B_{2r}} \int_{B_{2r}}|A-a_0|^{\frac{2(1+\delta_0)}{\delta_0}}\,dx\,dy\right)^{\frac{\delta_0}{1+\delta_0}}+\frac{[w-v]_{W^{s,2}(\bbR^n)}^2}{8},
    \end{align*}
    where $c=c(\mathsf{data})$. Thus, as in the estimate of \eqref{ineq1.lem.comp2}, we obtain 
    \begin{align*}
        J_{2,1,1}\leq \frac{[w-v]^2_{W^{s,2}(\bbR^n)}}8+c[\|w\|^2_{L^{2}(B_{7r/4})}+r^{n-2s}\mathrm{Tail}(w;B_{7r/4})^2]r^{-\frac{2n\delta_0}{1+\delta_0}}\|A-a_0\|_{L^{\frac{2(1+\delta_0)}{\delta_0}}(B_{2r}\times B_{2r})}^{2}
    \end{align*}
    for some constant $c=c(\mathsf{data})$.
 Using Young's inequality and the Sobolev embedding, we next estimate $J_{2,1,2}$ as 
    \begin{align*}
        J_{2,1,2}&\leq cr^{2s}\left(\sup_{x\in B_r}\int_{\bbR^n\setminus B_{2r}}\frac{|A-a_0|}{|y|^{n+2s}}\,dy\right)^2\|w\|^2_{L^2(B_r)}+\frac{[w-v]^2_{W^{s,2}(\bbR^n)}}8.
    \end{align*}
    Similarly, we also have 
    \begin{align*}
        J_{2,2}&\leq c\|w\|^2_{L^{2}(B_{7r/4})}r^{-\frac{2n\delta_0}{1+\delta_0}}\|A-a_0\|_{L^{\frac{2(1+\delta_0)}{\delta_0}}(B_{2r}\times B_{2r})}^{2}\\
        &\quad +cr^{2s+n}\left(\sup_{x\in B_r}\int_{\bbR^n\setminus B_{2r}}\frac{|A-a_0||w(y)|}{|y|^{n+2s}}\,dy\right)^2+\frac{[w-v]^2_{W^{s,2}(\bbR^n)}}4.
    \end{align*}
    By combining all the estimates $J_1$, $J_{2,1,1}$ $J_{2,1,2}$ and $J_{2,2}$ together with the fact that 
    \begin{equation*}
       \|w\|_{L^2(B_{7r/4})}^2+r^{n}\mathrm{Tail}(w;B_{7r/4})^2\leq c( r^{-n}\|w\|^2_{L^{1}(B_{2r})}+r^{n}\mathrm{Tail}(w;B_{2r})^2),
    \end{equation*}
    which follows from Lemma \ref{lem.locb}, we get the desired estimate \eqref{ineq3.comp2}.
\end{proof}

\section{The case of discontinuous coefficients}
In this section, we prove Theorem \ref{thm.vmo}.
Let us fix a weak solution $u\in W^{s,2}(B_1)\cap L^1_{2s}(\bbR^n)$ to \eqref{eq: defn}. By the localization argument, we see the following remark.
\begin{remark}\label{rmk.loc.sec4}
By Remark \ref{rmk.localization}, we observe that there is a weak solution $\widetilde{u}\in W^{s,2}(\bbR^n)$ to 
\begin{equation}\label{eq.sec4}
\left\{
\begin{alignedat}{3}
\mathcal{L}_{{A}}\widetilde{u}&= f+g_2&&\qquad \mbox{in  $\Omega \cap B_{3/4}$}, \\
\widetilde{u}&=0&&\qquad  \mbox{in $B_{3/4}\setminus \Omega$},\\
\widetilde{u}&=u&&\qquad  \mbox{in $B_{15/16}$},
\end{alignedat} \right.
\end{equation}    
where 
\begin{equation}\label{ineq.g2.sec4}
    \|g_2\|_{L^\infty(B_{3/4})}\leq c\mathrm{Tail}(u;B_{15/16}))
\end{equation}
for some constant $c=c(n,s,\Lambda)$.
\end{remark}
\subsection{Localized problems.} In this subsection, on account of Remark \ref{rmk.loc.sec4}, we fix a weak solution $u\in W^{s,2}(\bbR^n)$ to 
\begin{equation*}
\left\{
\begin{alignedat}{3}
\mathcal{L}_{{A}}{u}&= f&&\qquad \mbox{in  $\Omega \cap B_{3/4}$}, \\
{u}&=0&&\qquad  \mbox{in $B_{3/4}\setminus \Omega$}.
\end{alignedat} \right.
\end{equation*}   

We first provide the following weighted H\"older estimates for homogeneous problems with a locally constant kernel coefficient.
\begin{lemma}\label{lem.loc.const}
Let $z\in {\Omega}\cap B_{1/4}$ and $r\in(0,1/64]$. Let $v\in W^{s,2}(\bbR^n)$ be a weak solution to
\begin{equation}
\label{eq: homo}
\left\{
\begin{alignedat}{3}
\mathcal{L}_{a_{0}}v&= 0&&\qquad \mbox{in  $\Omega \cap B_{r}(z)$}, \\
v&=0&&\qquad  \mbox{in $B_{r}(z)\setminus \Omega$},
\end{alignedat} \right.
\end{equation}
where 
\begin{equation}\label{defn.a0.vmo}
\begin{aligned}
        a_0(x,y)\coloneqq\begin{cases} (A)_{B_{2r}(z)}&\quad\text{if }(x,y)\in B_{2r}(z)\times B_{2r}(z),\\
        A(x,y)&\quad\text{if }(x,y)\in \bbR^n\times\bbR^n\setminus (B_{2r}(z)\times B_{2r}(z)).
        \end{cases}
    \end{aligned}
    \end{equation}
Then we have 
\begin{equation*}
    r^s\|v/d^s\|_{L^\infty(\Omega \cap B_{r/4}(z))}\leq c\big(r^{-n}\|v\|_{L^1(B_{r}(z))}+\mathrm{Tail}(v;B_{r}(z))\big)
\end{equation*}
for some constant $c=c(\mathsf{data})$.
\end{lemma}
\begin{proof}
We may assume $z=0$. With the help of Lemma \ref{lem.locb}, we first observe that  
    \begin{equation}\label{vmo.homo.ineq1}
        \|v\|_{L^\infty(B_{3r/4})}\leq c(\|v\|_{L^1(B_{r})}+\mathrm{Tail}(v;B_{r}))
    \end{equation}
 for some constant $c=c(\mathsf{data})$. By following the same lines as in the proof of Lemma \ref{lem.localization} with $A(x,y)=a_0(x,y)$, $x_0=z$, $R=r$ and $a=(A)_{B_{2r}(z)\times B_{2r}(z)}$, we see that there is a weak solution $\widetilde{v}$ to 
    \begin{equation*}
\left\{
\begin{alignedat}{3}
\mathcal{L}_{\widetilde{A}}\widetilde{v}&= g_1+g_2&&\qquad \mbox{in  $\Omega \cap B_{r/2}$}, \\
\widetilde{v}&=0&&\qquad  \mbox{in $B_{r/2}\setminus \Omega$},\\
\widetilde{v}&=v&&\qquad  \mbox{in $B_{5r/8}$}.
\end{alignedat} \right.
\end{equation*}
In addition, we get
\begin{equation*}
    \widetilde{A}(x,y)=(A)_{B_{2r}(z)\times B_{2r}(z)}
\end{equation*}
and
\begin{equation}\label{vmo.homo.ineq2}
\begin{aligned}
    \sup_{x\in B_{r/2}}(|g_1(x)|+|g_2(x)|)&\leq cr^{-2s}\left[\sup_{x\in B_{r/2}}|v(x)|+\mathrm{Tail}(v;B_{5r/8})\right]\\
    &\leq cr^{-2s}(r^{-n}\|v\|_{L^1(B_{r})}+\mathrm{Tail}(v;B_{r})),
\end{aligned}
\end{equation}
where we have used \eqref{vmo.homo.ineq1}.
In light of \cite[Proposition 2.7.8]{FerRos24} along with \eqref{vmo.homo.ineq1}, \eqref{vmo.homo.ineq2} and the fact that $\widetilde{v}\equiv 0$ in $\bbR^n \setminus B_{3r/4}$ and $|\widetilde{v}(x)|\leq|v(x)|$, we have 
\begin{align*}
    r^{s}\|\widetilde{v}/d^s\|_{L^{\infty}(\Omega\cap B_{r/4})}&\leq c(\|\widetilde{v}\|_{L^\infty(B_{3r/4})}+r^{2s}\|g_1+g_2\|_{L^\infty(\Omega\cap B_{r/2})})\\
    &\leq c(r^{-n}\|v\|_{L^1(B_{r})}+\mathrm{Tail}(v;B_{r})),
\end{align*}
where $c=c(\mathsf{data})$. Since $\widetilde{v}\equiv v$ in $B_{5r/8}$,  the desired estimate follows.
\end{proof}

By considering two cases $d(z)\leq 2r$ and $d(z)>2r$, we simply obtain the following inequality
\begin{equation}\label{ineq.simple}
    \dashint_{B_r(z)}|g|/r^{s}\,dx\leq c\max\{1,{d^s(z)}/{r^s}\} \dashint_{B_r(z)}|g|/d^{s}\,dx
\end{equation}
for some constant $c=c(n,s)$. We note that this inequality will be frequently used in the remaining parts.

Using comparison estimates given in Section \ref{sec.comp} along with Lemma \ref{lem.loc.const}, we have the following result.
\begin{lemma}\label{lem.disc.comp}
    Let $z\in{\Omega}\cap B_{1/4}$ and $r\in(0,1/64]$. For any $\epsilon>0$, there is a constant $\delta=\delta(\mathsf{data},\epsilon)\in(0,1)$ such that if 
    \begin{equation}\label{comp.vmo.ass1}
        \dashint_{\Omega\cap B_{4r}(z)}|u/d^s|\,dx+\max\{d(z),4r\}^{-s}\mathrm{Tail}(u;B_{4r}(z))\leq \lambda
    \end{equation}
    and
    \begin{equation}\label{comp.vmo.ass2}
        \lambda\dashint_{B_{2r}(z)}\dashint_{B_{2r}(z)}|A-(A)_{B_{2r}(z)\times B_{2r}(z)}|\,dx\,dy+\left(\dashint_{\Omega\cap B_{2r}(z)}(r^s|f|)^{2_*}\,dx\right)^{\frac1{2_*}}\leq \delta\lambda,
    \end{equation}
    then there is a weak solution $v$ to
    \begin{equation*}
        \mathcal{L}_{a_0}v=0\quad\text{in }\Omega\cap B_r(z),
    \end{equation*}
    where $a_0$ is as defined in \eqref{defn.a0.vmo}, satisfying 
    \begin{equation*}
        \|v/d^s\|_{L^\infty(\Omega\cap B_{r/4}(z))}\leq c\lambda\quad\text{and}\quad \dashint_{\Omega\cap B_r(z)}|u/d^s-v/d^s|\,dx\leq c\epsilon\lambda
    \end{equation*}
    for some constant $c=c(\mathsf{data})$.
\end{lemma}
\begin{proof}
    Let $w$ and $v$ be the weak solutions to \eqref{eq.comp.w} and \eqref{eq.comp.v}, respectively. By Lemma \ref{comp1.conti} and Lemma \ref{comp2.conti} together with the fact that $A(x,y)=a_0(x,y)$ in $\bbR^{2n}\setminus (B_{2r}(z)\times B_{2r}(z))$, we get 
    \begin{align*}
        &r^{-n}[u-v]^2_{W^{s,2}(\bbR^n)}\leq 2r^{-n}[u-w]^2_{W^{s,2}(\bbR^n)}+2r^{-n}[w-v]^2_{W^{s,2}(\bbR^n)}\\
        &\leq c\left(\dashint_{\Omega\cap B_{2r}(z)}(r^{s}|f|)^{2_*}\,dx\right)^{\frac2{2_*}}\\
        &\quad+c\big(r^{-2n}\|w/r^s\|^2_{L^{1}(B_{2r}(z))}+r^{-2s}\mathrm{Tail}(w;B_{2r}(z))^2\big)r^{-\frac{2n\delta_0}{1+\delta_0}}\|A-a_0\|_{L^{\frac{2(1+\delta_0)}{\delta_0}}(B_{2r}(z)\times B_{2r}(z))}^{2}
    \end{align*}
    for some constant $c=c(\mathsf{data})$. 
    We note from H\"older's inequality, the Sobolev embedding and Lemma \ref{comp1.conti} that 
    \begin{align*}
        r^{-2n}\|w/r^s\|^2_{L^{1}(B_{2r}(z))}&\leq c\left(r^{-2n}\|(w-u)/r^s\|^2_{L^{1}(B_{2r}(z))} +r^{-2n}\|u/r^s\|^2_{L^{1}(B_{2r}(z))}\right)\\
        &\leq c\left(r^{-n}[w-u]_{W^{s,2}(\bbR^n)}^2+r^{-2n}\|u/r^s\|^2_{L^{1}(B_{2r}(z))}\right)\\
        &\leq c\left[\left(\dashint_{\Omega\cap B_{2r}(z)}(r^{s}|f|)^{2_*}\,dx\right)^{\frac2{2_*}}+r^{-2n}\|u/r^s\|^2_{L^{1}(B_{2r}(z))}\right].
    \end{align*}
    Combining the above two inequalities along with \eqref{comp.vmo.ass2}, we get
    \begin{align}\label{ineq3.comp.vmo}
        r^{-n}[u-v]^2_{W^{s,2}(\bbR^n)}
        &\leq c\left[(\delta\lambda)^2+\delta^{\frac{\delta_0}{1+\delta_0}}(r^{-2n}\|u/r^s\|^2_{L^{1}(B_{2r}(z))}+r^{-2s}\mathrm{Tail}(u;B_{2r}(z))^2)\right].
    \end{align}
    We now estimate the left-hand side of the above inequality by considering two cases that $d(z)\leq 2r$ or $d(z)> 2r$.

    We first assume $d(z)\leq 2r$. Then we  observe from \eqref{ineq.simple} and \eqref{comp.vmo.ass1} that
    \begin{equation}\label{ineq.vmo.ud}
    \begin{aligned}
        &r^{-n}\|u/r^s\|_{L^{1}(B_{2r}(z))}+r^{-s}\mathrm{Tail}(u;B_{2r}(z))\\
        &\leq c\left(\dashint_{\Omega\cap B_{4r}(z)}|u/d^s|\,dx+\max\{d(z),r\}^{-s}\mathrm{Tail}(u;B_{4r}(z))\right)\leq c\lambda.
    \end{aligned}
    \end{equation}
    In addition, by H\"older's inequality together with  \eqref{rmk.bdry} and Lemma \ref{lem.sobpoi.ds}, we have
    \begin{equation*}
        r^{-2n}\|u/d^s-v/d^s\|_{L^1(\Omega\cap B_{r}(z))}\leq cr^{-n}\|u/d^s-v/d^s\|_{L^2(\Omega\cap B_{r}(z))}\leq cr^{-n}[u-v]^2_{W^{s,2}(\bbR^n)}.
    \end{equation*}
    Therefore, we have 
    \begin{equation}\label{delta1.lem.disc.comp}
        \dashint_{\Omega\cap B_r(z)}|u/d^s-v/d^s|\,dx\leq  c\left(\delta\lambda+\delta^{\frac{\delta_0}{2(1+\delta_0)}}\lambda\right) \leq c\epsilon\lambda
    \end{equation}
    by taking $\delta$ sufficiently small depending only on $\mathsf{data}$ and $\epsilon$. In addition, from Lemma \ref{lem.loc.const}, the Sobolev embedding and Lemma \ref{comp1.conti}, we deduce that 
    \begin{align*}
        \|v/d^s\|_{L^\infty(\Omega\cap B_{r/4}(z))}&\leq c(r^{-n}\|v/r^s\|_{L^1(B_{r}(z))}+\mathrm{Tail}(v/r^s;B_{r}(z)))\\
        &\leq c\left(\dashint_{ B_{2r}(z)}(r^{s}|f|)^{2_*}\,dx\right)^{\frac1{2_*}}+r^{-n}\|u/r^s\|_{L^{1}(B_{2r}(z))}+r^{-s}\mathrm{Tail}(u;B_{2r}(z)).
    \end{align*}
    We then use \eqref{comp.vmo.ass2} and \eqref{ineq.vmo.ud} to see that there is a constant $c=c(\mathsf{data})$ such that
    \begin{align}\label{ineq1.lem.dis.comp}
        \|v/d^s\|_{L^\infty(\Omega\cap B_{r/4}(z))}\leq c\lambda.
    \end{align}
   We now consider the case $d(z)>2r$.
    We observe from \eqref{ineq.simple} and \eqref{comp.vmo.ass1} that 
    \begin{equation}\label{ineq.vmo.ud2}
    \begin{aligned}
        &r^{-n}\|u/r^s\|_{L^{1}(B_{2r}(z))}+r^{-s}\mathrm{Tail}(u;B_{2r}(z))\\
        &\leq \frac{cd^{s}(z)}{r^{s}}\left[\dashint_{\Omega\cap B_{4r}(z)}|u/d^s|\,dx+c\max\{d(z),4r\}^{-s}\mathrm{Tail}(u;B_{4r}(z))\right]\leq \frac{cd^{s}(z)}{r^{s}}\lambda.
    \end{aligned}
    \end{equation}
    Then, using the Sobolev embedding, we get 
    \begin{equation*}
        r^{-n}\|u/d^s-v/d^s\|_{L^1(\Omega\cap B_{r}(z))}\leq \frac{cr^{s}}{d^{s}(z)}r^{-n}\|u/r^s-v/r^s\|_{L^1(\Omega\cap B_{r}(z))}\leq \frac{cr^{s}}{d^{s}(z)}r^{-n/2}[u-v]_{W^{s,2}(\bbR^n)}.
    \end{equation*}
 Plugging these inequalities into \eqref{ineq3.comp.vmo} yields 
    \begin{equation}\label{delta2.lem.disc.comp}
        \dashint_{\Omega\cap B_r(z)}|u/d^s-v/d^s|\,dx\leq  c\delta\lambda+c\delta^{\frac{\delta_0}{2(1+\delta_0)}}\lambda \leq c\epsilon\lambda
    \end{equation}
    by taking $\delta$ sufficiently small. Moreover, we also have
    \begin{equation}\label{ineq2.lem.dis.comp}
    \begin{aligned}
        \|v/d^s\|_{L^\infty(\Omega\cap B_{r/4}(z))}&\leq \frac{cr^{s}}{d^{s}(z)}(r^{-n}\|v/r^s\|_{L^1(B_{r}(z))}+\mathrm{Tail}(v/r^s;B_{r}(z)))\\
        &\leq c\left(\dashint_{ B_{2r}(z)}(r^{s}|f|)^{2_*}\,dx\right)^{\frac1{2_*}}\\
        &\quad+\frac{cr^{s}}{d^{s}(z)}\left(r^{-n}\|u/r^s\|_{L^{1}(B_{2r}(z))}+r^{-s}\mathrm{Tail}(u;B_{2r}(z))\right)\\
        &\leq c\lambda
    \end{aligned}
    \end{equation}
    for some constant $c=c(\mathsf{data})$, where we have used Lemma \ref{lem.locb}, \eqref{comp.vmo.ass2} and \eqref{ineq.vmo.ud2}. Therefore combining the estimates \eqref{delta1.lem.disc.comp}, \eqref{ineq1.lem.dis.comp}, \eqref{delta2.lem.disc.comp} and \eqref{ineq2.lem.dis.comp} together with the choice of $\delta=\delta(\mathsf{data},\epsilon)$ determined by \eqref{delta1.lem.disc.comp} and \eqref{delta2.lem.disc.comp} yields the desired result. 
\end{proof}

Let us set
\begin{equation}\label{choi.lambda1}
\begin{aligned}
    \lambda_1\coloneqq &\dashint_{\Omega\cap B_{1/2}}|u/d^s|\,dx+\max\{d(0),1/2\}^{-s}\mathrm{Tail}(u;B_{1/2})+\left(\dashint_{\Omega\cap B_{1/2}}|f|^{2_*}\,dx\right)^{\frac1{2_*}}.
\end{aligned}
\end{equation}
For $\lambda\geq \lambda_1$ and $N\geq1$, we define 
\begin{equation*}
    \mathcal{C}\coloneqq\{z\in \Omega\cap B_{1/4}\,:\,M_{2^{-6}}^{\Omega}(|u/d^s|)(z)>N\lambda\}
\end{equation*}
and
\begin{equation*}
    \mathcal{D}\coloneqq \{z\in \Omega\cap B_{1/4}\,:\,M_{2^{-6}}^{\Omega}(|u/d^s|)(z)>\lambda\}\cup \{z\in \Omega\cap B_{1/4}\,:\,M_{2_*s,2^{-6}}^{\Omega}(|f|^{2_*})(z)>(\delta\lambda)^{2_*}\}.
\end{equation*}
By the weak 1-1 estimate, we get 
\begin{align}\label{weak11.sec4}
    |\mathcal{C}|\leq \frac{c}N
\end{align}
for all $\lambda\geq \lambda_1$, where $c=c(\mathsf{data})$.

\begin{lemma}\label{lem.disc.measd}
Let us fix $z\in {\Omega}\cap B_{1/4}$ and $\rho\in(0,2^{-10}]$.
    There exists $N\geq1$ so that for any $\epsilon\in(0,1)$, there are $\delta=\delta(\mathsf{data},\epsilon)\in(0,1)$ such that for any $\lambda\geq\lambda_1$ with $A$ being $(\delta,2\rho)$-vanishing, if 
    \begin{equation*}
        |\mathcal{C}\cap B_{\rho/8}(z)|\geq \epsilon|B_{\rho/8}|,
    \end{equation*}
    then we have 
    \begin{equation*}
        \Omega \cap B_{\rho/8}(z) \subset \mathcal{D}.
    \end{equation*}
\end{lemma}
\begin{proof}
For $\epsilon_1>0$, we select $\delta=\delta(\mathsf{data},\epsilon_1)\in (0,1)$ as given by Lemma \ref{lem.disc.comp}. On the contrary, assume that $ \Omega\cap  B_{\rho/8}(z) \not\subset \mathcal{D}$. This means that there exists $\tilde{x}\in  B_{\rho/8}(z)\cap \Omega$ so that 
\begin{align}\label{ineq.vmo.mdn}
\sup_{0<r\leq2^{-6}}\dashint_{\Omega\cap B_r(\tilde{x})}|u/d^s|\,dx
\leq \lambda \quad \mbox{and }
 \sup_{0<r\leq2^{-6}}\,\dashint_{\Omega\cap B_r(\tilde{x})}|r^sf|^{2_*}\,dx \leq (\delta\lambda)^{2_*}.
\end{align}
We now verify the hypotheses of Lemma \ref{lem.disc.comp} by establishing
\begin{equation}\label{comp.vmo.md.ass1}
    \dashint_{\Omega\cap B_{4\rho}(z)}|u/d^s|\,dx+\max\{d(z),4\rho\}^{-s}\mathrm{Tail}(u;B_{4\rho}(z))\leq c\lambda
    \end{equation}
    and
    \begin{equation}\label{comp.vmo.md.ass2}
        \lambda\dashint_{B_{2\rho}(z)}\dashint_{B_{2\rho}(z)}|A-(A)_{B_{2\rho}(z)\times B_{2\rho}(z)}|\,dx\,dy+\left(\dashint_{\Omega\cap B_{4\rho}(z)}(\rho^s|f|)^{2_*}\,dx\right)^{\frac1{2_*}}\leq c\delta\lambda,
    \end{equation}
for some $c=c(\mathsf{data})$. We fix  a non-negative integer $N_\rho$ such that 
\begin{align*}
    {1}/{16} < 2^{N_\rho+7} \rho \leq 1/8.
\end{align*}
Then, we have the following inclusion
\begin{align}\label{ineq.vmo.mdn4}
    B_{2^{i+2} \rho}(z) \subset B_{2^{i+3} \rho}(\tilde{x})\subset B_{1/2}\quad\mbox{for all }i=0,1,\dots,N_\rho+4.
\end{align}
Thus, using \eqref{ineq.vmo.mdn} and \eqref{ineq.vmo.mdn4}, we have 
\begin{align}
 \dashint_{\Omega\cap B_{4\rho}(z)}|u/d^s|\,dx \leq c \dashint_{\Omega\cap B_{8\rho}(\tilde{x})}|u/d^s|\,dx \leq c\lambda
\end{align}
and 
\begin{align}
  \dashint_{\Omega\cap B_{4\rho}(z)} (\rho^s|f|)^{2_*}\,dx \leq c \dashint_{\Omega\cap B_{8\rho}(\tilde{x})} (\rho^s|f|)^{2_*}\,dx \leq c\delta\lambda.
\end{align}
Now by \cite[Lemma 2.2]{DieKimLeeNow24n}, we estimate the nonlocal tail term as 
\begin{align*}
 &\max\{d(z),4\rho\}^{-s}\mathrm{Tail}(u;B_{4\rho}(z))\\
 &\leq c\max\{d(z),4\rho\}^{-s}\rho^{2s}\left[\sum_{i=0}^{N_\rho}(2^i\rho)^{-2s}\dashint_{B_{2^{i+2}\rho}(z)}|u|\,dy+\int_{\bbR^n\setminus B_{2^{N_\rho}\rho}(z)}\frac{|u(y)|}{|y-z|^{n+2s}}\,dy\right] \eqqcolon J_1+J_2.
\end{align*}
Using \eqref{ineq.simple} and \eqref{ineq.vmo.mdn} together with $u\equiv0$ on $B_{1/2}\setminus \Omega$,  we further estimate $J_1$ as
\begin{align*}
    J_1\leq c\sum_{i=0}^{N_\rho}2^{-is}\dashint_{\Omega\cap B_{2^{i+2}\rho}(z) }|u/d^s|\,dy\leq c\sum_{i=0}^{N_\rho}2^{-is}\dashint_{\Omega\cap B_{2^{i+3}\rho}(\widetilde{x})}|u/d^s|\,dy\leq c\lambda\sum_{i=0}^{N_\rho}2^{-is}\leq c\lambda
\end{align*}
for some constant $c=c(\mathsf{data})$, where we have also used \eqref{rmk.bdry} and the fact that $2^{N_\rho+3}\rho\leq 2^{-6}$. On the other hand, using \eqref{ineq.simple} and the fact that $d(0)\leq c$ for some constant $c=c(\Omega)$, we estimate $J_2$ as 
\begin{align*}
    J_2&\leq c\max\{d(z),4\rho\}^{-s}\rho^{2s}\left[\int_{B_{1/2}\setminus \Omega}|u(y)|\,dy+\mathrm{Tail}(u;B_{1/2})\right]\\
    &\leq c\left(\dashint_{B_{1/2}\setminus \Omega}|u/d^s|\,dy+\max\{d(0),1/2\}^{-s}\mathrm{Tail}(u;B_{1/2})\right)\leq c\lambda,
\end{align*}
where $c=c(\mathsf{data})$.
Therefore, combining the above displays and noting that $A$ is $(\delta,2\rho)$-vanishing (see Definition \ref{defn.del.van}), we observe that \eqref{comp.vmo.md.ass1} and \eqref{comp.vmo.md.ass2} hold true. Thus, by Lemma \ref{lem.disc.comp} with $r$ and $\lambda$ replaced by $\rho$ and $c\lambda$, respectively, we get that there is a function $v$ such that 
  \begin{equation}\label{comp.vmo.v-rg}
        \|v/d^s\|_{L^\infty(\Omega\cap B_{\rho/4}(z))}\leq c\lambda\quad\text{and}\quad \dashint_{\Omega\cap B_{\rho}(z)}|u/d^s-v/d^s|\,dx\leq c\epsilon_1\lambda
    \end{equation}
for some constant $c=c(\mathsf{data})$. We next observe \begin{align}\label{comp.vmo.u-v.ms}
 \mathcal{C}\cap B_{\rho/8}(z) \subset \{x\in \Omega\cap B_{\rho/8}(z)\,:\,M_{\rho/16}^{\Omega}(|u/d^s-v/d^s|)(x)>\lambda\}
\end{align} 
provided that $N$ is chosen sufficiently large depending only on the $\mathsf{data}$. Indeed, let $\hat{x}\in \Omega\cap B_{\rho/8}(z)$ such that $M_{\rho/16}^{\Omega}(|u/d^s-v/d^s|)(\hat{x})\leq\lambda$. 
\begin{enumerate}
    \item If $r\in(0,\rho/16)$, then we note from \eqref{comp.vmo.v-rg} that there is a constant $c=c(\mathsf{data})$ such that
    \begin{align*}
        \dashint_{\Omega\cap B_{r}(\hat{x})}|u/d^s|\,dx&\leq \dashint_{\Omega\cap B_{r}(\hat{x})}|(u-v)/d^s|\,dx+\dashint_{\Omega\cap B_{r}(\hat{x}}|v/d^s|\,dx\\
        &\leq M_{\rho/16}^{\Omega}(|u/d^s-v/d^s|)(\hat{x})+\|v/d^s\|_{L^{\infty}(\Omega\cap B_{\rho/4}(z))}\leq c\lambda.
    \end{align*}
    \item If $r\in[\rho/16,2^{-10}]$, then we first observe $B_{r}(\hat{x})\subset B_{8r}(\widetilde{x})$.
    Therefore, using this and \eqref{ineq.vmo.mdn}, we have that there is a constant $c=c(\mathsf{data})$ satisfying
    \begin{align*}
        \dashint_{\Omega\cap B_{r}(\hat{x})}|u/d^s|\,dx\leq c\dashint_{\Omega\cap B_{8r}(\widetilde{x})}|u/d^s|\,dx\leq c\lambda.
    \end{align*}
    \item If $r\in[2^{-10},2^{-6}]$, then we get 
    \begin{align*}
        \dashint_{\Omega\cap B_{r}(\hat{x})}|u/d^s|\,dx\leq c\dashint_{\Omega\cap B_{1/2}}|u/d^s|\,dx\leq c\lambda
    \end{align*}
    for some constant $c=c(\mathsf{data})$.
\end{enumerate}
From $(a)$, $(b)$ and $(c)$, selecting $N$ large enough depending only on $c$ as appearing above, we get that $M_{2^{-6}}^{\Omega}(|u/d^s|)(\hat{x})\leq N\lambda$, which implies $\hat{x}\notin \mathcal{C}\cap B_{\rho/8}(z)$. This proves \eqref{comp.vmo.u-v.ms}. Therefore, by \eqref{comp.vmo.u-v.ms}, the weak $(1,1)$-estimate and \eqref{comp.vmo.v-rg}, we get 
 \begin{align*}
     |\mathcal{C}\cap B_{\rho/8}(z)| &\leq | \{x\in \Omega\cap B_{\rho/8}(z)\,:\,M_{\rho/16}^{\Omega}(|u/d^s-v/d^s|)(x)>\lambda\} |\nonumber\\
     &\leq \frac{c}{\lambda} |B_\rho| \dashint_{\Omega\cap B_{\rho/4}(z)}|u/d^s-v/d^s|\,dx\leq c\epsilon_1 |B_\rho|
 \end{align*}
 which contradicts the assumption of the lemma for $\epsilon=c\epsilon_1$. This completes the proof.
\end{proof}

We next observe the Vitali type covering lemma (see \cite[Lemma 2.7]{ByuWan04}).
\begin{lemma}\label{lem.vit.sec4}
    Let $\epsilon>0$ and $\mathcal{C}\subset\mathcal{D}\subset\Omega\cap B_{1/4}$. Assume that
    \begin{equation*}
        |\mathcal{C}|\leq \epsilon
    \end{equation*}
    and that for any $z\in \Omega\cap B_{1/4}$ and $\rho\in(0,2^{-9}]$,
    \begin{equation*}
        B_{\rho/8}(z)\cap\Omega\subset\mathcal{D}
    \end{equation*}
    holds whenever $|\mathcal{C}\cap B_{\rho/8}(z)|\geq \epsilon|B_{\rho/8}|$. Then there is a constant $c=c(\mathsf{data})$ such that
    \begin{equation*}
        |\mathcal{C}|\leq c\epsilon|\mathcal{D}|.
    \end{equation*}
\end{lemma}

From \eqref{weak11.sec4}, Lemma \ref{lem.disc.measd} and Lemma \ref{lem.vit.sec4} together with the usage of an inductive argument, we can obtain the following. 
\begin{corollary}
 For $\epsilon>0$, there is a constant $\delta=\delta(\mathsf{data},\epsilon)$ such that if $A$ is $(\delta,1/2)$-vanishing, then for all $\lambda\geq \lambda_1$ and non-negative integer $k$, there holds
 \begin{equation}\label{ineq.cor.46}
 \begin{aligned}
 &\big|\{z\in \Omega\cap B_{1/4}\,:\,M_{2^{-6}}^{\Omega}(|u/d^s|)(z)>N^k\lambda\} \big| \nonumber\\
 &\leq (c\epsilon)^k \big| \{z\in \Omega\cap B_{1/4}\,:\,M_{2^{-6}}^{\Omega}(|u/d^s|)(z)>\lambda\} \big| \nonumber\\
 &+ \sum_{i=1}^{k} (c\epsilon)^i \big|\{z\in \Omega\cap B_{1/4}\,:\,M_{2_*s,2^{-6}}^{\Omega}(|f|^{2_*})(z)>(N^{k-i}\delta\lambda)^{2_*}\} \big|,
 \end{aligned}
 \end{equation}
 for some $c=c(\mathsf{data})$, where $N=N(\mathsf{data})\geq 1$ is given by Lemma \ref{lem.disc.measd}.
\end{corollary}

Using this, we are able to prove the following lemma.
\begin{lemma}\label{lem.vmo}
    Let $f\in L^q(\Omega\cap B_{1/2})$ for some $q\in(2_*,n/s)$. Then there is a constant $\delta=\delta(\mathsf{data},q)$ such that if $A$ is $(\delta,1/2)$-vanishing, then 
    \begin{equation*}
        \|u/d^s\|_{L^{\frac{nq}{n-sq}}(\Omega\cap B_{1/4})}\leq c\left(\|u\|_{L^1_{2s}(\bbR^n)}+\|f\|_{L^q(\Omega\cap B_{1/2})}\right)
    \end{equation*}
    for some constant $c=c(\mathsf{data},q)$.
\end{lemma}
\begin{proof}
Let us fix $N=N(\mathsf{data})\geq1$, which is determined in Lemma \ref{lem.disc.measd}. Then for any $\epsilon>0$, there is a constant $\delta=\delta(\mathsf{data},\epsilon)$ such that \eqref{ineq.cor.46} with $\lambda=\lambda_1$ holds, where the constant $\lambda_1$ is determined in \eqref{choi.lambda1}. There is a constant $\epsilon_1=\epsilon_1(\mathsf{data})\in(0,1)$ such that if $\epsilon\leq \epsilon_1$, then by taking summation over $k$ on both sides of \eqref{ineq.cor.46}, we get 
\begin{align}\label{ineq1.lem.vmo}
    &\sum_{i=1}^{l}(N^k\lambda_1)^{\frac{nq}{n-sq}}\big|\{z\in \Omega\cap B_{1/4}\,:\,M_{2^{-6}}^{\Omega}(|u/d^s|)(z)>N^k\lambda_1\} \big| \nonumber\\
 &\leq\sum_{k=1}^{l} \sum_{i=1}^{k} (N^{k-1}\lambda_1)^{\frac{nq}{n-sq}}(c\epsilon)^i \big|\{z\in \Omega\cap B_{1/4}\,:\,M_{2_*s,2^{-6}}^{\Omega}(|f|^{2_*})(z)>(N^{k-i}\delta\lambda_1)^{2_*}\} \big|\eqqcolon I,
\end{align}
where $c=c(\mathsf{data})$. Moreover, we further estimate the right-hand side of \eqref{ineq1.lem.vmo} as 
\begin{align*}
    I&\leq \sum_{i=1}^{l}(c\epsilon N^{\frac{nq}{n-sq}})^i\sum_{k=i}^{l}\left(N^{k-i}\lambda_1\right)^{\frac{nq}{n-sq}}\big|\{z\in \Omega\cap B_{1/4}\,:\,M_{2_*s,2^{-6}}^{\Omega}(|f|^{2_*})(z)>(N^{k-i}\delta\lambda_1)^{2_*}\} \big|\\
    &\leq c_0\delta^{-\frac{nq}{n-sq}}\sum_{i=1}^{l}(c\epsilon N^{\frac{nq}{n-sq}})^i\left(\|M_{2_*s,2^{-6}}^{\Omega}(|f|^{2_*})(z)\|^{2_*}_{L^{\frac{nq}{2_*(n-sq)}}(\Omega\cap B_{1/4})}+\lambda_1\right)
\end{align*}
for some constant $c=c(\mathsf{data})$ and $c_0=c_0(\mathsf{data},q)$, where we have also used \cite[Lemma 7.3]{CafCar95}. Thus we now take $\epsilon_2=\epsilon_2(\mathsf{data},q)$ such that if $\epsilon\leq \epsilon_2$, then
\begin{align*}
    &\sum_{i=1}^{l}(N^k\lambda_1)^{\frac{nq}{n-sq}}\big|\{z\in \Omega\cap B_{1/4}\,:\,M_{2^{-6}}^{\Omega}(|u/d^s|)(z)>N^k\lambda_1\} \big|\leq c\|f\|_{L^q(\Omega\cap B_{1/2})}+c\lambda_1,
\end{align*}
where $c=c(\mathsf{data},q)$. We now fix $\epsilon=\min\{\epsilon_1,\epsilon_2\}$ which depends only on $\mathsf{data}$ and $q$. Then the constant $\delta=\delta(\mathsf{data},q)$ is now chosen by Lemma \ref{lem.disc.measd}.  By \cite[Lemma 7.3]{CafCar95} and \eqref{choi.lambda1}, we get 
\begin{align}\label{ineq11.lem.vmo}
    \|u/d^s\|_{L^{\frac{nq}{n-sq}}(\Omega\cap B_{1/4})}\leq c\left(\|u/d^s\|_{L^1(\Omega\cap B_{1/2})}+\|u\|_{L^1_{2s}(\bbR^n)}+\|f\|_{L^q(\Omega\cap B_{1/2})}\right),
\end{align}
where we have used 
\begin{equation*}
    \mathrm{Tail}(u;B_{1/2})\leq c\|u\|_{L^1_{2s}(\bbR^n)}.
\end{equation*}
We are going to prove for any $z\in \Omega\cap B_{1/2}$,
\begin{align}\label{ineq2.lem.vmo}
    \|u/d^s\|_{L^1(\Omega\cap B_{2^{-6}}(z))}\leq c\left(\|u\|_{L^1_{2s}(\bbR^n)}+\|f\|_{L^{2_*}(\Omega\cap B_{2^{-3}}(z))}\right)
\end{align}
for some constant $c=c(\mathsf{data})$. 
Let $w$ be a weak solution to \eqref{eq.comp.w} with $r=2^{-4}$ and let $v$ be a weak solution to \eqref{eq.comp.v} with $r=2^{-4}$. Then we obtain
\begin{equation}\label{ineq3.thm.dini}
    \begin{aligned}
    \dashint_{\Omega\cap B_{2^{-6}}(z)}|{u}/d^s|\,dx&\leq c\left(\dashint_{B_{2^{-6}}(z)}|({u}-v)/d^s|\,dx+\dashint_{B_{2^{-6}}(z)}|v/d^s|\,dx\right)\\
    &\leq c\left([{u}-v]_{W^{s,2}(\bbR^n)}+(\|v\|_{L^1(B_{2^{-4}}(z))}+\mathrm{Tail}({u};B_{2^{-4}}(z)))\right),
    \end{aligned}
\end{equation}
where we have used Lemma \ref{lem.loc.const} . We now use Lemma \ref{comp1.conti} and Lemma \ref{comp2.conti} to see that 
\begin{equation}\label{ineq4.thm.dini}
\begin{aligned}
    [{u}-v]_{W^{s,2}(\bbR^n)}&\leq
    [{u}-w]_{W^{s,2}(\bbR^n)}+[w-v]_{W^{s,2}(\bbR^n)}\\
    &\leq c\left(\left(\dashint_{\Omega\cap B_{2^{-4}}(z)}|f|^{2_*}\,dx\right)^{\frac1{2_*}}+\|w\|_{L^1(B_{2^{-3}}(z))}+\mathrm{Tail}(w;B_{2^{-3}}(z))\right)
\end{aligned}
\end{equation}
for some constant $c=c(\mathsf{data})$. In addition, by an aid of Lemma \ref{comp1.conti}, we have 
\begin{equation}\label{ineq5.thm.dini}
\begin{aligned}
    &\|w\|_{L^1(B_{2^{-3}}(z))}+\mathrm{Tail}(w;B_{2^{-3}}(z))\\
    &\leq  c\left[\left(\dashint_{\Omega\cap B_{2^{-3}}(z)}|f|^{2_*}\,dx\right)^{\frac1{2_*}}+\|u\|_{L^1(B_{2^{-3}}(z))}+\mathrm{Tail}(u;B_{2^{-3}}(z))\right].
\end{aligned}
\end{equation}
Combining \eqref{ineq4.thm.dini} and \eqref{ineq5.thm.dini} yields 
\begin{align}\label{ineq6.thm.dini}
    [u-v]_{W^{s,2}(\bbR^n)}\leq c\left[\left(\dashint_{\Omega\cap B_{2^{-3}}(z)}|f|^{2_*}\,dx\right)^{\frac1{2_*}}+\|{u}\|_{L^1(B_{2^{-3}}(z))}+\mathrm{Tail}({u};B_{2^{-3}}(z))\right]
\end{align}
for some constant $c=c(\mathsf{data})$. We now employ Sobolev embedding as in \cite[Lemma 4.8]{Coz17} along with  \eqref{ineq4.thm.dini} and  \eqref{ineq5.thm.dini} to see that
\begin{align*}
    \|v\|_{L^1(B_{2^{-3}}(z))}&\leq \|v-w\|_{L^1(B_{2^{-3}}(z))}+\|w-u\|_{L^1(B_{2^{-3}}(z))}+\|{u}\|_{L^1(B_{2^{-3}}(z))}\\
    &\leq c\left([{v}-w]_{W^{s,2}(\bbR^n)}+[w-u]_{W^{s,2}(\bbR^n)}+\|{u}\|_{L^1(B_{2^{-3}}(z))}\right)\\
    &\leq c\left[\left(\dashint_{\Omega\cap B_{2^{-3}}(z)}|f|^{2_*}\,dx\right)^{\frac1{2_*}}+\|{u}\|_{L^1(B_{2^{-3}}(z))}+\mathrm{Tail}({u};B_{2^{-3}}(z))\right].
\end{align*}
Plugging this and \eqref{ineq6.thm.dini} into \eqref{ineq3.thm.dini} together with \eqref{ineq1.thm.dini} gives
\begin{equation*}
\begin{aligned}
    \dashint_{B_{2^{-6}}(z)}|{u}/d^s|\,dx&\leq c\left[\left(\dashint_{\Omega\cap B_{2^{-3}}(z)}|f|^{2_*}\,dx\right)^{\frac1{2_*}}+\|{u}\|_{L^1(B_{2^{-3}}(z))}+\mathrm{Tail}({u};B_{2^{-3}}(z))\right]
\end{aligned}
\end{equation*}
for some constant $c=c(\mathsf{data})$, which implies \eqref{ineq2.lem.vmo}. Therefore, using a standard covering argument, we have 
\begin{align}
    \|u/d^s\|_{L^1(\Omega\cap B_{1/2})}\leq c\left(\|u\|_{L^1_{2s}(\bbR^n)}+\|f\|_{L^{2_*}(\Omega\cap B_{1/2})}\right).
\end{align}
Plugging this into \eqref{ineq11.lem.vmo} yields the desired estimate.
\end{proof}

\subsection{Proof of Theorem \ref{thm.vmo}}
We are now in a position to present the proof of Theorem \ref{thm.vmo}.
\begin{proof}[Proof of Theorem \ref{thm.vmo}.] 
By the scaling invariant property, we may assume $\rho_0=1$. We now fix $\delta=\delta(\mathsf{data},q)$ as given by Lemma \ref{lem.vmo}. In addition, by Remark \ref{rmk.loc.sec4}, we have that there is a weak solution $\widetilde{u}$ to \eqref{eq.sec4} with $(\delta,1)$-vanishing coefficient $A$. By Lemma \ref{lem.vmo}, we obtain 
\begin{align*}
   \|{u}/d^s\|_{L^{\frac{nq}{n-sq}}(\Omega\cap B_{1/4})} =\|\widetilde{u}/d^s\|_{L^{\frac{nq}{n-sq}}(\Omega\cap B_{1/4})}
   &\leq c\left(\|\widetilde{u}\|_{L^1_{2s}(\bbR^n)}+c\|f+g_2\|_{L^{q}(\Omega\cap B_1)}\right)\\
   &\leq c\left(\|{u}\|_{L^1_{2s}(\bbR^n)}+\|f\|_{L^{q}(\Omega\cap B_1)}\right)
\end{align*}
for some constant $c=c(\mathsf{data},q)$, where we have also used \eqref{ineq.g2.sec4} and the fact that $\widetilde{u}\equiv u$ on $B_{3/4}$ and $|\widetilde{u}(x)|\leq |u(x)|$ for any $x\in\bbR^n$. This completes the proof.
\end{proof}

\section{The case of continuous coefficients}\label{sec5}
In this section, we consider a weak solution $u\in W^{s,2}(B_1)\cap L^1_{2s}(\bbR^n)$ to \eqref{eq: defn}, where the associated kernel coefficient $A$ is continuous in $B_1\times B_1$ and there is a non-decreasing function $\omega:\bbR_+\to\bbR_+$ such that
\begin{equation*}
    |A(x_0,y_0)-A(x_1,y_1)|\leq \omega(\max\{|x_0-x_1|,|y_0-y_1|\})\quad \text{for any }x_0,x_1,y_0,y_1\in B_1.
\end{equation*}

We now fix $z\in \overline{\Omega}\cap B_{1/2}$. Then there is a function $\oldphi\equiv \oldphi_z$ which is the solution to
\begin{equation}\label{eq.barrier}
\left\{
\begin{alignedat}{3}
\mathcal{L}_{A_z}{\oldphi}&= 0&&\qquad \mbox{in  $\Omega \cap B_{1}$}, \\
\oldphi&=g&&\qquad  \mbox{in $\bbR^n\setminus (\Omega \cap B_{1})$},
\end{alignedat} \right.
\end{equation}
where $A_z\coloneqq A(z,z)$ and $g\in C_c^\infty(\bbR^n\setminus {B}_{3/2})$ with $g\equiv 1$ on $B_2\setminus B_{7/4}$ and $0\leq g\leq 1$.
By \cite[Lemma 2.3.9, Proposition 2.6.4, Proposition 2.6.6]{FerRos24}, the maximum principle given in \cite[Lemma 2.3.3]{FerRos24} and \cite[Proposition 2.7.8]{FerRos24}, we observe 
\begin{align}\label{rel.phid1}
{d^s(x)}/{c}\leq\oldphi(x)\leq c d^s(x)\quad\text{if }x\in B_{3/4},
\end{align}
\begin{equation}\label{bdd.phid1}
    \|\oldphi\|_{L^\infty(\bbR^n)}\leq c
\end{equation}
and 
\begin{align}\label{hol.phid1}
    \|\oldphi/d^s\|_{C^{\alpha}(\Omega\cap B_{3/4})}\leq c,
\end{align}
where $c=c(\mathsf{data})$.
By Lemma \ref{lem.localization}, we have the following remark, which enables us to consider a localized problem instead, where the associated kernel coefficient is constant when it is sufficiently far from the given domain.
\begin{remark}\label{rmk.sec5}
We observe that $u\in W^{s,2}(B_{1/8}(z))\cap L^1_{2s}(\bbR^n)$ is a weak solution to 
\begin{equation*}
\left\{
\begin{alignedat}{3}
\mathcal{L}_Au&= f&&\qquad \mbox{in  $\Omega \cap B_{1/8}(z)$}, \\
u&=0&&\qquad  \mbox{in $B_{1/8}(z)\setminus \Omega$}.
\end{alignedat} \right.
\end{equation*}
By Lemma \ref{lem.localization} with $R=1/8$, $x_0=z$ and $a=A(z,z)$, there is a weak solution $\widetilde{u}\in W^{s,2}(\bbR^n)$ to 
\begin{equation}
\label{eq: sec5.loc}
\left\{
\begin{alignedat}{3}
\mathcal{L}_{\widetilde{A}}\widetilde{u}&= F&&\qquad \mbox{in  $\Omega \cap B_{1/16}(z)$}, \\
\widetilde{u}&=0&&\qquad  \mbox{in $B_{1/16}(z)\setminus \Omega$},\\
\widetilde{u}&=u&&\qquad  \mbox{in $B_{5/64}(z)$},
\end{alignedat} \right.
\end{equation}
where $\widetilde{u}\equiv 0$ in $\bbR^n\setminus B_{1/8}(z)$ and
\begin{equation}\label{sec5.F}
    |F(x)|\leq |f(x)|+c\left(|u(x)|+\mathrm{Tail}(u;B_{5/64}(z))\right)\quad\text{for any }x\in B_{1/16}(z),
\end{equation}
and some $c=c(n,s,\Lambda)$. In addition, there is a constant $c=c(n,s,\Lambda)$ such that
\begin{equation*}
    |\widetilde{A}(x,y)-\widetilde{A}(z,z)|\leq {\widetilde{\omega}}(\max\{|x-z|,|y-z|\})\quad\text{for any }x,y\in \bbR^n,
\end{equation*}
where
\begin{align*}
    {\widetilde{\omega}}(\rho)=\begin{cases}
        \omega(\rho)\quad&\text{if }\rho\leq 1/8,\\
        0&\text{if }\rho> 1/8.
    \end{cases} 
\end{align*}
\end{remark}

\subsection{Localized problems}
Throughout this subsection, in light of the Remark \ref{rmk.sec5}, we consider a weak solution $u\in W^{s,2}(\bbR^n)$ to 
\begin{equation}\label{eq.sec5.loc}
\left\{
\begin{alignedat}{3}
\mathcal{L}_Au&= f&&\qquad \mbox{in  $\Omega \cap B_{1/16}(z)$}, \\
u&=0&&\qquad  \mbox{in $B_{1/8}(z)\setminus \Omega$}.
\end{alignedat} \right.
\end{equation}
In addition, the kernel coefficient $A$ satisfies
\begin{equation}\label{ker.cond.sec5.loc}
    |{A}(x,y)-{A}(z,z)|\leq {\pmb{\omega}}(\max\{|x-z|,|y-z|\})\quad\text{for any }x,y\in \bbR^n,
\end{equation}
where
\begin{equation}\label{defn.omega.sec5}
\begin{aligned}
    {\pmb{\omega}}(\rho)=\begin{cases}
        \omega(\rho)\quad&\text{if }\rho\leq 1/8,\\
        0&\text{if }\rho> 1/8
    \end{cases} 
\end{aligned}
\end{equation}
for some non-decreasing function $\omega:\bbR^+\to\bbR^+$.
To get a control on the global behavior of the oscillation of the kernel coefficient, we define a function ${\pmb{\omega}}_G:\bbR_+\to \bbR_+$ as
\begin{equation*}
    {\pmb{\omega}}_{G}(\rho)\coloneqq \rho^s\int_{\rho}^\infty\frac{{\pmb{\omega}}(\xi)}{\xi^{1+s}}\,d\xi \quad\text{for all }\rho>0.
\end{equation*}

We now define the following excess functionals 
\begin{equation*}
    E_{\mathrm{loc}}(u;B_{\rho}(z))\coloneqq \dashint_{B_{\rho}(z)}|u/\oldphi-(u/\oldphi)_{B_{\rho}(z)}|\,dx
\end{equation*}
and
\begin{align*}
    {E}(u;B_{\rho}(z))&\coloneqq E_{\mathrm{loc}}(u;B_{\rho}(z))+\max\{\oldphi(z),\rho^s\}^{-1}\mathrm{Tail}(u-(u/\oldphi)_{B_{\rho}(z)}\oldphi;B_{\rho}(z)).
\end{align*}
We point out that such excess functionals were first introduced in \cite{KimWei24}.
We then observe the following inequality.
\begin{lemma}\label{lem.sim.exc}
    There is a constant $c=c(n,s)$ such that for any $0<\rho<r\leq1/16$,
    \begin{equation*}
    {E}(u;B_{r}(z))\leq c {(r/\rho)^n}{E}(u;B_{r}(z)).
\end{equation*}
\end{lemma}
\begin{proof}
We first observe that
\begin{align*}
    &\max\{\oldphi(z),\rho^s\}^{-1}\mathrm{Tail}(u-(u/\oldphi)_{B_{\rho}(z)}\oldphi;B_{\rho}(z))\\
    &\leq\max\{\oldphi(z),\rho^s\}^{-1}\rho^{2s}\int_{\bbR^n\setminus B_{\rho}(z)}\frac{|u-(u/\oldphi)_{B_{r}(z)}\oldphi|}{|y-z|^{n+2s}}\,dy\\
    &\quad+\max\{\oldphi(z),\rho^s\}^{-1}\rho^{2s}\int_{\bbR^n\setminus B_{\rho}(z)}\frac{|((u/\oldphi)_{B_{r}(z)}-(u/\oldphi)_{B_{\rho}(z)})\oldphi|}{|y-z|^{n+2s}}\,dy\coloneqq I+J.
\end{align*}
Furthermore, we have
\begin{align*}
    I&\leq\max\{\oldphi(z),\rho^s\}^{-1}\left[\rho^{2s}\int_{B_r(z)\setminus B_{\rho}(z)}\frac{|u-(u/\oldphi)_{B_{r}(z)}\oldphi|}{|y-z|^{n+2s}}\,dy+\rho^{2s}\int_{\bbR^n\setminus B_{r}(z)}\frac{|u-(u/\oldphi)_{B_{r}(z)}\oldphi|}{|y-z|^{n+2s}}\,dy\right].
\end{align*}
We note from \eqref{rel.phid1} that 
\begin{equation}\label{ineq.oldphidk}
       |\oldphi(y)|\leq c|d(y)|^s\leq c(d(z)+2^kr)^{s}\quad\text{for any }y\in B_{2^kr}(z),
   \end{equation}
for all non-negative integer $k$ such that $B_{2^kr}(z)\subset B_{3/4}$. In addition, there is a nonnegative integer $j$ such that $2^j\rho<r<2^{j+1}\rho$. 
Thus, we get 
\begin{align*}
I&\leq c\left(\sum_{i=0}^{j-1}2^{-is}\dashint_{B_{2^{i}\rho}(z)}|u/\oldphi-(u/\oldphi)_{B_{r}(z)}|\,dy+E(u/B_r(z))\right)\leq c(r/\rho)^nE(u;B_r(z))
\end{align*}
for some constant $c=c(\mathsf{data})$.  Let us denote $l$ as a positive integer such that $1/8< 2^l\rho\leq 1/4$.
We now estimate $J$ as 
\begin{align*}
    J&\leq \max\{\oldphi(z),\rho^s\}^{-1}\rho^{2s}\sum_{i=0}^{l-1}\int_{B_{2^{i+1}\rho}(z)\setminus B_{2^ir}(z)}\frac{|((u/\oldphi)_{B_{\rho}(z)}-(u/\oldphi)_{B_{r}(z)})\oldphi|}{|y-z|^{n+2s}}\,dy\\
    &\quad+ \max\{\oldphi(z),\rho^s\}^{-1}\rho^{2s}\int_{\bbR^n\setminus B_{2^l\rho}(z)}\frac{|((u/\oldphi)_{B_{\rho}(z)}-(u/\oldphi)_{B_{r}(z)})\oldphi|}{|y-z|^{n+2s}}\,dy\\
    &\leq c\max\{\oldphi(z),\rho^s\}^{-1}\rho^{2s}\sum_{i=0}^{l-1}(2^i\rho)^{-2s}\max\{d(z),(2^i\rho)\}^s|(u/\oldphi)_{B_{\rho}(z)}-(u/\oldphi)_{B_{r}(z)}|\\
    &\quad +c  |(u/\oldphi)_{B_{\rho}(z)}-(u/\oldphi)_{B_{r}(z)}|\\
    &\leq c|(u/\oldphi)_{B_{\rho}(z)}-(u/\oldphi)_{B_{r}(z)}|
\end{align*}
for some constant $c=c(\mathsf{data})$, where we have used \eqref{ineq.oldphidk} with $r$ replaced by $\rho$ and \eqref{bdd.phid1}. Thus, we get
\begin{align*}
    J\leq c(r/\rho)^nE_{\mathrm{loc}}(u;B_r(z)).
\end{align*}
Combining the estimates of $I$ and $J$ together with 
\begin{align*}
    E_{\mathrm{loc}}(u;B_\rho(z))\leq c(r/\rho)^nE_{\mathrm{loc}}(u;B_r(z)),
\end{align*} we have the desired estimate.
\end{proof}

In light of the comparison estimates given in Section \ref{sec.comp} and \cite[Lemma 6.1]{KimWei24}, we now prove the following.
\begin{lemma}\label{lem.comp.conti}
    Let $r\in(0,1/64]$  and let $v$ be the weak solution to \eqref{eq.comp.v} with $a_0\equiv A_z=A(z,z)$. Then for any $\rho\in(0,1/4]$, 
\begin{equation}\label{hig.fir.ineq}
   {E}_{\mathrm{loc}}(v;B_{\rho r}(z))\leq c\rho^{\alpha}{E}(v;B_r(z))
\end{equation}
holds for some constant $c=c(\mathsf{data})$. In addition, we have 
\begin{equation}\label{comp2.ineq}
\begin{aligned}
    {E}_{\mathrm{loc}}(u-v;B_{r}(z))&\leq c({\pmb{\omega}}(2r)+{\pmb{\omega}}_G(2r))({E}(u;B_{2r}(z))+|(u/\oldphi)_{B_{2r}(z)}|)\\
    &\quad+c\max\{\oldphi(z),r^s\}^{-1}{\mathrm{Tail}}(u(\cdot){\pmb{\omega}}(|\cdot-z|);B_{2r}(z))\\
    &\quad+ c\left(\dashint_{\Omega\cap B_{2r}(z)}(r^s|f|)^{2_*}\,dx\right)^{\frac1{2_*}}.
\end{aligned}
\end{equation}
\end{lemma}
\begin{proof}
We first note that \eqref{hig.fir.ineq} follows from \cite[Lemma 6.6]{KimWei24} and \eqref{rmk.bdry}. 
We are now going to prove \eqref{comp2.ineq}. Let $w$ be the weak solution to \eqref{eq.comp.w}.
    By Lemma \ref{comp1.conti} and Lemma \ref{comp2.conti}, we get 
    \begin{equation}\label{ineq1.comp.conti}
    \begin{aligned}
        r^{-n}[u-v]^2_{W^{s,2}(\bbR^n)}
        &\leq c\left(\dashint_{\Omega\cap B_{2r}(z)}(r^{s}|f|)^{2_*}\,dx\right)^{\frac2{2_*}}\\
        &\quad+c({\pmb{\omega}}(2r)+{\pmb{\omega}}_G(2r))^2(r^{-2n}\|w/r^s\|^2_{L^{1}(B_{2r}(z))}+r^{-2s}\mathrm{Tail}(u;B_{2r}(z))^2)\\
        &\quad+ cr^{2s}\left(\sup_{x\in B_r(z)}\int_{\bbR^n\setminus B_{2r}(z)}\frac{|w(y)||A-A(z,z)|}{|y-z|^{n+2s}}\,dy\right)^2
    \end{aligned}
    \end{equation}
   for some constant $c=c(\mathsf{data})$, where we have also used the following simple observation
   \begin{align*}
       r^{4s}\left(\sup_{x\in B_r(z)}\int_{\bbR^n\setminus B_{2r}(z)}\frac{|A-A_z|}{|y-z|^{n+2s}}\,dy\right)^2\leq r^{4s}\left(\int_{2r}^{\infty}\frac{{\pmb{\omega}}(\rho)}{\rho^{1+2s}}\,d\rho\right)^2&\leq r^{2s}\left(\int_{2r}^{\infty}\frac{{\pmb{\omega}}(\rho)}{\rho^{1+s}}\,d\rho\right)^2.
   \end{align*}
   Moreover, by H\"older's inequality, Sobolev's inequality and Lemma \ref{comp1.conti}, we have
   \begin{align*}
       r^{-2n}\|w/r^s\|^2_{L^1(B_{2r}(z))}&\leq r^{-2n}\|(w-u)/r^s\|^2_{L^1(B_{2r}(z))}+r^{-2n}\|u/r^s\|^2_{L^1(B_{2r}(z))}\\
       &\leq cr^{-n}[w-u]^2_{W^{s,2}(\bbR^n)}+ r^{-2n}\|u/r^s\|^2_{L^1(B_{2r}(z))}\\
       &\leq c\left[\left(\dashint_{\Omega\cap B_{2r}(z)}(r^{s}|f|)^{2_*}\,dx\right)^{\frac2{2_*}}+r^{-2n}\|u/r^s\|^2_{L^1(B_{2r}(z))}\right],
   \end{align*}
   where $c=c(\mathsf{data})$. Additionally, by \eqref{ineq.simple} and \eqref{rel.phid1}, we get
   \begin{align*}
       r^{-n}\|u/r^s\|_{L^1(B_{2r}(z))}&\leq c\max\{1,d(z)/r\}^{s}r^{-n}\|u/d^s\|_{L^1(B_{2r}(z))}\\
       &\leq c\max\{1,d(z)/r\}^{s}r^{-n}\|u/\oldphi^s\|_{L^1(B_{2r}(z))}\\
       &\leq  c\max\{1,d(z)/r\}^{s}(E_{\mathrm{loc}}(u;B_{2r}(z))+|(u/\oldphi)_{B_{2r}(z)}|)
   \end{align*}
   for some constant $c=c(\mathsf{data})$, which gives
   \begin{align*}
       r^{-2n}\|w/r^s\|^2_{L^1(B_{2r}(z))}&\leq c\max\{1,d(z)/r\}^{s}(E_{\mathrm{loc}}(u;B_{2r}(z))+|(u/\oldphi)_{B_{2r}(z)}|)\\
  &\quad+c\left(\dashint_{\Omega\cap B_{2r}(z)}(r^{s}|f|)^{2_*}\,dx\right)^{\frac2{2_*}}
   \end{align*}
   for some constant $c=c(\mathsf{data})$.
   Plugging this into \eqref{ineq1.comp.conti} yields 
   \begin{equation}\label{ineq2.comp.conti}
   \begin{aligned}
       r^{-n}[u-v]^2_{W^{s,2}(\bbR^n)}
        &\leq c\left(\dashint_{\Omega\cap B_{2r}(z)}(r^{s}|f|)^{2_*}\,dx\right)^{\frac2{2_*}}\\
        &\hspace{0.2pt}+c({\pmb{\omega}}(2r)+{\pmb{\omega}}_G(2r))^2\max\{1,d(z)/r\}^{2s}(E(u;B_{2r}(z))+|(u/\oldphi)_{B_{2r}(z)}|)^2\\
        &+c({\pmb{\omega}}(2r)+{\pmb{\omega}}_G(2r))^2r^{-2s}\mathrm{Tail}(u;B_{2r}(z))^2\\
        &+ cr^{-2s}{\mathrm{Tail}}(u(\cdot){\pmb{\omega}}(|\cdot-z|);B_{2r}(z))^2.
   \end{aligned}
   \end{equation}
  To further simplify the third term on the right-hand side of the above expression, we observe that
   \begin{align*}
     &r^{-s}\mathrm{Tail}(u;B_{2r}(z)) \leq r^{-s}\mathrm{Tail}(u-(u/\oldphi)_{B_{2r}(z)}\oldphi;B_{2r}(z)) + r^s\int_{\bbR^n\setminus B_{2r}(z)}\frac{|(u/\oldphi)_{B_{2r}(z)}\oldphi(y)|}{|y-z|^{n+2s}}\,dy  \nonumber\\
     &\leq  c\max\{1,d(z)/r\}^{s}E(u;B_{2r}(z))+  c\underbrace{\sum_{k=0}^{l}(2^{2k}r)^{-s}\dashint_{B_{2^kr}(z)}| (u/\oldphi)_{B_{2r}(z)}\oldphi|\,dy}_{{=:I_1}}\\
       &\quad+\underbrace{r^s\int_{\bbR^n\setminus B_{2^lr}(z)}\frac{|(u/\oldphi)_{B_{2r}(z)}\oldphi(y)|}{|y-z|^{n+2s}}\,dy}_{\eqqcolon I_2}
   \end{align*}
   for some constant $c=c(n,s)$, where the positive integer $l$ satisfies $1/32\leq 2^lr<1/16$ and $B_{2^lr}(z)\subset B_{3/8}$.
   Using \eqref{ineq.oldphidk} and \eqref{rel.phid1}, we estimate $I_1$ as 
   \begin{align*}
       I_1\leq c\sum_{k=0}^{l}(2^{2k}r)^{-s}| (u/\oldphi)_{B_{2r}(z)}|(d(z)+2^kr)^{s}&\leq c(d^s(z)/r^s+1)| (u/\oldphi)_{B_{2r}(z)}|\\
       &\leq c\max\{1,d(z)/r\}^{s}| (u/\oldphi)_{B_{2r}(z)}|
   \end{align*}
   for some constant $c=c(\mathsf{data})$. We next estimate $I_2$ as 
   \begin{align*}
       I_2\leq cr^s\int_{\bbR^n\setminus B_{2^lr}(z)}\frac{|(u/\oldphi)_{B_{2r}(z)}|}{|y-z|^{n+2s}}\,dy\leq c| (u/\oldphi)_{B_{2r}(z)}|,
   \end{align*}
   where we have used \eqref{bdd.phid1} and the fact that $B_{1/8}(z)\subset B_{2^l\rho}(z)$. Plugging the estimates $I_1$ and $I_2$ into \eqref{ineq2.comp.conti} gives us
   \begin{equation}\label{ineq3.comp.conti}
   \begin{aligned}
       r^{-n}[u-v]^2_{W^{s,2}(\bbR^n)}
        &\leq c\left(\dashint_{\Omega\cap B_{2r}(z)}(r^{s}|f|)^{2_*}\,dx\right)^{\frac2{2_*}}\\
        &\quad+c({\pmb{\omega}}(2r)+{\pmb{\omega}}_G(2r))^2\max\{1,d(z)/r\}^{2s}(E(u;B_{2r}(z))+|(u/\oldphi)_{B_{2r}(z)}|)^2\\
        &\quad+ cr^{-2s}{\mathrm{Tail}}(u(\cdot){\pmb{\omega}}(|\cdot-z|);B_{2r}(z))^2.
   \end{aligned}
   \end{equation}
   To estimate the left-hand side of \eqref{ineq3.comp.conti},
   we use \eqref{rel.phid1} and Lemma \ref{lem.sobpoi.ds} to get 
   \begin{align*}
       \left(\dashint_{\Omega\cap B_{r}(z)}|u/\oldphi-v/\oldphi|\,dx\right)^2\leq c \dashint_{\Omega\cap B_{r}(z)}|u/d^s-v/d^s|^2\leq cr^{-n}\frac{1}{\max\{1,d(z)/r\}^{2s}}[u-v]_{W^{s,2}(\bbR^n)}^2
   \end{align*}
   for some constant $c=c(\mathsf{data})$. 
 Combining this and \eqref{ineq3.comp.conti} together with \eqref{rel.phid1}, we have \eqref{comp2.ineq}.
\end{proof}

We now prove the following decay estimates.
\begin{lemma}\label{lem.decay}
   Let $r\in(0,1/64]$. Then, for any $\rho\in(0,1/4]$, we have 
    \begin{align*}
    {E}(u;B_{\rho r}(z))&\leq c\rho^\alpha{E}(u;B_{r}(z))\\
         &\quad+ c\rho^{-n}({\pmb{\omega}}(r)+{\pmb{\omega}}_G(r))({E}(u;B_{r}(z))+|(u/\oldphi)_{B_{r}(z)}|)\\
         &\quad+ c\rho^{-n}\max\{\oldphi(z),r^s\}^{-1}{\mathrm{Tail}}(u(\cdot){\pmb{\omega}}(|\cdot-z|);B_{r}(z))\\
         &\quad+c\rho^{-n}\left(\dashint_{\Omega\cap B_{r}(z)}(r^s|f|)^{2_*}\,dx\right)^{\frac1{2_*}},
\end{align*}
    for some constant $c=c(\mathsf{data})$.
\end{lemma}
\begin{proof}
 If $\rho\in[2^{-10},2^{-3}]$, then the desired estimate follows directly from Lemma \ref{lem.sim.exc}. Thus, we now assume $\rho\in(0,2^{-10})$. Then there is a positive integer $N_\rho$ such that 
    \begin{equation}\label{cond.nrho}
        2^{-7}\leq 2^{N_\rho}\rho<2^{-5 }.
    \end{equation}
     We note that
     \begin{align*}
         {E}_{\mathrm{loc}}(u;B_{\rho r}(z))\leq {E}_{\mathrm{loc}}(v;B_{\rho r}(z))+{E}_{\mathrm{loc}}(u-v;B_{\rho r}(z))\eqqcolon I_1+I_2,
     \end{align*}
     where $v$ is the weak solution to \eqref{eq.comp.w} with $a_0$ and $r$ replaced by $A(z,z)$ and $r/2$, respectively.
     In light of \eqref{hig.fir.ineq} in Lemma \ref{lem.comp.conti} with $r$ replaced by $r/2$, we first estimate $I_1$ as 
     \begin{align*}
         I_1&\leq c\rho^{\alpha}{E}(v;B_{r/2}(z))\leq c\rho^\alpha{E}(u;B_{r/2}(z))+c\rho^{\alpha}E_{\mathrm{loc}}(v-u;B_{r/2}(z))
     \end{align*}
     for some constant $c=c(\mathsf{data})$.
     By \eqref{comp2.ineq} in Lemma \ref{lem.comp.conti} with $r$ replaced by $r/2$, we next estimate $I_2$ as 
     \begin{align*}
         I_2&\leq c\rho^{-n}({\pmb{\omega}}(r)+{\pmb{\omega}}_G(r))({E}(u;B_{r}(z))+|(u/\oldphi)_{B_{r}(z)}|)\\
    &\quad+c\rho^{-n}{\max\{\oldphi(z),r^s\}^{-1}{\mathrm{Tail}}(u(\cdot){\pmb{\omega}}(|\cdot-z|);B_{r}(z))}+ c\rho^{-n}\left(\dashint_{\Omega\cap B_{r}(z)}(r^s|f|)^{2_*}\,dx\right)^{\frac1{2_*}},
     \end{align*}
    where $c=c(\mathsf{data})$.
     Combining all the estimates $I_1$ and $I_2$ yields 
     \begin{equation}\label{ineq1.decay}
     \begin{aligned}
         {E}_{\mathrm{loc}}(u;B_{\rho r}(z))&\leq c\rho^\alpha {E}(u;B_{r}(z))\\
         &\quad+ c\rho^{-n}({\pmb{\omega}}(r)+{\pmb{\omega}}_G(r))({E}(u;B_{r}(z))+|(u/\oldphi)_{B_{r}(z)}|)\\
         &\quad+ c\rho^{-n}\max\{\oldphi(z),r^s\}^{-1}{\mathrm{Tail}}(u(\cdot){\pmb{\omega}}(|\cdot-z|);B_{r}(z))\\
         &\quad+c\rho^{-n}\left(\dashint_{\Omega\cap B_{r}(z)}(r^s|f|)^{2_*}\,dx\right)^{\frac1{2_*}}\eqqcolon I,
     \end{aligned}
     \end{equation}
     where $c=c(\mathsf{data})$. Using \cite[Lemma 2.2]{DieKimLeeNow24n}, we observe that
     \begin{equation}\label{ineq.tail.decay}
     \begin{aligned}
         &\max\{\oldphi(z),(\rho r)^s\}^{-1}{\mathrm{Tail}}(u(\cdot)-(u/\oldphi)_{B_{\rho r}(z)}\oldphi;B_{\rho r}(z))\\
         &\leq c\max\{\oldphi(z),(\rho r)^s\}^{-1}\sum_{k=1}^{N_\rho}2^{-2sk}\dashint_{B_{2^k\rho r}(z)}|u-(u/\oldphi)_{B_{\rho r}(z)}\oldphi|\,dx\\
         &\quad+ c\max\{\oldphi(z),(\rho r)^s\}^{-1}(\rho r)^{2s}\int_{\bbR^n\setminus B_{2^{N_\rho}\rho r}(z)}\frac{|u-(u/\oldphi)_{B_{\rho r}(z)}\oldphi|}{|y-z|^{n+2s}}\,dy\eqqcolon J_1+J_2.
     \end{aligned}
     \end{equation}
     We next note from \eqref{ineq.oldphidk} with $r$ and $z$ replaced by $\rho r$ and $0$, respectively, and \eqref{rel.phid1} that 
     \begin{equation}\label{ineq2.decay}
         |\oldphi(y)|\leq c(\oldphi(z)+(2^k\rho r)^s)\leq c\max\{\oldphi(z),(2^k\rho r)^s\}\quad\text{for any }y\in B_{2^k\rho r}(z)\text{ with }k=0,1,\cdots, N_\rho
     \end{equation}
     for some constant $c=c(\mathsf{data})$.
    Using this, we now estimate $J_1$ as 
     \begin{align*}
         J_1\leq c\sum_{k=1}^{N_\rho} 2^{-ks}\dashint_{B_{2^k\rho r}(z)}|u/\oldphi-(u/\oldphi)_{B_{\rho r}(z)}|\,dx&\leq c\sum_{k=1}^{N_\rho} 2^{-ks}\sum_{j=1}^{k}\dashint_{B_{2^j\rho r}(z)}|u/\oldphi-(u/\oldphi)_{B_{2^j\rho r}(z)}|\,dx\\
         &\leq c\sum_{j=1}^{N_\rho} 2^{-js}\dashint_{B_{2^j\rho r}(z)}|u/\oldphi-(u/\oldphi)_{B_{2^j\rho r}(z)}|\,dx,
     \end{align*}
where we have also used Fubini's theorem and the fact that 
\begin{align*}
    \sum_{k=j}^{N_\rho}2^{-ks}\leq c2^{-js}.
\end{align*}
After a few simple calculations along with \eqref{cond.nrho}, \eqref{ineq2.decay} and the fact that $\alpha<s$, we get
     \begin{align*}
         J_2& \leq \rho^{2s}\max\{\oldphi(z),(\rho r)^s\}^{-1}r^{2s}\int_{\bbR^n\setminus B_{2^{N_\rho}\rho r}(z)}\frac{|u-(u/\oldphi)_{B_{\rho r}(z)}\oldphi|}{|y-z|^{n+2s}}\,dy\\
         &\leq c\sum_{k=1}^{N_\rho} 2^{-N_\rho s}\dashint_{B_{2^k\rho r}(z)}|u/\oldphi-(u/\oldphi)_{B_{2^k\rho r}(z)}|\,dx\\
         &\quad+c\left[\rho^{s}\dashint_{B_{r}(z)}|u/\oldphi-(u/\oldphi)_{B_{r}(z)}|\,dx+ \rho^{s}\max\{\oldphi(z),r^s\}^{-1}\mathrm{Tail}(u-(u/\oldphi)_{B_{r}(z)}\oldphi;B_{r}(z))\right]\\
         &\leq c\left[\sum_{k=1}^{N_\rho} 2^{-ks}\dashint_{B_{2^k\rho r}(z)}|u/\oldphi-(u/\oldphi)_{B_{2^k\rho r}(z)}|\,dx+\rho^{s}{E}(u;B_{r}(z))\right]
     \end{align*}
     for some constant $c=c(\mathsf{data})$.
Combining all the estimates $J_1$ and $J_2$, we have 
\begin{align}\label{ineq3.decay}
    J_1+J_2&\leq c\left[\sum_{k=1}^{N_\rho} 2^{-ks}{E}_{\mathrm{loc}}(u;B_{2^k\rho r}(z))+\rho^{s}{E}(u;B_{r}(z))\right].
\end{align}
We now use \eqref{ineq1.decay} and the fact that $\alpha<s$ to find 
\begin{align*}
    \sum_{k=1}^{N_\rho} 2^{-ks}{E}_{\mathrm{loc}}(u;B_{2^k\rho r}(z))
    &\leq \sum_{k=1}^{N_\rho}c2^{-ks}(2^k\rho)^{\alpha } {E}(u;B_{r}(z))\\
         &\quad+ \sum_{k=1}^{N_\rho}c2^{-ks}(2^k\rho)^{-n}({\pmb{\omega}}(r)+{\pmb{\omega}}_G(r))({E}(u;B_{r}(z))+|(u/\oldphi)_{B_{r}(z)}|)\\
         &\quad+ \sum_{k=1}^{N_\rho}c2^{-ks}(2^k\rho)^{-n}\max\{\oldphi(z),r^s\}^{-1}{\mathrm{Tail}}(u(\cdot){\pmb{\omega}}(|\cdot-z|);B_{r}(z))\\
         &\quad+\sum_{k=1}^{N_\rho}c2^{-ks}(2^k\rho)^{-n}\left(\dashint_{\Omega\cap B_{r}(z)}(r^s|f|)^{2_*}\,dx\right)^{\frac{1}{2_*}}\\
         &\leq cI,
\end{align*}
where  $I$ is as determined in \eqref{ineq1.decay}. By plugging this into \eqref{ineq3.decay} and taking into account \eqref{ineq.tail.decay}, we get
\begin{align*}
    \max\{\oldphi(z),(\rho r)^s\}^{-1}{\mathrm{Tail}}((u(\cdot)-(u/\oldphi)_{B_{\rho r}(z)}\oldphi;B_{\rho r}(z))\leq cI+c\rho^{s}{E}(u;B_{r}(z))
\end{align*}
for some constant $c=c(\mathsf{data})$. This coupled with \eqref{ineq1.decay} gives the desired estimate.
\end{proof}

We are now able to get a uniform control on the averages of $u/\oldphi$ on small balls by the Wolff-potential of $|f|^{2_*}\mbox{\Large$\chi$}_{\Omega}$.
\begin{lemma}\label{lem.dini}
Let us assume $A$ satisfies
\begin{equation}\label{ass.lem.dini}
    \int_{0}^{1/8}\frac{{{\omega}}(\rho)}{\rho}\,d\rho\leq N
\end{equation}
for some constant $N>0$, where $\omega$ is the non-decreasing function given in \eqref{defn.omega.sec5}. Then there is a constant $m=m(\mathsf{data})\geq1$  such that for any $r\in(0,1/64]$ and $i\geq1$, we have
\begin{align}\label{ineq.lem.dini}
    \dashint_{B_{2^{-im}r}(z)}|u/\oldphi|\,dx\leq c\left({E}(u;B_r(z))+|(u/\oldphi)_{B_r(z)}|+{\mathrm{Tail}}(u/r^s;B_{r}(z))+\mathbf{W}^{|f|^{2_*}\mbox{\Large$\chi$}_{\Omega}}_{\frac{2_* s}{2_*+1},2_*+1}(z,r)\right),
\end{align}
where $c=c(\mathsf{data},N,
{{\omega}})$.
\end{lemma}
\begin{proof}
We first fix $m\geq1$ and $R_0\in(0,1/32]$ which are determined later in \eqref{cond.m} and \eqref{dini.cho.r0}, respectively. We then select $r\in(0,R_0]$.  By applying Lemma \ref{lem.decay} with $r$ and $\rho$ replaced by $2^{-km}r$ and $2^{-m}$, respectively,  to each $k\in \{0,\dots,l\}$, for any $l\geq 1$ fixed, we have 
\begin{align*}
  &\sum_{k=0}^{l}  E(u;B_{2^{-m(k+1)}r}(z))\\
  &\leq c2^{-m\alpha}\sum_{k=0}^{l}E(u;B_{2^{-mk}r}(z))\\
    &\quad+c2^{-mn}\sum_{k=0}^{l}({\pmb{\omega}}(2^{-mk}r)+{\pmb{\omega}}_G(2^{-mk}r))({E}(u;B_{2^{-mk}r}(z))+|(u/\oldphi)_{B_{2^{-mk}r}(z)}|)\\
         &\quad+ c2^{-mn}\sum_{k=0}^{l}\max\{\oldphi(z),(2^{-mk}r)^s\}^{-1}{\mathrm{Tail}}(u(\cdot){\pmb{\omega}}(|\cdot-z|);B_{2^{-mk}r}(z))\\
         &\quad+c2^{-mn}\sum_{k=0}^{l}\left(\dashint_{\Omega\cap B_{2^{-mk}r}(z)}((2^{-mk}r)^s|f|)^{2_*}\,dx\right)^{\frac1{2_*}}\eqqcolon \sum_{i-1}^{4}I_i
\end{align*}
for some constant $c=c(\mathsf{data})$.
We first choose $m=m(\mathsf{data})\geq1$ so that 
\begin{equation}\label{cond.m}
    c2^{-m\alpha}\leq 1/4.
\end{equation}
Thus we have 
\begin{equation}\label{ineq.dini.ineqi1}
    I_1\leq \frac14\left[\sum_{k=1}^{l}E(u;B_{2^{-mk}r}(z))+E(u;B_r(z))\right].
\end{equation}
Using \eqref{ass.lem.dini} and \eqref{defn.omega.sec5}, we observe 
\begin{equation}\label{cond.dr}
    d(R)\coloneqq \int_{0}^{R}\frac{{\pmb{\omega}}(\rho)}{\rho}\,d\rho=\int_{0}^{\min\{R,1/8\}}\frac{{{\omega}}(\rho)}{\rho}\,d\rho\leq N\quad\text{for any }R>0.
\end{equation}
Moreover, we also get that for any $\epsilon>0$, there are $r_0\in(0,1/8]$ and $k_0\geq2$ such that 
    \begin{equation}\label{cond.r0k0}
        d(r_0)\leq \epsilon \quad\text{and}\quad 2^{-sk_0m}\leq \epsilon,
    \end{equation}
where $r_0$ and $k_0$ depend on $\mathsf{data}$ and $\epsilon$.
We now observe from \eqref{defn.omega.sec5} that 
\begin{align*}
        \int_{2^{km}R}^{2^{(k+1)m}R}\frac{{\pmb{\omega}}(\rho)}{\rho^{1+s}}\,d\rho\leq \int_{2^{km}R}^{2^{(k+1)m}R}\frac{{{\omega}}(2^{(k+1)m}R)}{(2^{km}R)^{1+s}}\,d\rho\leq c\frac{{\pmb{\omega}}(2^{(k+1)m}R)}{(2^{(k+1)m}R)^s}\quad\text{if } 2^{{k}m}R\leq 2^{-3-2m},
    \end{align*}
 \begin{align*}
        \int_{2^{km}R}^{2^{(k+1)m}R}\frac{{\pmb{\omega}}(\rho)}{\rho^{1+s}}\,d\rho\leq \int_{2^{-2m-3}}^{2^{-3}}\frac{{{\omega}}(\rho)}{\rho^{1+s}}\,d\rho\quad\text{if } 2^{-3+m}<2^{{k}m}R\leq 2^{-3}
    \end{align*}
    and
    \begin{align}\label{ineq0.sec5.dini}
        \int_{2^{km}R}^{2^{(k+1)m}R}\frac{{\pmb{\omega}}(\rho)}{\rho^{1+s}}\,d\rho=0 =\frac{{\pmb{\omega}}(2^{(k+1)m}R)}{(2^{(k+1)m}R)^s}\quad\text{if } 2^{km}R> 2^{-3}
    \end{align}
    for some constant $c=c(\mathsf{data})$.
Using these observations, we have for any $R\leq R_0\coloneqq 2^{-k_0m}r_0$ and $l\geq1$,
\begin{align*}
        \sum_{i=0}^l{\pmb{\omega}}_G(2^{-im}R&)\leq\sum_{i=0}^{l}(2^{-im}R)^{s}\int_{2^{-im}R}^\infty\frac{{\pmb{\omega}}(\rho)}{\rho^{1+s}}\,d\rho\\
        &\leq c\sum_{i=0}^{l}(2^{-im}R)^{s}\left[\sum_{j=0}^{k_0-1}\frac{{\pmb{\omega}}(2^{(-i+j)m}R)}{(2^{(-i+j)m}R)^{s}}+\sum_{j=0}^{\infty}\frac{{\pmb{\omega}}(2^{(-i+k_0+j)m}R)}{(2^{(-i+k_0+j)m}R)^{s}}+\int^{2^{-3}}_{2^{-2m-3}}\frac{\omega(\rho)}{\rho^{1+s}}\,d\rho\right]\\
        &\eqqcolon J_1+J_2+J_3
    \end{align*}
   for some constant $c=c(\mathsf{data})$
     We note that in the above estimates, the constant $c$ depends on $m$ but by virtue of \eqref{cond.m}, eventually it depends only on $\mathsf{data}$.
    By Fubini's theorem, \eqref{cond.r0k0}, we have 
    \begin{align*}
        J_1\leq \sum_{j=0}^{k_0-1}\sum_{i=0}^{l}{\pmb{\omega}}(2^{(-i+j)m}R)2^{-sjm}&\leq c\sum_{j=0}^{k_0-1}2^{-sjm}\sum_{i=0}^{l}\int_{2^{(-i+j)m}R}^{{2^{(-i+j+1)m}R}}\frac{{{\omega}}(2^{(-i+j)m}R)}{\rho}\,d\rho\\
        &\leq c\sum_{j=0}^{k_0-1}2^{-sjm}\sum_{i=0}^{l}\int_{2^{(-i+j)m}R}^{{2^{(-i+j+1)m}R}}\frac{{\pmb{\omega}}(\rho)}{\rho}\,d\rho\\
        &\leq c\sum_{j=0}^{k_0-1}d(2^{(j+1)m}R)2^{-sjm}\leq cd(r_0)\sum_{j=0}^{k_0-1}2^{-sjm}\leq c\epsilon,
    \end{align*}
    where $c=c(\mathsf{data})$.
    Similarly, we employ Fubini's theorem, \eqref{ineq0.sec5.dini} and \eqref{cond.dr} to see that
    \begin{align*}
        J_2&\leq c\sum_{j=0}^{\infty}\sum_{i=0}^{l}\frac{{\pmb{\omega}}(2^{(-i+k_0+j)m}R)}{2^{s(k_0+j)m}}\\
        &\leq  c\sum_{j=0}^{\infty}2^{-s(k_0+j)m}\left[\sum_{i=0}^{l}\int_{2^{(-i+j+k_0)m}R}^{{2^{(-i+j+1+k_0)m}R}}\frac{{\pmb{\omega}}(\rho)}{\rho}\,d\rho+2^{-s(k_0+j)m}\|\omega\|_{L^\infty(B_{1/8)}}\right]\\
        &\leq c\sum_{j=0}^{\infty}2^{-s(k_0+j)m}d(2^{(k_0+j+1)m}R)\leq c2^{-sk_0m}\leq c\epsilon
    \end{align*}
    for some constant $c=c(\mathsf{data},{{\omega}},N)$. After a few simple computations, we estimate $J_3$ as
    \begin{equation*}
        J_3\leq c\sum_{i=0}^l(2^{-im}R)^sN\leq cR_0^s\leq c2^{-sk_0m}\leq c\epsilon
    \end{equation*}
    for some $c=c(\mathsf{data},N)$, where we have used \eqref{cond.dr} and \eqref{cond.r0k0}.
    On the other hand, we have 
    \begin{equation*}
        \sum_{i=0}^{l}{\pmb{\omega}}(2^{-im}R)\leq c\sum_{i=0}^{l}\int_{2^{-im}R}^{2^{(-i+1)m}R}\frac{{{\omega}}(\rho)}{\rho}\,d\rho\leq cd(r_0)\leq c\epsilon.
    \end{equation*}
    Using this together with the estimates $J_1$, $J_2$ and $J_3$, we have if $R\leq R_0$
    \begin{align}\label{ineq2.dini}
        \sum_{i=0}^l({\pmb{\omega}}(2^{-im}R)+{\pmb{\omega}}_G(2^{-im}R))\leq c\epsilon
    \end{align}
    for some constant $c=c(\mathsf{data},{{\omega}},N)$. 
    
    We therefore estimate $I_2$ as
    \begin{equation}\label{ineq3.dini.i2}
    \begin{aligned}
        I_2&\leq c2^{-mn}\left[\sum_{k=0}^{l}({\pmb{\omega}}(2^{-mk}r)+{\pmb{\omega}}_G(2^{-mk}r))\right]\left[\sum_{i=1}^{l}{E}(u;B_{2^{-mi}r}(z))+E(u;B_r(z))+|(u/\oldphi)_{B_{r}(z)}|\right]\\
        &\leq c\epsilon\left(\sum_{k=0}^{l}{E}(u;B_{2^{-mk}r}(z))+E(u;B_r(z))+|(u/\oldphi)_{B_{r}(z)}|\right)
    \end{aligned}
    \end{equation}
    for some constant $c=c(\mathsf{data})$ provided that $r\leq R_0$, where we have used
    \begin{equation}\label{ineq4.dini}
    \begin{aligned}
        |(u/\oldphi)_{B_{2^{-km}r}}(z)|&\leq c\sum_{i=1}^{k}E_{\mathrm{loc}}(u;B_{2^{-im}r}(z)) +|(u/\oldphi)_{B_{r}(z)}|\\
        &\leq c\sum_{i=1}^{l}E_{\mathrm{loc}}(u;B_{2^{-im}r}(z)) +|(u/\oldphi)_{B_{r}(z)}|
    \end{aligned}
    \end{equation}
    and \eqref{ineq2.dini}. To estimate $I_3$, we first note from \cite[Lemma 2.2]{DieKimLeeNow24n} that
\begin{align*}
    &\max\{\oldphi(z),(2^{-mk}r)^s\}^{-1}{\mathrm{Tail}}(u(\cdot){\pmb{\omega}}(\cdot-z);B_{2^{-mk}r}(z))\\
    &\leq c\max\{\oldphi(z),(2^{-mk}r)^s\}^{-1}\sum_{i=1}^{k} 2^{-2s(im)}{\pmb{\omega}}(2^{(-k+i)m}r)\dashint_{B_{2^{(-k+i)m}r}(z)}|u|\,dy\\
    &\quad+c\max\{\oldphi(z),(2^{-mk}r)^s\}^{-1} 2^{-2skm}{\mathrm{Tail}}(u(\cdot){\pmb{\omega}}(\cdot-z);B_{r}(z))\coloneqq I_{3,1}(k)+I_{3,2}(k).
\end{align*}
By means of \eqref{ineq2.decay} with $\rho=1$, we estimate $I_{3,1}(k)$ as 
\begin{align*}
    I_{3,1}(k)&\leq c\sum_{i=0}^{k} 2^{-s(im)}{\pmb{\omega}}(2^{(-k+i)m}r)\dashint_{B_{2^{(-k+i)m}r}(z)}|u/\oldphi|\,dy\\
    &\leq c\sum_{i=0}^{k} 2^{-s(im)}{\pmb{\omega}}(2^{(-k+i)m}r)(E_{\mathrm{loc}}(u;B_{2^{(-k+i)m}r}(z))+|(u/\oldphi)_{B_{2^{(-k+i)m}r}(z)}|).
\end{align*}
In addition, we estimate $I_{3,2}(k)$ as 
\begin{align*}
    I_{3,2}(k)\leq c\max\{\oldphi(z),r^s\}^{-1} 2^{-skm}{\mathrm{Tail}}(u(\cdot){\pmb{\omega}}(\cdot-z);B_{r}(z)).
\end{align*}
By combining the estimates $I_{3,1}(k)$ and $I_{3,2}(k)$ along with Fubini's theorem, we have
\begin{align*}
    I_3&\leq c\sum_{k=0}^{l}I_{3,1}(k)+I_{3,2}(k)\\
    &\leq c\sum_{k=0}^{l}\sum_{i=0}^{k} 2^{-s(im)}{\pmb{\omega}}(2^{(-k+i)m}r)(E_{\mathrm{loc}}(u;B_{2^{(-k+i)m}r}(z))+|(u/\oldphi)_{B_{2^{(-k+i)m}r}(z)}|)\\
    &\quad+ c\max\{\oldphi(z),r^s\}^{-1}{\mathrm{Tail}}(u(\cdot){\pmb{\omega}}|\cdot|;B_{r}(z))\\
    &\leq c\sum_{i=0}^{l} \sum_{k=i}^{l} 2^{-s(im)}{\pmb{\omega}}(2^{(-k+i)m}r)(E_{\mathrm{loc}}(u;B_{2^{(-k+i)m}r}(z))+|(u/\oldphi)_{B_{2^{(-k+i)m}r}(z)}|)\\
    &\quad+ c\max\{\oldphi(z),r^s\}^{-1}{\mathrm{Tail}}(u(\cdot){\pmb{\omega}}(\cdot-z);B_{r}(z))
\end{align*}
for some constant $c=c(\mathsf{data})$.
We then employ \eqref{ineq2.dini} to further estimate $I_3$ as 
\begin{equation}\label{ineq5.dini.i3}
\begin{aligned}
    I_3 &\leq c\sum_{i=0}^{l}2^{-s(im)} \left[\sum_{k=i}^{l} {\pmb{\omega}}(2^{(-k+i)m}r)\right]\left[\sum_{j=0}^{l}E_{\mathrm{loc}}(u;B_{2^{-jm}r}(z))+E_{\mathrm{loc}}(u;B_r(z))+|(u/\oldphi)_{B_{r}(z)}|\right]\\
    &\quad+ c\max\{\oldphi(z),r^s\}^{-1}{\mathrm{Tail}}(u(\cdot){\pmb{\omega}}(\cdot-z);B_{r}(z))\\
    &\leq c\epsilon\sum_{i=0}^{l}2^{-s(im)}\left[\sum_{j=0}^{l}E_{\mathrm{loc}}(u;B_{2^{-jm}r}(z))+E_{\mathrm{loc}}(u;B_r(z))+|(u/\oldphi)_{B_{r}(z)}|\right]\\
    &\quad+c\max\{\oldphi(z),r^s\}^{-1}{\mathrm{Tail}}(u(\cdot){\pmb{\omega}}(\cdot-z);B_{r}(z))\\
    &\leq c\epsilon\left[\sum_{j=0}^{l}E_{\mathrm{loc}}(u;B_{2^{-jm}r}(z))+E_{\mathrm{loc}}(u;B_r(z))+|(u/\oldphi)_{B_{r}(z)}|\right]\\
    &\quad+c\max\{\oldphi(z),r^s\}^{-1}{\mathrm{Tail}}(u(\cdot){\pmb{\omega}}(\cdot-z);B_{r}(z)),
\end{aligned}
\end{equation}
where $c=c(\mathsf{data},{{\omega}},N)$ whenever $r\leq R_0$. 

We now estimate $I_4$ as 
\begin{equation}\label{ineq6.dini.i4}
\begin{aligned}
    I_4&\leq c\sum_{k=0}^{l}\int_{2^{-m(k+1)}r}^{2^{-mk}r}\left(\frac{\||f|^{2_*}\chi_{\Omega}\|_{L^1(B_{2^{-mk}r})(z)}}{(2^{-mk}r)^{n-2_* s}}\right)^{\frac{1}{2_*}}\frac{\,d\rho}{\rho}\\
    &\leq c\sum_{k=0}^{l}\int_{2^{-m(k+1)}r}^{2^{-mk}r}\left(\frac{\||f|^{2_*}\chi_{\Omega}\|_{L^1(B_{\rho}(z))}}{\rho^{n-2_* s}}\right)^{\frac{1}{2_*}}\frac{\,d\rho}{\rho}\\
    &\leq c\mathbf{W}^{|f|^{2_*}\chi_{\Omega}}_{\frac{2_* s}{2_*+1},2_*+1}(z,r),
\end{aligned}
\end{equation}
for some constant $c=c(\mathsf{data})$, where we have used \eqref{rmk.bdry}.
Combining all the estimates \eqref{ineq.dini.ineqi1}, \eqref{ineq3.dini.i2}, \eqref{ineq5.dini.i3} and \eqref{ineq6.dini.i4}, we obtain 
\begin{align*}
    \sum_{k=0}^{l}  E(u;B_{2^{-m(k+1)}r}(z))&\leq \frac{1}{4}\sum_{k=1}^{l}E(u;B_{2^{-mk}r}(z))+cE(u;B_r(z))\\
    &\quad +c\epsilon\left(\sum_{k=1}^{l}{E}(u;B_{2^{-mk}r}(z))+E(u;B_r(z))+|(u/\oldphi)_{B_{r}(z)}|\right)\\
    &\quad+c\max\{\oldphi(z),r^s\}^{-1}{\mathrm{Tail}}(u(\cdot){\pmb{\omega}}(\cdot-z);B_{r}(z))+c\mathbf{W}^{|f|^{2_*}\mbox{\Large$\chi$}_{\Omega}}_{\frac{2_* s}{2_*+1},2_*+1}(z,r)
\end{align*}
for some constant $c=c(\mathsf{data},{\omega},N)$, whenever $r\leq R_0$.
We now take $r_0$ and $k_0$ sufficiently small so that 
\begin{equation*}
    c\epsilon\leq 1/4,
\end{equation*}
where the constants $r_0$ and $k_0$ depend on $\mathsf{data}$, ${\omega}$ and $N$. Thus we now fix
\begin{equation}\label{dini.cho.r0}
    R_0=2^{-k_0}r_0
\end{equation} 
to see that for any $r\leq R_0$ 
\begin{equation}\label{dini.ineq.r0}
\begin{aligned}
     \sum_{k=0}^{l}  E(u;B_{2^{-m(k+1)}r}(z))&\leq c\left[E(u;B_r(z))+|(u/\oldphi)_{B_{r}(z)}|+{\mathrm{Tail}}\left(\frac{u}{r^s};B_{r}(z)\right)+\mathbf{W}^{|f|^{2_*}\mbox{\Large$\chi$}_{\Omega}}_{\frac{2_* s}{2_*+1},2_*+1}(z,r)\right]
\end{aligned}
\end{equation}
for some constant $c=c(\mathsf{data},{{\omega}},N)$, where we have used the fact that 
\begin{equation*}
    {\pmb{\omega}}(\rho)\leq \|\omega\|_{L^\infty(B_{1/8})}\quad\text{for any }\rho>0.
\end{equation*}
Using this along with the fact that 
\begin{align*}
    \dashint_{B_{2^{-lm}r}(z)}|u/\oldphi|\,dx&\leq E_{\mathrm{loc}}(u;B_{2^{-lm}r}(z))+|(u/\oldphi)_{B_{2^{-lm}r}(z)}|\leq \sum_{i=0}^{l}E_{\mathrm{loc}}(u;B_{2^{-im}r}(z))+|(u/\oldphi)_{B_{r}(z)}|,
\end{align*} we obtain that for any $l\geq0$ and $r\leq R_0$, there holds
\begin{align}\label{ineq7.dini}
    \dashint_{B_{2^{-lm}r}(z)}|u/\oldphi|\,dx\leq c\left({E}(u;B_r(z))+|(u/\oldphi)_{B_r(z)}|+{\mathrm{Tail}}(u/r^s;B_{r}(z))+\mathbf{W}^{|f|^{2_*}\mbox{\Large$\chi$}_{\Omega}}_{\frac{2_* s}{2_*+1},2_*+1}(z,r)\right),
\end{align}
where $c=c(\mathsf{data},{{\omega}},N)$. 
We now assume $r>R_0$. Let us fix a positive integer $l$. If $2^{-lm}r\geq 2^{-m}R_0$, then we have 
\begin{align*}
   \dashint_{B_{2^{-lm}r}(z)}|u/\oldphi|\,dx\leq c\dashint_{B_{r}(z)}|u/\oldphi|\,dx\leq  c\left({E}(u;B_r(z))+|(u/\oldphi)_{B_r(z)}|\right)
\end{align*}
for some constant $c=c(\mathsf{data},{{\omega}},N)$. If $2^{-lm}r< 2^{-m}R_0$, then there is a positive integer $i$ such that $2^{-(i+1)m}R_0\leq 2^{-lm}r<2^{-im}R_0$. By \eqref{ineq7.dini}, we have 
\begin{equation}\label{ineq8.dini}
\begin{aligned}
    \dashint_{B_{2^{-lm}r}(z)}|u/\oldphi|\,dx&\leq  c\dashint_{B_{2^{-im}R_0}(z)}|u/\oldphi|\,dx\\
    &\leq c\left[{E}(u;B_{R_0}(z))+c|(u/\oldphi)_{B_{R_0}(z)}|+{\mathrm{Tail}}(u;B_{R_0}(z))+\mathbf{W}^{|f|^{2_*}\mbox{\Large$\chi$}_{\Omega}}_{\frac{2_* s}{2_*+1},2_*+1}(z,R_0)\right]
\end{aligned}
\end{equation}
for some constant $c=c(\mathsf{data},{{\omega}},N)$. Using \eqref{ineq.simple}, \eqref{rel.phid1} and the fact that $d(z)\leq c(\Omega)$, we have
\begin{align*}
    {\mathrm{Tail}}(u;B_{R_0}(z))&\leq c\left({\mathrm{Tail}}(u/{r}^s;B_{r}(z))+\|u\|_{L^1(B_{r}(z))}\right)\\
    &\leq c\left({E}(u;B_{r}(z))+|(u/\oldphi)_{B_{r}(z)}|+{\mathrm{Tail}}(u/{r}^s;B_{r}(z))\right).
\end{align*}
Plugging this into \eqref{ineq8.dini} together with Lemma \ref{lem.sim.exc} yields \eqref{ineq7.dini} provided $r\in(R_0,1/64]$. This completes the proof.
\end{proof}
Using this lemma \ref{lem.dini}, we now prove a pointwise estimate of the maximal function of $u/d^s$.
\begin{lemma}\label{lem.dini.max}
    Let us assume \eqref{ass.lem.dini}. Then we have 
    \begin{align*}
        M^{\Omega}_{1/64}(u/d^s)(z)\leq c\|u\|_{L^1_{2s}(\bbR^n)}+c\mathbf{W}^{|f|^{2_*}\mbox{\Large$\chi$}_{\Omega}}_{\frac{2_* s}{2_*+1},2_*+1}(z,2^{-2})
    \end{align*}
    for some constant $c=c(\mathsf{data},N,{\pmb{\omega}})$.
\end{lemma}
\begin{proof}
    Let us fix $r\in(0,1/64]$. Then there is a positive integer $i$ such that $2^{-im-6}<r\leq 2^{-(i-1)m-6}$. By Lemma \ref{lem.dini}, we have
    \begin{equation}\label{ineq1.lem.dini.max}
        \dashint_{B_r(z)}|u/\oldphi|\,dx\leq c{E}(u;B_{2^{-6}}(z))+c|(u/\oldphi)_{B_{2^{-6}}(z)}|+c{\mathrm{Tail}}(u;B_{2^{-6}}(z))+c\mathbf{W}^{|f|^{2_*}\mbox{\Large$\chi$}_{\Omega}}_{\frac{2_* s}{2_*+1},2_*+1}(z,2^{-6})
    \end{equation}
    for some constant $c=c(\mathsf{data},N,{\pmb{\omega}})$. Moreover, we observe 
    \begin{equation}\label{ineq2.lem.dini.max}
    \begin{aligned}
        {E}(u;B_{2^{-6}}(z))+|(u/\oldphi)_{B_{2^{-6}}(z)}|&\leq c\left(\|u/\oldphi\|_{L^1(B_{2^{-6}}(z))}+\|u\|_{L^1_{2s}}(\bbR^n)\right)\\
        &\leq c\left(\|u/d^s\|_{L^1(B_{2^{-6}}(z))}+\|u\|_{L^1_{2s}}(\bbR^n)\right)\\
        &\leq c\|u\|_{L^1_{2s}(\bbR^n)}+c\|f\|_{L^{2_*}(\Omega\cap B_{2^{-3}}(z))}
    \end{aligned}
    \end{equation}
 for some constant $c=c(\mathsf{data},N,{\pmb{\omega}})$, where we have used \eqref{rel.phid1}, \eqref{ineq2.lem.vmo} and
    \begin{align*}
        \int_{\bbR^n\setminus B_{2^{-6}}(z)}\frac{|\oldphi(y)|}{|y-z|^{n+2s}}\,dy\leq c.
        \end{align*}
Using this, we further estimate the right-hand side of \eqref{ineq1.lem.dini.max} as 
\begin{align*}
    \dashint_{B_r(z)}|u/\oldphi|\,dx\leq c\left(\|u\|_{L^1_{2s}(\bbR^n)}+\mathbf{W}^{|f|^{2_*}\mbox{\Large$\chi$}_{\Omega}}_{\frac{2_* s}{2_*+1},2_*+1}(z,2^{-2})\right).
\end{align*} 
Since this holds for every $r\in(0,1/64]$, we have the desired estimate.
\end{proof}

\begin{remark}\label{rmk.hol.sec5}
    We note that if $\omega(\rho)\leq L\rho^\alpha$ for some $L\geq1$, where $\rho\in[0,1/8]$, then we can choose the constants $r_0$ and $\epsilon$ given in \eqref{cond.r0k0}, where such constants depend only on $\mathsf{data}$, $L$ and $\alpha$. Therefore, \eqref{ineq.lem.dini} holds with the constant $c=c(\mathsf{data},L,\alpha)$.
\end{remark}
We now use Lemma \ref{lem.dini} to obtain a point-wise estimate of the following nonlocal sharp maximal function
\begin{equation}\label{defn.glosharp}
    {N}^{\sigma}_{1/64}(u;z)\coloneqq\sup_{0<r<1/64}\frac{{E}(u;B_r(z))}{r^\sigma},
\end{equation}
where $\sigma\in(0,\alpha)$ and $z\in B_{1/4}$.  We note that such a maximal function is first introduced in \cite{DieNow23}.
\begin{lemma}\label{lem.hol}
Let us assume for any $\rho\in(0,1/8]$,
\begin{equation}\label{ass.lem.hol}
    {{\omega}}(\rho)\leq L\rho^\alpha,
\end{equation}
where $L>0$ and $u\equiv 0$ in $\bbR^n\setminus B_2$. Then for any $\sigma\in(0,\alpha)$, we have 
\begin{equation}\label{ineq.lem.hol}
\begin{aligned}
    {N}^{\sigma}_{1/64}(u;z)&\leq c\left(\|u\|_{L^1_{2s}(\bbR^n)}+\mathbf{W}^{|f|^{2_*}}_{\frac{2_* s}{2_*+1},2_*+1}(z,1/4)+\left({M}^{\Omega}_{2_*(s-\sigma),1/4}(|f|^{2_*})(z)\right)^{\frac1{2_*}}\right),
\end{aligned}
\end{equation}for some constant $c=c(\mathsf{data},\alpha,L,\sigma)$.
\end{lemma}
\begin{remark}
    We note from Remark \ref{rmk.sec5} that we always assume $u\equiv 0$ in $\bbR^n\setminus B_2$.
\end{remark}

\begin{proof}
Let us fix $\rho\in(0,1/2]$ which will be determined later. Then for any $r\in(0,\rho/64]$, Lemma \ref{lem.decay} yields
\begin{align*}
    {E}(u;B_{r}(z))&\leq c\rho^\alpha{E}(u;B_{r/\rho}(z))\\
         &\quad+ c\rho^{-n}({\pmb{\omega}}(r/\rho)+{\pmb{\omega}}_G(r/\rho))({E}(u;B_{r/\rho}(z))+|(u/\oldphi)_{B_{r/\rho}(z)}|)\\
         &\quad+ c\rho^{-n}\max\{\oldphi(z),(r/\rho)^s\}^{-1}{\mathrm{Tail}}(u(\cdot){\pmb{\omega}}(|\cdot-z|);B_{r/\rho}(z))\\
         &\quad+c\rho^{-n}\left(\dashint_{\Omega\cap B_{r/\rho}(z)}((r/\rho)^s|f|)^{2_*}\,dx\right)^{\frac1{2_*}}\coloneqq\sum_{i=1}^{4}I_i,
\end{align*}
where $c=c(\mathsf{data})$. We now choose $\rho=\rho(\mathsf{data},\sigma)$ such that 
\begin{equation}\label{cho.rho.lem.hol}
    c\rho^{\alpha-\sigma}\leq 1/4
\end{equation}
to see that
\begin{equation}\label{ineq1.lem.hol}
    I_1\leq \frac{\rho^{\sigma}}{4}E(u;B_{r/\rho}(z)).
\end{equation}
By \eqref{ass.lem.hol}, \eqref{defn.omega.sec5} and \eqref{cho.rho.lem.hol}, we have 
\begin{equation*}
    {\pmb{\omega}}(r/\rho)\leq cr^\alpha
\end{equation*}
and
\begin{equation*}
    {\pmb{\omega}}_G(r/\rho)=(r/\rho)^s\int_{r/\rho}^\infty\frac{{\pmb{\omega}}(\xi)}{\xi^{1+s}}\,d\xi\leq (r/\rho)^s\int_{r/\rho}^{1/8}\frac{\omega(\xi)}{\xi^{1+s}}\,d\xi\leq c(r/\rho)^{\alpha}\leq cr^\alpha
\end{equation*}
for some constant $c=c(\mathsf{data},L,\alpha)$.
Therefore, we have
\begin{equation*}
   I_2\leq cr^\alpha({E}(u;B_{r/\rho}(z))+|(u/\oldphi)_{B_{r/\rho}(z)}|)
\end{equation*}
for some constant $c=c(\mathsf{data},L,\alpha,\sigma)$. Since $r/\rho\leq1/64$, there is a positive integer $i$ such that 
\begin{equation}\label{i.lem.hol}
    2^{-6-im}\leq r/\rho\leq 2^{-6-(i-1)m}.
\end{equation} 
By Remark \ref{rmk.hol.sec5}, we have for any nonnegative integer $k\in[0,(i-1)m]$,
\begin{align}\label{ineq11.lem.hol}
  \dashint_{B_{2^kr/\rho}(z)}|u/\oldphi|\,dx&\leq c\left(\|u\|_{L^1_{2s}(\bbR^n)}+\mathbf{W}^{|f|^{2_*}\mbox{\Large$\chi$}_{\Omega}}_{\frac{2_* s}{2_*+1},2_*+1}(z,2^{-2})\right)
\end{align}
for some $c=c(\mathsf{data},L,\alpha,\sigma)$. As in the estimate of $J_1$ and $J_2$ given in \eqref{ineq.tail.decay} and \eqref{ineq3.decay}, we get 
\begin{equation}\label{ineq12.lem.hol}
\begin{aligned}
    \max\{\oldphi(z),(r/\rho)^s\}^{-1}\mathrm{Tail}(u-(u/\oldphi)_{B_{r/\rho}(z)}\oldphi;B_{r/\rho}(z))&\leq c\sum_{k=0}^{(i-1)m}2^{-ks}E_{\mathrm{loc}}(u;B_{2^kr/\rho}(z))\\
    &\quad+c2^{-(i-1)ms}E(u;B_{2^{-3}}(z)).
\end{aligned}
\end{equation}
Using \eqref{ineq11.lem.hol}, \eqref{ineq12.lem.hol} and \eqref{ineq2.lem.dini.max}, we have 
\begin{equation}\label{ineq2.lem.hol}
\begin{aligned}
    I_2&\leq cr^\alpha\left(\sum_{k=0}^{(i-1)m}2^{-ks}E_{\mathrm{loc}}(u;B_{2^kr/\rho}(z))+2^{-(i-1)ms}E(u;B_{2^{-3}}(z))\right)\\
    &\leq  cr^\alpha\left(\|u\|_{L^1_{2s}(\bbR^n)}+\mathbf{W}^{|f|^{2_*}\mbox{\Large$\chi$}_{\Omega}}_{\frac{2_* s}{2_*+1},2_*+1}(z,2^{-2})\right),
\end{aligned}
\end{equation}
where $c=c(\mathsf{data},L,\alpha,\sigma)$. We now estimate $I_3$. By \cite[Lemma 2.2]{DieKimLeeNow24n}, we have
\begin{align*}
    I_3&\leq   c\rho^{-n}\max\{\oldphi(z),(r/\rho)^{s}\}^{-1}(r/\rho)^{2s}\int_{\bbR^n\setminus B_{r/\rho}(z)}\frac{|u(y)|}{|y-z|^{n+2s-\alpha}}\,dy\\
    &\leq c\rho^{-n}\max\{\oldphi(z),(r/\rho)^{s}\}^{-1}(r/\rho)^{2s}\sum_{k=0}^{(i-1)m}(2^kr/\rho)^{-2s+\alpha}\dashint_{B_{2^kr/\rho}(z)}|u|\,dy\\
    &\quad+ c\rho^{-n}\max\{\oldphi(z),(r/\rho)^{s}\}^{-1}(r/\rho)^{2s}\dashint_{B_{2^{-6}}(z)}|u|\,dy\\
    &\quad+c\rho^{-n}\max\{\oldphi(z),(r/\rho)^{s}\}^{-1}(r/\rho)^{2s}\int_{\bbR^n\setminus B_{2^{-6}}(z)}\frac{|u(y)|}{|y-z|^{n+2s-\alpha}}\,dy\coloneqq \sum_{j=1}^{3}I_{3,j}.
\end{align*}
By \eqref{ineq2.decay}, we estimate $I_{3,1}$ and $I_{3,2}$ as 
\begin{align*}
    I_{3,1}+I_{3,2}&\leq cr^{2s}\left[\sum_{k=0}^{(i-1)m}2^{ks}(2^kr)^{-2s+\alpha}\dashint_{B_{2^kr/\rho}(z)}|u/\oldphi|\,dy+2^{ims}\dashint_{B_{2^{-6}}(z)}|u/\oldphi|\,dy\right]\\
    &\leq cr^\alpha\left(\|u\|_{L^1_{2s}(\bbR^n)}+\mathbf{W}^{|f|^{2_*}\mbox{\Large$\chi$}_{\Omega}}_{\frac{2_* s}{2_*+1},2_*+1}(z,2^{-2})\right)
\end{align*}
for some constant $c=c(\mathsf{data},L,\alpha,\sigma)$, where we have used Lemma \ref{lem.dini.max}, \eqref{ineq1.lem.dini.max} and \eqref{i.lem.hol}. Thus we have 
\begin{equation}\label{ineq3.lem.hol}
\begin{aligned}
    I_3&\leq cr^\sigma\left(\|u\|_{L^1_{2s}(\bbR^n)}+\mathbf{W}^{|f|^{2_*}\mbox{\Large$\chi$}_{\Omega}}_{\frac{2_* s}{2_*+1},2_*+1}(z,2^{-2})\right)
\end{aligned}
\end{equation}
for some constant $c=c(\mathsf{data},L,\alpha,\sigma)$, where we have used the fact that $u\equiv 0$ in $\bbR^n\setminus B_2$ and $\sigma<\alpha$. We next estimate $I_4$ as 
\begin{align}\label{ineq4.lem.hol}
    I_4\leq cr^{\sigma}\left(r^{2_*(s-\sigma)}\dashint_{\Omega\cap B_r(z)}|f|^{2_*}\,dx\right)^{\frac{1}{2_*}}\leq cr^{\sigma}\left(M^{\Omega}_{2_*(s-\sigma),1/4}(|f|^{2_*})(z))\right)^{\frac{1}{2_*}},
\end{align}
where $c=c(\mathsf{data},L,\alpha,\sigma)$.
Therefore, combining four estimates $I_1,I_2,I_3$ and $I_4$ given in \eqref{ineq1.lem.hol}, \eqref{ineq2.lem.hol}, \eqref{ineq3.lem.hol} and \eqref{ineq4.lem.hol}, respectively,  we discover that for $r\leq \rho/64$,
\begin{equation}\label{ineq5.lem.hol}
\begin{aligned}
    E(u;B_{r}(z))/r^{\sigma}&\leq \frac{1}{4}E(u;B_{r/\rho}(z))/(r/\rho)^{\sigma}\\
    &\quad+c\left(\|u\|_{L^1_{2s}(\bbR^n)}+\mathbf{W}^{|f|^{2_*}\mbox{\Large$\chi$}_{\Omega}}_{\frac{2_* s}{2_*+1},2_*+1}(z,1/4)+\left(M^{\Omega}_{2_*(s-\sigma),1/4}(|f|^{2_*})(z))\right)^{\frac{1}{2_*}}\right)
\end{aligned}
\end{equation}
where $c=c(\mathsf{data},L,\alpha,\sigma)$. In addition, for $r> \rho /64$, we note from \eqref{ineq2.lem.dini.max} that 
\begin{align}\label{ineq6.lem.hol}
    E(u;B_{r}(z))/r^{\sigma}&\leq cE(u;B_{1/8}(z))\leq c\left(\|u\|_{L^1_{2s}(\bbR^n)}+\mathbf{W}^{|f|^{2_*}\mbox{\Large$\chi$}_{\Omega}}_{\frac{2_* s}{2_*+1},2_*+1}(z,1/4)\right)
\end{align}
for some constant $c=c(\mathsf{data},L,\alpha,\sigma)$. To get the desired estimate \eqref{ineq.lem.hol}, we now consider
\begin{equation*}
    {N}^{\sigma,\epsilon}_{2^{-6}}(u;z)\coloneqq\sup_{\epsilon/8<r<2^{-6}}\frac{{E}(u;B_r(z))}{r^\sigma}
\end{equation*}
for any $\epsilon\in(0,1)$. Then by \eqref{ineq5.lem.hol} and \eqref{ineq6.lem.hol}, we have 
\begin{align*}
    {N}^{\sigma,\epsilon}_{2^{-6}}(u;z)&\leq c\left(\|u\|_{L^1_{2s}(\bbR^n)}+\mathbf{W}^{|f|^{2_*}\mbox{\Large$\chi$}_{\Omega}}_{\frac{2_* s}{2_*+1},2_*+1}(z,1/4)+\left(M^{\Omega}_{2_*(s-\sigma),1/4}(|f|^{2_*})(z))\right)^{\frac{1}{2_*}}\right)
\end{align*}
for some constant $c=c(\mathsf{data},L,\alpha,\sigma)$. By taking $\epsilon\to0$, we get \eqref{ineq.lem.hol}.
\end{proof}

\subsection{Proof of Main Theorems}
In this subsection, we are going to prove our main results for continuous coefficient case, namely Theorem \ref{thm.dini}, Corollary \ref{cor.cs}, Theorem \ref{thm.hol.1} and Theorem \ref{thm.hol.2}. 

We start by proving Theorem \ref{thm.dini}. 
\begin{proof}[Proof of Theorem \ref{thm.dini}]
We first note that since $A$ is Dini-continuous in $B_{1}\times B_1$, $A$ is also VMO in $B_{1}\times B_1$. Therefore, by \eqref{ineq.simple} and Theorem \ref{thm.vmo}, we have $u\in L^{n/s}(B_{7/8})$  with 
\begin{equation}\label{ineq1.thm.dini}
\begin{aligned}
    \|u\|_{L^{n/s}(B_{7/8})}\leq c\max\{1,d(0)^s\}\|u/d^s\|_{L^{n/(2s)}(B_{7/8})}&\leq c(\|u\|_{L^1_{2s}(\bbR^n)}+\|f\|_{L^{{n}/{(2s)}}(\Omega\cap B_{1})})\\
    &\leq c(\|u\|_{L^1_{2s}(\bbR^n)}+\|f\|_{L^{{n}/{s},1}(\Omega\cap B_{1})})
\end{aligned}
\end{equation}
for some constant $c=c(\mathsf{data},\omega)$. Let us fix $z\in \overline{\Omega}\cap B_{1/2}$. We are going to prove
\begin{equation}\label{ineq2.thm.dini}
    M^{\Omega}_{2^{-6}}(u/d^s)(z)\leq c(\|u\|_{L^1_{2s}(\bbR^n)}+\|f\|_{L^{{n}/{s},1}(\Omega\cap B_{1})})
\end{equation}
for some constant $c=c(\mathsf{data},\omega)$. To do this, we first use Lemma \ref{lem.dini.max} to see that 
\begin{equation*}
    M^{\Omega}_{2^{-6}}(\widetilde{u}/d^s)(z)\leq c\|\widetilde{u}\|_{L^1_{2s}(\bbR^n)}+c\mathbf{W}^{|F|^{2_*}\mbox{\Large$\chi$}_{\Omega}}_{\frac{2_* s}{2_*+1},2_*+1}(z,2^{-2}),
\end{equation*}
where $\widetilde{u}\in W^{s,2}(\bbR^n)$ is a weak solution to \eqref{eq: sec5.loc}.
By the fact that $\widetilde{u}=u$ in $B_{5/64}(z)$ and $|\widetilde{u}(x)|\leq |u(x)|$, we have
\begin{align}\label{ineq8.thm.dini}
    M^{\Omega}_{2^{-6}}({u}/d^s)(z)&\leq c\|u\|_{L^1_{2s}(\bbR^n)}+c\mathbf{W}^{|F|^{2_*}\mbox{\Large$\chi$}_{\Omega}}_{\frac{2_* s}{2_*+1},2_*+1}(z,2^{-2})
\end{align}
for some constant $c=c(\mathsf{data},\omega)$. It remains to estimate the last term given in the right-hand side of \eqref{ineq8.thm.dini}. To this end, we use \eqref{sec5.F} to get that 
\begin{align}\label{ineq81.thm.dini}
    \mathbf{W}^{|F|^{2_*}\mbox{\Large$\chi$}_{\Omega}}_{\frac{2_* s}{2_*+1},2_*+1}(z,2^{-2})\leq c\left[\mathbf{W}^{|f|^{2_*}\mbox{\Large$\chi$}_{\Omega}}_{\frac{2_* s}{2_*+1},2_*+1}(z,2^{-2})+\mathbf{W}^{|u|^{2_*}\mbox{\Large$\chi$}_{\Omega}}_{\frac{2_* s}{2_*+1},2_*+1}(z,2^{-2})+\mathrm{Tail}(u;B_{5/8}(z))\right]
\end{align}
for some constant $c=c(n,s,\Lambda)$. We first observe from \eqref{ineq2.wolff} that 
\begin{equation}\label{ineq82.thm.dini}
    \mathbf{W}^{|f|^{2_*}\mbox{\Large$\chi$}_{\Omega}}_{\frac{2_* s}{2_*+1},2_*+1}(z,2^{-2})\leq c\||f|^{2_*}\|_{L^{n/(2_*s),1/2_*}(\Omega\cap B_{7/8})}\leq c\|f\|_{L^{n/s,1}(\Omega\cap B_{1})}
\end{equation}
for some constant $c=c(n,s)$, where we have used a simple property of Lorentz space for the last inequality. On the other hand, by \eqref{ineq1.thm.dini}, we have 
\begin{equation}\label{ineq83.thm.dini}
    \mathbf{W}^{|u|^{2_*}\mbox{\Large$\chi$}_{\Omega}}_{\frac{2_* s}{2_*+1},2_*+1}(z,2^{-2})\leq c\|u\|_{L^{n/s,1}(\Omega\cap B_{7/8})}\leq c(\|u\|_{L^1_{2s}(\bbR^n)}+\|f\|_{L^{n/s,1}(\Omega\cap B_{1})})
\end{equation}
for some constant $c=c(\mathsf{data},\omega)$. Combining \eqref{ineq81.thm.dini}, \eqref{ineq82.thm.dini} and \eqref{ineq83.thm.dini} gives 
\begin{align}
    \mathbf{W}^{|F|^{2_*}\mbox{\Large$\chi$}_{\Omega}}_{\frac{2_* s}{2_*+1},2_*+1}(z,2^{-3})\leq c(\|u\|_{L^1_{2s}(\bbR^n)}+\|f\|_{L^{n/s,1}(\Omega\cap B_{1})})
\end{align}
for some constant $c=c(\mathsf{data},\omega)$. Plugging this into \eqref{ineq8.thm.dini} yields 
\begin{align*}
    M^{\Omega}_{2^{-6}}({u}/d^s)(z)&\leq c\|u\|_{L^1_{2s}(\bbR^n)}+c\|f\|_{L^{n/s,1}(\Omega\cap B_{1})}),
\end{align*}
where $c=c(\mathsf{data},\omega)$. Since  $z\in\Omega\cap B_{1/2}$ was arbitrary, we have 
\begin{align*}
    \sup_{z\in\Omega\cap B_{1/2}}M^{\Omega}_{2^{-6}}({u}/d^s)(z)&\leq c\left(\|u\|_{L^1_{2s}(\bbR^n)}+\|f\|_{L^{n/s,1}(\Omega\cap B_{1})}\right),
\end{align*}
which implies the desired estimate.
\end{proof}

Using Theorem \ref{thm.dini}, we now prove Corollary
 \ref{cor.cs}.
\begin{proof}[Proof of Corollary \ref{cor.cs}]
    Let $x,y\in  B_{1/2}$. We distinguish the following cases to complete the proof.\\
    \texttt{Case (a)}:   Suppose 
    \begin{equation*}
        5|x-y|\geq \max\{d(x),d(y)\}.
    \end{equation*}
    Then by Theorem \ref{thm.dini}, we have
    \begin{equation}\label{ineq1.cor.hol}
    \begin{aligned}
        |u(x)-u(y)|&\leq |u(x)|+|u(y)|\\
        &\leq c(\|u\|_{L^1_{2s}(\bbR^n)}+\|f\|_{L^{n/s,1}(\Omega\cap B_{1})})(|d^s(x)|+|d^s(y)|)\\
        &\leq c(\|u\|_{L^1_{2s}(\bbR^n)}+\|f\|_{L^{n/s,1}(\Omega\cap B_{1})})|x-y|^{s}
    \end{aligned}
    \end{equation}
    for some constant $c=c(\mathsf{data})$.\\ 
  \texttt{Case (b)}: Suppose 
    \begin{equation*}
        5|x-y|\leq \max\{d(x),d(y)\}.
    \end{equation*}
    We may assume $d(x)\geq d(y)$. Then $y\in B_{d(x)}(x)\subset \Omega$. Since $A$ is VMO in $B_1\times B_1$ and $f\in L^{n/s,1}(\Omega\cap B_1)\subset L^{n/s,\infty}(\Omega\cap B_1)$, we note from \cite[Theorem 1.7]{DieNow23} that 
    \begin{equation}\label{ineq2.cor.hol}
       \frac{|u(x)-u(y)|}{|x-y|^s}\leq \frac{c}{d^s(x)}\left(\|u\|_{L^\infty(B_{2d(x)/3}(x))}+\mathrm{Tail}(u;B_{2d(x)/3}(x))+d^s(x)\|f\|_{L^{n/s,1}(\Omega\cap B_{d(x)}(x))}\right)
    \end{equation}
    for some constant $c=c(\mathsf{data},\omega)$.
    Let us fix a positive integer $l$ such that 
    \begin{equation}\label{l.cor.hol}
        1/16\leq 2^ld(x)<1/8\quad\text{and}\quad  B_{2^ld(x)}(x)\subset B_{3/4}.
    \end{equation}
    By \eqref{ineq2.decay} together with \eqref{rel.phid1}, there is a constant $c=c(\mathsf{data})$ such that
    \begin{align*}
        |d(y)|\leq c2^kd(x)\quad\text{for any }y\in B_{2^kd(x)}(x).
    \end{align*}
    Thus we have 
    \begin{equation*}
        \|u\|_{L^\infty\left(B_{2^kd(x)}(x)\right)}\leq cd^s(x)2^{ks}\|u/d^s\|_{L^\infty\left(B_{2^kd(x)}(x)\right)}\leq cd^s(x)2^{ks}\|u/d^s\|_{L^\infty\left(\Omega\cap B_{1/8}(x)\right)}
    \end{equation*}
    for some constant $c=c(\mathsf{data})$.
   Therefore, using this, \cite[Lemma 2.2]{DieKimLeeNow24n} and \eqref{l.cor.hol}, we get
    \begin{align*}
    \|u\|_{L^\infty(B_{2d(x)/3}(x))}+\mathrm{Tail}(u;B_{2d(x)/3}(x))
    &\leq c\left[\sum_{i=1}^{l}2^{-2si}\|u\|_{L^\infty\left(B_{2^id(x)}(x)\right)}+2^{-2sl}\mathrm{Tail}(u;B_{1/8}(x))\right]\\
        &\leq cd(x)^s(\|u/d^s\|_{L^\infty(\Omega\cap B_{1/8}(x))}+\|u\|_{L^1_{2s}(\bbR^n)}).
    \end{align*}
    Plugging this into \eqref{ineq2.cor.hol} together with Theorem \ref{thm.dini} yields
    \begin{align}\label{ineq3.cor.hol}
        \frac{|u(x)-u(y)|}{|x-y|^s}\leq c(\|u\|_{L^1_{2s}(\bbR^n)}+\|f\|_{L^{n/s,1}(\Omega\cap B_1)})
    \end{align}
    for some constant $c=c(\mathsf{data})$.
   The desired estimate then follows from \eqref{ineq1.cor.hol} and \eqref{ineq3.cor.hol}.
\end{proof}

We next prove Theorem \ref{thm.hol.1} and Theorem \ref{thm.hol.2}.
\begin{proof}[Proof of Theorem \ref{thm.hol.1}]
Let us fix $\sigma\in(0,\alpha)$ and $q\in[2_*,n/(s-\sigma))$. Let us denote 
\begin{equation*}
    \widetilde{\sigma}\coloneqq\frac{\sigma+\alpha}{2}<\alpha.
\end{equation*}
We now claim that for any $z\in \overline{\Omega}\cap B_{1/2}$, there holds 
\begin{equation}\label{ineq1.thm.hol.1}
\begin{aligned}
    M^{\#,\Omega}_{\widetilde{\sigma},2^{-6}}(u/d^s)(z)&\leq c\left(\|u\|_{L^1_{2s}(\bbR^n)}+M^{\Omega}_{2^{-6}}(u/d^s)(z)\right)\\
&\quad+c\left(\mathbf{W}^{|f|^{2_*}\mbox{\Large$\chi$}_{\Omega}}_{\frac{2_* s}{2_*+1},2_*+1}(z,2^{-2})+\mathbf{W}^{|u|^{2_*}\mbox{\Large$\chi$}_{\Omega}}_{\frac{2_* s}{2_*+1},2_*+1}(z,2^{-2})\right)\\
&\quad+c\left(\left(M^{\Omega}_{2_*(s-\widetilde{\sigma}),2^{-2}}(|u|^{2_*})(z)\right)^{\frac1{2_*}}+\left(M^{\Omega}_{2_*(s-\widetilde{\sigma}),2^{-2}}(|f|^{2_*})(z)\right)^{\frac1{2_*}}\right)
\end{aligned}
\end{equation}
for some constant $c=c(\mathsf{data},q,\sigma,K)$. To this end, we fix $z\in \overline{\Omega}\cap B_{1/2}$. By \eqref{ass.thm.hol}, for any $\delta>0$, there is a constant $\rho_0=\rho_0(\alpha,K)$ such that $A$ is $(\delta,\rho_0)$-vanishing.
Thus by Theorem \ref{thm.vmo} along with \eqref{ineq.simple}, we have 
\begin{equation}\label{ineq2.thm.hol.1}
\begin{aligned}
    \|u\|_{L^{\frac{nq}{n-(s-\widetilde{\sigma})q}}(B_{7/8})}\leq c\max\{1,d^s(0)\}\|u/d^s\|_{L^{\frac{nq}{n-(s-\widetilde{\sigma})q}}(B_{7/8})}&\leq c\left[\|u\|_{L^1_{2s}(\bbR^n)}+\|f\|_{L^{\frac{nq}{n+\widetilde{\sigma}q}}(\Omega\cap B_{1})}\right]\\
    &\leq c\left[\|u\|_{L^1_{2s}(\bbR^n)}+\|f\|_{L^{q}(\Omega\cap B_{1})}\right]
\end{aligned}
\end{equation}
for some constant $c=c(\mathsf{data},K)$. We next employ Remark \ref{rmk.sec5} to observe that there is a weak solution $\widetilde{u}\in W^{s,2}(\bbR^n)$ to \eqref{eq: sec5.loc} and \eqref{defn.omega.sec5} with $\widetilde{u}\equiv 0$ in $\bbR^n\setminus B_{2}$. We note from \eqref{ass.thm.hol} that \eqref{ass.lem.hol} holds with $L=K$. Therefore, Lemma \ref{lem.hol} yields 
\begin{equation}\label{ineq3.thm.hol.1}
\begin{aligned}
    {N}^{\widetilde{\sigma}}_{2^{-6}}(\widetilde{u};z)&\leq c\left(\|\widetilde{u}\|_{L^1_{2s}(\bbR^n)}+\mathbf{W}^{|F|^{2_*}}_{\frac{2_* s}{2_*+1},2_*+1}(z,2^{-2})+\left({M}^{\Omega}_{2_*(s-\widetilde{\sigma}),2^{-2}}(|F|^{2_*})(z)\right)^{\frac1{2_*}}\right),
\end{aligned}
\end{equation}
where $c=c(\mathsf{data},K,\sigma)$. For any $x,y\in \overline{\Omega}\cap B_{2^{-6}}(z)$, we observe  
    \begin{align*}
    |(u/d^s)(x)-(u/d^s)(y)|&=|(\widetilde{u}/d^s)(x)-(\widetilde{u}/d^s)(y)|\\
    &\leq |(\widetilde{u}/\oldphi)(x)-(\widetilde{u}/\oldphi)(y)||(\oldphi/d^s)(x)|+|(\widetilde{u}/\oldphi)(y)||(\oldphi/d^s)(x)-(\oldphi/d^s)(y)|\\
    &\leq c|(\widetilde{u}/\oldphi)(x)-(\widetilde{u}/\oldphi)(y)|+c|(\widetilde{u}/\oldphi)(y)||x-y|^{\widetilde{\sigma}}\\
    &\leq c|(\widetilde{u}/\oldphi)(x)-(\widetilde{u}/\oldphi)(y)|+c|({u}/d^s)(y)||x-y|^{\widetilde{\sigma}}
\end{align*}
for some constant $c=c(\mathsf{data},\sigma)$, where we have used the third condition given in \eqref{eq: sec5.loc}, \eqref{rel.phid1} and \eqref{hol.phid1}. Combining this and \eqref{ineq3.thm.hol.1}, we get
\begin{equation}\label{ineq4.thm.hol.1}
\begin{aligned}
    M^{\#,\Omega}_{\widetilde{\sigma},2^{-6}}(u/d^s)(z)&\leq c\left({N}^{\widetilde{\sigma}}_{2^{-6}}(\widetilde{u};z)+M^{\Omega}_{2^{-6}}({u}/d^s)(z)\right)\\
    &\leq c\left(\|u\|_{L^1_{2s}(\bbR^n)}+\left(\dashint_{\Omega\cap B_{2^{-3}}(z)}|f|^{2_*}\,dx\right)^{\frac1{2_*}}+M^{\Omega}_{2^{-6}}({u}/d^s)(z)\right)\\
    &\quad+c\left(\mathbf{W}^{|F|^{2_*}}_{\frac{2_* s}{2_*+1},2_*+1}(z,2^{-2})+\left({M}^{\Omega}_{2_*(s-\widetilde{\sigma}),2^{-2}}(|F|^{2_*})(z)\right)^{\frac1{2_*}}\right)
\end{aligned}
\end{equation}
for some constant $c=c(\mathsf{data},K,\sigma)$.
As in the estimate of \eqref{ineq81.thm.dini} along with \eqref{sec5.F}, we have
\begin{align*}
&\mathbf{W}^{|F|^{2_*}\mbox{\Large$\chi$}_{\Omega}}_{\frac{2_* s}{2_*+1},2_*+1}(z,2^{-2})+\left({M}^{\Omega}_{2_*(s-\widetilde{\sigma}),2^{-2}}(|F|^{2_*})(z)\right)^{\frac1{2_*}}\\
&\leq c\left(\mathbf{W}^{|f|^{2_*}\mbox{\Large$\chi$}_{\Omega}}_{\frac{2_* s}{2_*+1},2_*+1}(z,2^{-2})+\mathbf{W}^{|u|^{2_*}\mbox{\Large$\chi$}_{\Omega}}_{\frac{2_* s}{2_*+1},2_*+1}(z,2^{-2})\right)\\
&\quad +c\left(\left({M}^{\Omega}_{2_*(s-\widetilde{\sigma}),2^{-2}}(|f|^{2_*})(z)\right)^{\frac1{2_*}}+\left({M}^{\Omega}_{2_*(s-\widetilde{\sigma}),2^{-2}}(|u|^{2_*})(z)\right)^{\frac1{2_*}}+\|u\|_{L^1_{2s}(\bbR^n)}\right)
\end{align*}
for some constant $c=c(\mathsf{data},\sigma)$. Combining this and \eqref{ineq4.thm.hol.1} proves the claim \eqref{ineq1.thm.hol.1}.
We now employ the standard strong p-p estimate and \eqref{ineq2.thm.hol.1} to see that
\begin{align*}
    \|M^{\Omega}_{2^{-6}}(u/d^s)(z)\|_{L^{\frac{nq}{n-(s-\widetilde{\sigma}q)}}(B_{1/2})}\leq c\|u/d^s\|_{L^{\frac{nq}{n-(s-\widetilde{\sigma}q)}}(B_{3/4})}\leq c(\|u\|_{L^1_{2s}(\bbR^n)}+\|f\|_{L^q(\Omega\cap B_1)}).
\end{align*}
Plugging this into \eqref{ineq1.thm.hol.1} and applying \eqref{ineq1.wolff} along with \eqref{ineq1.pre.fmax} to the terms on the second and third lines in \eqref{ineq1.thm.hol.1}, we have
\begin{align*}
    \|M^{\#,\Omega}_{\widetilde{\sigma},2^{-6}}(u/d^s)(z)\|_{L^{\frac{nq}{n-(s-\widetilde{\sigma})q}}(B_{1/2})}\leq c(\|u\|_{L^1_{2s}(\bbR^n)}+\|f\|_{L^q(\Omega\cap B_1)})
\end{align*}
for some constant $c=c(\mathsf{data},K,\sigma,q)$. We now apply Lemma \ref{lem.max} with $p=\frac{nq}{n-(s-\widetilde{\sigma})q}$, $\beta=\widetilde{\sigma}$ and $t=\sigma$ to see that
\begin{align*}
    \|u/d^s\|_{W^{\sigma,\frac{nq}{n-(s-\sigma)q}}(B_{2^{-7}})}&\leq c(\|u\|_{L^{\frac{nq}{n-(s-\widetilde{\sigma})q}}(B_{2^{-6}})}+\|u\|_{L^1_{2s}(\bbR^n)}+\|f\|_{L^q(\Omega\cap B_1)})\\
    &\leq c(\|u\|_{L^1_{2s}(\bbR^n)}+\|f\|_{L^q(\Omega\cap B_1)})
\end{align*}
for some constant $c=c(\mathsf{data},K,\sigma,q)$, where we have used \eqref{ineq2.thm.hol.1} for the last inequality. By the standard covering argument, we obtain the desired estimate.
\end{proof}

\begin{proof}[Proof of Theorem \ref{thm.hol.2}]
Let us fix $\sigma\in(0,\alpha)$. As in the proof of \eqref{ineq1.thm.hol.1}, we have
\begin{equation}\label{ineq1.thm.hol.2}
\begin{aligned}
    M^{\#,\Omega}_{{\sigma},2^{-6}}(u/d^s)(z)&\leq c\left(\|u\|_{L^1_{2s}(\bbR^n)}+\left(\dashint_{\Omega\cap B_1}|f|^{2_*}\,dx\right)^{\frac1{2_*}}+cM^{\Omega}_{2^{-6}}(u/d^s)(z)\right)\\
&\quad+c\left(\mathbf{W}^{|f|^{2_*}\mbox{\Large$\chi$}_{\Omega}}_{\frac{2_* s}{2_*+1},2_*+1}(z,2^{-3})+\mathbf{W}^{|u|^{2_*}\mbox{\Large$\chi$}_{\Omega}}_{\frac{2_* s}{2_*+1},2_*+1}(z,2^{-3})\right)\\
&\quad+c\left(\left(M^{\Omega}_{2_*(s-{\sigma}),2^{-3}}(|u|^{2_*})(z)\right)^{\frac1{2_*}}+\left(M^{\Omega}_{2_*(s-{\sigma}),2^{-3}}(|f|^{2_*})(z)\right)^{\frac1{2_*}}\right)
\end{aligned}
\end{equation}
for some constant $c=c(\mathsf{data},K,\sigma)$.
Since $A$ is H\"older continuous and $f\in L^{n/(s-\sigma),\infty}(\Omega\cap B_1)\subset L^{n/s,1}(\Omega\cap B_1)$, we have 
\begin{align}\label{ineq2.thm.hol.2}
    \|u/d^s\|_{L^\infty(\Omega\cap B_{3/4})}\leq c(\|u\|_{L^1_{2s}(\bbR^n)}+\|f\|_{ L^{n/(s-\sigma),\infty}(\Omega\cap B_1)})
\end{align}
for some constant $c=c(\mathsf{data},K,\sigma)$, which implies 
\begin{equation*}
    \|M^{\Omega}_{2^{-6}}(u/d^s)(z)\|_{L^\infty(B_{1/2})}\leq\|u/d^s\|_{L^\infty(\Omega\cap B_{3/4})}\leq c(\|u\|_{L^1_{2s}(\bbR^n)}+\|f\|_{ L^{n/(s-\sigma),\infty}(\Omega\cap B_1)}).
\end{equation*}
Using this and applying \eqref{ineq2.wolff} and \eqref{ineq2.pre.fmax} into second and third lines in \eqref{ineq1.thm.hol.2}, we obtain 
\begin{align*}
     M^{\#,\Omega}_{{\sigma},2^{-6}}(u/d^s)(z)&\leq c\left(\|u\|_{L^1_{2s}(\bbR^n)}+\|u\|_{L^{n/(s-\sigma)}(\Omega\cap B_{3/4})}+\|f\|_{L^{n/(s-\sigma)}(\Omega\cap B_{3/4})}\right)\\
     &\leq c\left(\|u\|_{L^1_{2s}(\bbR^n)}+\|f\|_{L^{n/(s-\sigma)}(\Omega\cap B_{1})}\right)
\end{align*}
for some constant $c=c(\mathsf{data},\sigma,K)$, whenever $z\in B_{1/2}$. Therefore, using this and \cite[Theorem 2.9]{Giu03} together with the fact that $\Omega$ is a $C^{1,\alpha}$-domain, we have
\begin{align*}
    \|u/d^s\|_{C^\sigma(\overline{\Omega}\cap B_{1/2})}\leq c\left(\|u\|_{L^1_{2s}(\bbR^n)}+\|f\|_{L^{n/(s-\sigma)}(\Omega\cap B_{1})}\right),
\end{align*}
which completes the proof.
\end{proof}

We end this section by providing some details on the validity of our results to more general kernel coefficients which are comparable to the recent work \cite{KimWei24} in the case of the H\"older continuity assumption.
\subsection{General kernel coefficients}\label{sec5.3}
In this subsection, we apply our arguments to obtain the same estimates as in Theorem \ref{thm.vmo}, Theorem \ref{thm.dini}, Corollary \ref{cor.cs}, Theorem \ref{thm.hol.1} and Theorem \ref{thm.hol.2} with more general kernel coefficients. 

 Let us fix $\rho_0>0$. We then assume that the associated kernel coefficient $A$ satisfies
\begin{equation}\label{defn.delta2}
    \sup_{0<R\leq \rho_0}\sup_{z\in B_1}\dashint_{B_R(z)}\dashint_{B_R(z)}|A-\overline{A}_{B_R(z)}(x,y)|\,dx\,dy\leq \delta,
\end{equation}
where 
\begin{equation*}
    \overline{A}_{B_R(z)}(x,y)\coloneqq \frac{1}{2}\dashint_{B_R(z)}(A(x-y+h,h)+A(y-x+h,h))\,dh \quad\text{for }x,y\in\bbR^n.
\end{equation*}
Since the function $\overline{A}_{B_R(z)}(x,y)$ is translation invariant, we first note from \cite[Theorem 1.4 \& Theorem 1.6]{RosWei24} that any weak solution $v$ to 
\begin{equation*}
\left\{
\begin{alignedat}{3}
\mathcal{L}_{\overline{A}_{B_{2r}(z)}}{v}&= f&&\qquad \mbox{in  $\Omega \cap B_{r}(z)$}, \\
{v}&=0&&\qquad  \mbox{in $B_{r}(z)\setminus \Omega$},
\end{alignedat} \right.
\end{equation*}
satisfies
\begin{equation*}
    r^s\|v/d^s\|_{L^\infty(B_{r/2}(z))}\leq c\left(\|v\|_{L^1_{2s}(\bbR^n)}+\|f\|_{L^\infty(\Omega\cap B_{r}(z))}\right)
\end{equation*}
for some constant $c=c(\mathsf{data})$.
Therefore, we obtain the estimate given in Lemma \ref{lem.loc.const} when 
\begin{align*}
    a_0(x,y)\coloneqq \begin{cases}
        \overline{A}_{B_{2r}(z)}(x,y)&\text{if }(x,y)\in B_{2r}(z)\times B_{2r}(z),\\
        A(x,y)&\text{if }(x,y)\in \bbR^{2n}\setminus(B_{2r}(z)\times B_{2r}(z)).
    \end{cases} 
\end{align*}
By following the same lines as in the proof of Lemmas \ref{lem.disc.comp}-\ref{lem.vmo} with $(A)_{B_{2r}(z)\times B_{2r}(z)}$ replaced by $\overline{A}_{B_{2r}(z)}$, we prove the same result in Theorem \ref{thm.vmo} when the vanishing condition given in Definition \ref{defn.del.van} is replaced by \eqref{defn.delta2}.

We now consider the case of continuous kernel coefficients.
Let us assume that the kernel coefficient $A$ satisfies
\begin{equation}\label{ker.tran}
    \sup_{0<|h|<1}\sup_{x,y\in B_1}|A(x+h,y+h)-A(x,y)|\leq \omega(|h|),
\end{equation}
where $\omega:\bbR_+\to\bbR_+$ is a non-decreasing function with $\omega(0)=0$. We first note that if $A$ satisfies \eqref{ker.tran}, then there is a sufficiently small $\rho_0=\rho_0(\omega)$ such that $A$ satisfies the vanishing condition \eqref{defn.delta2}.
By \cite[Theorem 1.1, Theorem 1.4 \& Theorem 1.6]{RosWei24}, we next observe that  for any $z\in \Omega\cap B_{1/2}$, there is the weak solution $\oldphi\equiv \oldphi_z$ to \eqref{eq.barrier} with 
\begin{equation}\label{defn.az}
    A_z(x,y)\coloneqq(A(x-y+z,z)+A(y-x+z,z))/2\quad \text{for } x,y\in\bbR^n,
\end{equation}
satisfying \eqref{rel.phid1} and \eqref{bdd.phid1}. Since we can apply Lemma \ref{lem.localization} with $a=A_z$ there, to obtain Remark \ref{rmk.sec5}, we now consider a weak solution $u\in W^{s,2}(\bbR^n)$ to \eqref{eq.sec5.loc} with 
\begin{equation*}
    |A(x,y)-A_z(x,y)|\leq ({\pmb{\omega}}(|y-z|)+{\pmb{\omega}}(|x-z|))/2,
\end{equation*}
where $\pmb{\omega}$ satisfies the same condition as in \eqref{defn.omega.sec5}.
By following the same lines as in the proof of \cite[Lemma 6.6]{KimWei24} together with \cite[Theorem 1.1,Theorem 1.4 \& Theorem 1.6]{RosWei24}, the weak solution $v$ to \eqref{eq.comp.v} with $a_0=A_z$ satisfies \eqref{hig.fir.ineq} with $\alpha$ replaced by $\alpha s/2$. In addition, we have 
\begin{align*}
       r^{4s}\left(\sup_{x\in B_r(z)}\int_{\bbR^n\setminus B_{2r}(z)}\frac{|A-A_z|}{|y-z|^{n+2s}}\,dy\right)^2
       &\leq r^{4s}\left(\sup_{x\in B_r(z)}\int_{\bbR^n\setminus B_{2r}(z)}\frac{\pmb{\omega}(|y-z|)+\pmb{\omega}(|x-z|)}{2|y-z|^{n+2s}}\,dy\right)^2\\
       &\leq r^{4s}\left(\sup_{x\in B_r(z)}\int_{B_{1/8}(z)\setminus B_{2r}(z)}\frac{{\omega}(|y-z|)}{|y-z|^{n+2s}}\,dy\right)^2\\
       &\leq r^{4s}\left(\int_{2r}^{1/8}\frac{{\omega}(\rho)}{\rho^{1+2s}}\,dy\right)^2\leq {\pmb{\omega}}_G(2r)
   \end{align*}
   and
\begin{align*}
    \sup_{x\in B_r(z)}\int_{\bbR^n\setminus B_{2r}(z)}\frac{|w(y)||A-A_z|}{|y-z|^{n+2s}}\,dy&\leq \int_{B_{1/8}(z)\setminus B_{2r}(z)}\frac{|w(y)|(\pmb{\omega}(|y-z|)+\pmb{\omega}(|x-z|))}{2|y-z|^{n+2s}}\,dy\\
    &\leq \int_{B_{1/8}(z)\setminus B_{2r}(z)}\frac{|w(y)|{\omega}(|y-z|)}{|y-z|^{n+2s}}\,dy\\
    &\leq \int_{\bbR^n\setminus B_{2r}(z)}\frac{|w(y)|{\pmb{\omega}}(|y-z|)}{|y-z|^{n+2s}}\,dy.
\end{align*}
Thus, we are able to prove \eqref{comp2.ineq} and Lemma \ref{lem.decay} with $\alpha$ replaced by $\alpha s/2$. Using this along with \eqref{rel.phid1} and \eqref{bdd.phid1}, we can prove Lemma \ref{lem.dini} and Lemma \ref{lem.dini.max} when
\begin{equation}\label{ass1} 
    \int_{0}^1\frac{\omega(\rho)}{\rho}\,d\rho<\infty.
\end{equation}
Finally, all results given in Theorem \ref{thm.dini} and Corollary \ref{cor.cs} hold if $A$ satisfies \eqref{ker.tran} and \eqref{ass1}. 

However, to get a higher differentiability of $u/d^s$, we need to impose more regularity assumption on the frozen kernel $A_z$.
To do this, we now assume that $A$ satisfies \eqref{ker.tran} with \begin{equation}\label{ass2} 
   {\omega(\rho)}\leq K\rho^\alpha,
\end{equation} and the frozen kernel coefficient $A_z$ given in \eqref{defn.az} is homogeneous for any $z\in B_1$ (see \cite[Definition 2.1.21]{FerRos24} for a precise definition of the homogeneity of the kernel coefficient). We first observe from \cite{FerRos24} that \eqref{rel.phid1}, \eqref{bdd.phid1} and \eqref{hol.phid1} hold with $A(z,z)$ replaced by $A_z$.
By \cite[Lemma 8.1]{KimWei24}, we obtain the same estimates as in Lemma \ref{lem.decay}. Since under the assumption \eqref{ass2}, \eqref{ass1} immediately follows. Therefore, we are able to prove Lemma \ref{lem.hol}. Using this together with \eqref{rel.phid1}, \eqref{bdd.phid1} and \eqref{hol.phid1}, we finally prove Theorem \ref{thm.hol.1} and Theorem \ref{thm.hol.2}.

\appendix
\section{Some sharpness results}\label{appen}
In this appendix, we construct some examples to establish the sharpness of the estimates of Theorem \ref{thm.vmo}, Theorem \ref{thm.hol.1} and Theorem \ref{thm.hol.2}.

Let us take a smooth domain $\Omega$ such that 
\begin{equation*}
    \Omega\cap B_1=B_1^+.
\end{equation*}
We take 
\begin{equation*}
    f(x)\coloneqq |x|^{-n/q}
\end{equation*}
for any $q>1$. We have 
\begin{equation*}
    u(x)=\int_{\Omega}G(x,y)f(y)\,dy
\end{equation*}
is a weak solution to \eqref{eq: thm} with $A=1$, where $G$ is the green function with respect to the domain $\Omega$. By \cite{CheKimSon10}, we have for any $x\in B_{1}^+$,
\begin{align*}
    u(x)&\geq \int_{B_1^+}\min\left\{\frac{x_n}{|x-y|},1\right\}^{s}\min\left\{\frac{y_n}{|x-y|},1\right\}^{s}\frac{|y|^{-n/q}}{|x-y|^{n-2s}}\,dy\\
    &\geq \int_{B_{x_n}(x)}\min\left\{\frac{x_n}{|x-y|},1\right\}^{s}\min\left\{\frac{y_n}{|x-y|},1\right\}^{s}\frac{|y|^{-n/q}}{|x-y|^{n-2s}}\,dy.
\end{align*}
Since for any $y\in B^{+}_{x_n}(x)\coloneqq\{(z',z_n)\in B_{x_n}(x)\,:\,z_n\geq x_n\}$
\begin{equation*}
    \min\left\{\frac{x_n}{|x-y|},1\right\}\leq 1 \quad\text{and}\quad \min\left\{\frac{y_n}{|x-y|},1\right\}\leq 2,
\end{equation*}
we have
\begin{align*}
    u(x)\geq \int_{B^{+}_{x_n}(x)}\frac{|y|^{-n/q}}{|x-y|^{n-2s}}\,dy\geq \int_{B^{+}_{x_n}(x)\cap |y|\geq\frac{|x|}2}\frac{|y|^{-n/q}}{|x-y|^{n-2s}}\,dy.
\end{align*}
We observe from
\begin{equation*}
    B^{+}_{x_n}(x)\cap |y|\geq\frac{|x|}2=\{x+x_n\theta\,:\, \theta\in S_{n}^{+}\quad\text{and}\quad |x+x_n\theta|\geq|x|/2 \}
\end{equation*}
that there is a set $A\subset S_{n}^{+}$ such that $B^{+}_{x_n}(x)\cap |y|=\{x+x_n\theta\,:\, \theta\in A \}$ and
\begin{equation*}
    \int_{A}\,dS\geq c|S_n^{+}|
\end{equation*}
for some constant $c$, where $dS$ is the surface measure on $S_n^{+}$. Therefore, we have
\begin{align*}
    u(x)\geq \int_{B^{+}_{x_n}(x)}\frac{|y|^{-n/q}}{|x-y|^{n-2s}}\,dy&\geq \int_{B^{+}_{x_n}(x)\cap |y|\geq\frac{|x|}2}\frac{|x|^{-n/q}}{|x-y|^{n-2s}}\,dy\\
    &= c|x|^{-n/q}\int_{0}^{x_n}\int_{A}\frac{1}{r^{1-2s}}\,dS\,dr\geq c|x|^{-n/q}x_n^{2s}
\end{align*}
for some constant $ c=c(n,s)$. Therefore, we have 
\begin{equation}\label{ineq1.appen}
    |u/d^s(x)|\geq c|x|^{-n/q}x_n^s\quad\text{in }B_1^+
\end{equation}
and
\begin{align*}
    \int_{B_1^+}|u/d^s(x)|^{p}\,dx\geq c\int_{B_1^+}\left||x|^{-n/q}x_n^s\right|^{p}\,dx&\geq c\int_{0}^{1}\int_{S^+_{n}}\left|r^{-n/q}r^s \theta_n^{s}\right|^{p}r^{n-1}\,dS\,dr\\
    &\geq c\int_{0}^{1} r^{(s-n/q)p-(n-1)}\,dr.
\end{align*}
This implies $u/d^s\notin L^{p}$ if $p\geq \frac{nq}{n-sq}$ but $u/d^s\in L^{p}$ if $p<\frac{nq}{n-sq}$. We now fix $q\in[2_*,n/s)$. Then we have that for any $\epsilon\in(0,n/s-q)$ there is a weak solution $u$ to \eqref{eq: defn} with $f(x)=|x|^{-\frac{n\left(1+\frac{\epsilon s(n-sq)}{2n^2}\right)}{q+\frac{\epsilon(n-sq)}{2n}}}\in L^q(\Omega\cap B_1)$, but
\begin{equation*}
    u/d^s\notin L^{\frac{nq}{n-sq}+\epsilon}(\Omega\cap B_1).
\end{equation*}

\printbibliography

\end{document}